\documentclass[12pt]{amsart}
\usepackage{amsmath,amsthm,amssymb,amscd} 
\usepackage{color}

\usepackage{enumerate}
\usepackage{amsmath,amsthm,amssymb,mathrsfs}

\newcommand{\jump}[1]{\ensuremath{[\![#1]\!]} }

\def\a{{\alpha}}
\def\b{{\beta}}
\def\g{\gamma}
\def\s{\sigma}
\def\l{\left}
\def\r{\right}
\def\v{\varpi}
\def\w{\omega}
\def\Om{\Omega}
\def\ep{\varepsilon}
\def\vp{\varphi}
\def\th{\theta}
\def\Th{\Theta}

\def\R{{\mathbb{R}}}
\def\C{{\mathbb{C}}}
\def\Z{{\mathbb{Z}}}

\def\Ss{{\mathscr{S}}}
\def\W{{\mathscr{W}}}
\def\Cs{{\mathscr{C}}}
\def\Bs{{\mathscr{B}}}
\def\Xs{{\mathscr{X}}}

\def\Tc{{\mathcal{T}}}
\def\Sc{{\mathcal{S}}}
\def\Sb{{\mathbb{S}}}
\def\Gb{{\mathbb{G}}}

\def\af{{\mathbf{a}}}

\def\nf{{\mathbf{n}}}
\def\bnf{{\bar{\mathbf{n}}}}

\def\bz{{\bar{\mathbf{z}}}}
\def\tf{{\mathbf{t}}}

\def\c{\mathfrak{c}}
\def\e{{\mathfrak{e}}}
\def\zb{{\mathbf{z}}}

\def\o{\mathscr{O}}
\def\p{{\mathscr{P}}}

\def\L{{\mathcal{L}}}

\def\G{{\mathbf{G}}}

\def\ab{{\mathbf{a}}}

\def\jb{{\mathbf{j}}}

\def\tb{\mathbf{t}}
\def\nb{\mathbf{n}}

\def\bnb{\bar{\mathbf{n}}}
\def\1{\mathbf{1}}%

\def\N{{\mathbf{N}}}%
\def\Tb{{\mathbf{T}}}
\def\Ab{{\mathbf{A}}}
\def\Bb{{\mathbf{B}}}
\def\Zb{{\mathbf{Z}}}
\def\Qb{{\mathbf{Q}}}
\def\Ub{{\mathbf{U}}}
\def\Kb{{\mathbf{K}}}
\def\Wb{{\mathbf{W}}}

\def\t{\times}
\def\ot{\otimes}
\def\bt{\boxtimes}
\def\op{\oplus}
\def\bs{\backslash}
\def\wt{\widetilde}
\def\ra{\rangle}
\def\la{\langle}
\def\Hom{{\mathrm{Hom}}}
\def\Alg{{\mathrm{Alg}}}
\def\ro{{\mathfrak{o}}}

\def\Ch{{\mathrm{Ch}}}
\def\supp{{\mathrm{supp}}}
 
\def\vol{{\mathrm{vol}}}

\def\Int{{\mathrm{Int}}}
\def\Ir{{\mathrm{Irr}}}
\def\Alg{{\mathrm{Alg}}}

\numberwithin{equation}{section}

\newtheorem{thm}{Theorem}[section]
\theoremstyle{plain}

\newtheorem{prop}[thm]{Proposition}
\newtheorem{Cor}[thm]{Corollary}
\newtheorem{lem}[thm]{Lemma}

\newtheorem{mthm}{Main Theorem}

\theoremstyle{definition}
\newtheorem{Rem}{Remark}

\newcommand{\FRAC}[2]{\leavevmode\kern.1em \raise.5ex\hbox{\the\scriptfont0 #1}\kern-.1em/\kern-.15em\lower.25ex\hbox{\the\scriptfont0 #2}}

\newdimen\argwidth
\def\[[#1\]]{
\setbox0=\hbox{$#1$}\argwidth=\wd0
\setbox0=\hbox{$\left[\box0\right]$}\advance\argwidth by -\wd0
\left[\kern.3\argwidth\box0\kern.3\argwidth\right]}

\title{LOCAL WHITTAKER-NEWFORMS FOR $GSp(4)$ MATCHING TO LANGLANDS PARAMETERS}
\author{TAKEO OKAZAKI}
\dedicatory{Dedicated to Professor Tomoyoshi Ibukiyama on his 70th birthday}
\date{}
\thanks{
The author is supported by JSPS Grant-in-Aid for Scientific Research No. 24740017.}
\address{Department of Mathematics, Faculty of Science, Nara Woman University,
Kitauoyahigashi-machi, Nara 630-8506, Japan. }
\email{okazaki@cc.nara-wu.ac.jp}
\begin{document}
\maketitle

\begin{abstract}
We extend the local newform theory of B. Roberts and R. Schmidt for generic, irreducible, admissible representations of $PGSp(4)$ to that for $GSp(4)$.
The newform matches to the Langlands parameter.
\end{abstract}
\section{Introduction}
Let $F$ be a non-archimedean local field of characteristic $0$ and residue characteristic $p$. 
Let $WD_F$ be the Weil-Deligne group. 
Let $\phi:WD_F \to GSp(4,\C)$ be a $L$-parameter.
The local Langlands correspondence for $GSp(4)$ showed by W. T. Gan and S. Takeda \cite{G-T2} says that, if $\phi$ is tempered, the $L$-packet of $\phi$ contains a unique generic, irreducible, admissible representation $\pi$ whose $L$- and $\ep$-factors defined by F. Shahidi \cite{Shahidi} coincide with those of $\phi$ respectively.
In the context of noncommutative class field theory, and Shimura type conjectures, for example, Yoshida-Brumer-Kramer conjecture \cite{Y}, \cite{B} on Abelian surfaces (see also \cite{O-Y} for Siegel threefold varieties), it is natural to quest which vector in $\pi$ possesses the $L$- and $\ep$-factors of $\phi$, and by which subgroup the vector is fixed.
For the generic $GL(d)$-case, the answer can be found in the series of the works of H. Jacquet, I. I. Piatetski-Shapiro, J. A. Shalika, and the subsequent works of S. Kondo, S. Yasuda \cite{K-Y}, N. Matrigne \cite{Ma}, and M. Miyauchi \cite{Mi}.
For the generic $PGSp(4)$-case, the answer was provided by B. Roberts and R. Schmidt \cite{R-S} for nondiscrete $L$-parameters (they provided also for some non-generic cases).
The `paramodular group' corresponding to the $L$-parameter is the fixing subgroup. 
After these works, in this paper, we will provide the following answer for the generic $GSp(4)$-case.
Let $\o$ be the ring of integers of $F$ and $\p = \v \o$ be its maximal ideal with a fixed generator $\v$.
Let $q = |\o/\p| = |\v|^{-1}$.
Let 
\begin{eqnarray}
\begin{bmatrix}
 & & & -1 \\
 & & -1 & \\
& 1& & \\
1& & & 
\end{bmatrix} \label{eqn:defmat}
\end{eqnarray}
be the defining matrix for $GSp(4)$.
We fix a continuous homomorphism $\psi: F \to \C^1$ such that $\psi(\o) = 1$ but $\psi(\p^{-1}) \neq 1$.
Let $\pi$ be a generic, irreducible, admissible representation of $GSp(4,F)$, and $\W_\psi(\pi)$ denote the representation space of consisting of (Whittaker) functions $W$ such that
\begin{eqnarray*}
W(\begin{bmatrix}
1 & -x & *&* \\
 &1 & y& *\\
 & &1 &x \\
 & & &1
\end{bmatrix}g) = \psi(y+x)W(g).
\end{eqnarray*}
Let $\w_\pi$ be the central character of $\pi$, and $\e$ its (order of) conductor.
For an integer $m \ge 2\e$, we define $\Kb(m;\e)$ to be the subgroup of all $k \in GSp(4,F)$ such that $\det(k) \in \o^\t$ and  
\begin{eqnarray*}
k \in \begin{bmatrix}
\o & \o & \o & \p^{-l} \\
\p^l & \o & \o &\o \\
\p^l & \o & \o & \o \\
\p^{m} & \p^l & \p^l & \o 
\end{bmatrix},
\end{eqnarray*}
where $l = m -\e$.
Define $\Kb_1(m;\e) = \{k \in \Kb(m;\e) \mid k_{44} \in 1+ \p^{\e}\}$.
We call these open compact subgroups the quasi-paramodular groups of level $m$.
They are contained in the paramodular group $\Kb(m-\e)$ of level $m - \e$.
In case of $\e = 0$, they coincide with $\Kb(m)$.
We call $\Kb_1(m;\e)$-invariant Whittaker functions quasi-paramodular forms of level $m$, including the case of $\e = 0$.
Let $V(m) \subset \W_\psi(\pi)$ denote the subspace consisting of quasi-paramodular forms of level $m$.
Observe that if $W \in V(m)$, then $\pi(k)W = \w_\pi(k_{44})W$ for $k \in \Kb(m;\e)$.
Although $\Kb_1(m+1;\e) \not\subset \Kb_1(m;\e)$, there exists an inclusion map $V(m) \hookrightarrow V(m+1)$.
The minimal integer $m$ such that $V(m) \neq \{0\}$ is called the minimal level of $\pi$, and denoted by $m_\pi$.
\begin{mthm}
Let $\pi$ be a generic, irreducible, admissible representation of $GSp(4,F)$ with $L$-parameter $\phi_\pi$.
Write $\ep(s,\phi_\pi,\psi) = \ep_\pi q^{-n_\pi(s-\frac{1}{2})}$.
Then, $m_\pi = n_\pi$, and $V(m_\pi)$ is one-dimensional.
There exists a unique $W$ in $V(m_\pi)$ such that 
\begin{eqnarray}
&& \int_{F^\t} W(\begin{bmatrix}
t & & & \\
 &t & & \\
& &1 & \\
& & &1 
\end{bmatrix}
)|t|^{s-\frac{3}{2}} d^\t t = L(s,\phi_\pi), \label{eqn:mthm1} \\
&& \int_{F^\t} W(\begin{bmatrix}
 & & t & \\
 & & & -t \\
\v^{n_\pi}& & & \\
& -\v^{n_\pi}& & 
\end{bmatrix}
)\w_\pi(t)^{-1} |t|^{s-\frac{3}{2}} d^\t t  = q^{\e} \ep_\pi L(s,\phi_\pi^{\vee}), \label{eqn:mthm2}
\end{eqnarray}
where $d^\t t$ is the Haar measure such that $\vol(\o^\t) = 1$.
\end{mthm}
The zeta integral (\ref{eqn:mthm1}) coincides with Novodvorsky's $Z(s,W)$ (\cite{N}), if $W \in \W_\psi(\pi)$ is quasi-paramodular (Proposition \ref{prop:zeta-simp}).
As well as in the $PGSp(4, F)$-case, for a tempered representation of $GSp(4,F)$, the genericity is equivalent to the quasi-paramodularity (Theorem \ref{thm:temp}). 
We now describe our method.
\begin{enumerate}[i)]
\item 
We show that if there exists a $W \in V(m)$ satisfying the equalities (\ref{eqn:mthm1}), and (\ref{eqn:mthm2}) up to a constant multiple, then $m = m_\pi$, and $V(m)$ is spanned by this $W$ (Theorem \ref{thm: coinLN}). 
Comparing with our functional equation (Theorem \ref{thm:fesp}), we find that the existence of such a $W$ means that $m_\pi$ equals $n_\pi'$, the analytic conductor, and that (\ref{eqn:mthm2}) with replacing $\ep_\pi$ by $\ep_\pi'$, the analytic root number, holds exactly.
Here, the functional equation is a modified version of Novodovorsky's \cite{N}, and $\ep_\pi', n_\pi'$ are defined by the $\ep$-factor (\ref{eqn:defepGSp4}). 
See also the remark in p.82 of \cite{R-S}.
Following the idea of B. Roberts and R. Schmidt, we use the $P_3$-representation theory (sect. \ref{sec:P3}), to prove the functional equation, and Theorem \ref{thm:vnthm} that says a quasi-paramodular form vanishing at all diagonal matrices is identically zero.
Theorem \ref{thm: coinLN} comes from Theorem \ref{thm:vnthm}.
\item 
To show the existence of $W$ as in Theorem \ref{thm: coinLN}, in sect. \ref{sec:Hecke}, we analyze Hecke actions on $V(m_\pi)$ when $\pi$ is supercuspidal, or when $\pi$ is a constituent of the induction of a supercuspidal representation of the Levi factor of the Klingen parabolic subgroup.
Since the $L$-function defined by \cite{N} of $\pi$ equals $1$ in this case, the Kirillov models corresponding to the quasi-paramodular forms have compact supports (Lemma \ref{lem:havev}).
This causes the analysis simple, and makes possible to determine all values at diagonal matrices of $W \in V(m_\pi)$ (Theorem \ref{thm:Ls1}).  
\item
For other generic constituents of parabolic inductions, we use the local $\th$-lift from $GL(2) \t GL(2)$ to $GSp(4)$.
It is known by \cite{G-T} that such constituents are obtained by the $\th$-lift.
In sect. \ref{sec:CQF}, the desired $W$ is constructed explicitly by the $\th$-lift.
\item
W. T. Gan and S. Takeda \cite{G-T2} showed the Langlands corespondence for $GSp(4)$ by observing the local $\th$-lift from $GSp(4)$ to $GL(4)$, and reducing to that for $GL(4)$ due to M. Harris and R. Taylor \cite{H-T}, and G. Henniart \cite{He}.
Following this line, in sect. \ref{sec:CN4}, by the $\th$-lift we construct the newform for $GL(4)$.
It matches to $\phi_\pi$, thanks to the newform theory for $GL(d)$ (sect. \ref{sec:GLd}).
Seeing that it is constructed by the above $W \in V(m_\pi)$, we obtain the coincidences $\ep_\pi' =\ep_\pi$, and $ n_\pi' = n_\pi$. 
\end{enumerate}
In the case of $\e > 0$, an elementary argument shows that, if $\pi(k)W = \chi(k) W$ for a quasi-character $\chi$ on a paramodular group, then $Z(s,W) = 0$, different from the case of $GL(d)$.  
In the case of $\e> 0$, the quasi-paramodular group is not normalized by the Weyl element $\j_m$ (c.f (\ref{eqn:weylelms})), and therefore $V(m)$ is not decomposed by the Atkin-Lehner operator defined by $\j_m$, different from the case of $\e = 0$.
We also consider the $\j_m$-conjugate of quasi-paramodular forms, which are called coquasi-paramodular forms.

{\bf Notation}
Let $F$ be a non-archimedean local field of characteristic $0$, and residue characteristic $p$. 
Let $\o$ be the ring of integers of $F$ and $\p = \v \o$ be its maximal ideal with a fixed generator $\v$.
Let $\p^* = \p \setminus \p^2$.
Let $q = |\o/\p| = |\v|^{-1}$.
Let $\ro(x)$ denote the $p$-adic order of $x \in F$, and let $\nu_s(x) = q^{-\ro(x)s}$ for $s \in \C$.
Let $\psi$ denote a continuous homomorphism $\psi: F \to \C^1$. 
We sometimes assume that the conductor of $\psi$ is $\o$, i.e., $\psi(\o) = 1$ but $\psi(\p^{-1}) \neq 1$.
If $G$ is a locally compact totally disconnected group (called an $l$-group), then we let $\Alg(G)$(resp. $\Ir(G)$) denote the category of smooth(resp. irreducible admissible) complex $G$-modules.
Let $\Xs(G)$ denote the subcategory of $\Ir(G)$ consisting of one-dimensional ones.
For $\chi \in \Xs(F^\t)$, let $\c(\chi)$ denote the order of the conductor of $\chi$.
If $\pi \in \Alg(G)$, then $\pi^\vee$ denotes the contragredient to $\pi$.
Let $L$ and $R$ denote the left and right translations of elements in $G$ on itself, respectively: $L(g) g' = g^{-1}g', R(g) g' = g'g$.
\section{Newforms for $GL(d)$}\label{sec:GLd}
We review the newform theory for a generic representation of $GL(d,F)$.
We will use the following notation for elements and subgroups of $G_d = GL(d, F)$:
\begin{eqnarray*}
N &=& \{n= (n_{ij}) \mid n_{ij} = 0 \ \mbox{for $i > j$}, \ n_{ii} = 1\},\\
\bar{N} &=& \{\bar{n} = \mbox{the transposition of $n \in N$}\},\\
K(m) &=& \{k \in G_d(\o) \mid k_{d1}, \cdots, k_{d,d-1} \in \p^m\}, \\
K_1(m) &=& \{k \in K(m) \mid k_{dd} \equiv 1 \pmod{\p^m} \}, \\
A &=& \{a(t) = \begin{bmatrix}
t & \\
 & 1_{d-1} 
\end{bmatrix} \mid t \in F^\t \}, \\ 
a_i &=& a(\v^i), \\
w_d &=& \mbox{the standard longest Weyl element in $G_d$}, \\ 
w_{1,d-1} &=& \begin{bmatrix}
1 & \\
 & w_{d-1}
\end{bmatrix}.
\end{eqnarray*}
In case of $r < d$, for an element $h \in G_r$, let 
\begin{eqnarray*}
h' = \begin{bmatrix}
h & \\
 & 1_{d-r}
\end{bmatrix} \in G_d.
\end{eqnarray*} 
For $h \in G_{d}$ and $b \in M_{d \t d}(F)$, let 
\begin{eqnarray*}
j(h) = \begin{bmatrix}
 & -h^{-1} \\
h &
\end{bmatrix}, \ 
n(b) = \begin{bmatrix}
1_{d} & b \\
 &1_{d}
\end{bmatrix}, \ \bar{n}(b) = \begin{bmatrix}
1_{d} &  \\
b &1_{d}
\end{bmatrix} \in SL(2d,F).
\end{eqnarray*}
The following identities are basic.
\begin{eqnarray}
\bar{n}(h) &=& n(h^{-1})j(h)n(h^{-1}), \label{eqn:usid} \\
\Int(\bar{n}(c))n(b) &=& \begin{bmatrix}
1_{d} -bc & b \\
cbc & 1_{d} + cb 
\end{bmatrix}. \label{eqn: nion}
\end{eqnarray}
Let $\W_\psi = {\rm Ind}_{N}^{G_d} \tilde{\psi}$ denote the induced representation consisting of smooth  functions $W: G_d \to \C$ (called Whittaker functions with respect to $\psi$) such that $L(n)W = \tilde{\psi}(n)^{-1} W$ for $n \in N$, where $\tilde{\psi} \in \Xs(N)$ is defined by $\tilde{\psi}(n) = \prod_{1 \le i \le d-1} \psi(n_{i,i+1})$.
We denote by $\Ir^{gn}(G_d)$ the subcategory consisting of $\pi$ such that $\Hom_{G_d}(\pi, \W_\psi) \neq \{0\}$.
If $(\pi,V) \in \Ir^{gn}(G_d)$, then $\Hom_{G_d}(\pi, \W_\psi) = \C \lambda$ for a functional $\lambda$, unique up to constant multiples, and we identify $V$ with $\W_\psi(\pi) : =\mathrm{Im}(\lambda)$. 
Let $W \in \W_\psi$.
For a nonnegative integer $r \le d-2$, let  
\begin{eqnarray*}
Z_r(s,W) = \iint W(\begin{bmatrix}
t & \\
x &1_r 
\end{bmatrix}')\nu_{s -\frac{n-1}{2}}(t) dx d^\t t \label{eqn:zetaGL}
\end{eqnarray*}
with the integration being over $t$ in $F^\t$ and $x$ in the column space $F^r$, where the Haar measures $dx$ and $d^\t t$ are chosen so that $\vol(\o^\t) = 1$ and $\vol(\o^r) = 1$ respectively.
Let $\pi \in \Ir^{gn}(G_d)$.
Let $L(s,\pi)$ and $\ep(s, \pi,\psi)$ denote the $L$- and $\ep$-factors respectively defined in \cite{G-J}, which coincide with those of the Rankin-Selberg convolution $\pi \t \1$ defined in \cite{J-PS-S2}(c.f. sect. 4 of \cite{J-PS-S3}), where $\1$ indicates the trivial quasi-character of $G_1 = F^\t$.
By the works of M. Harris and R. Taylor \cite{H-T}, and G. Henniart \cite{He}, these factors also coincide with those of the $L$-parameter $\phi_\pi: WD_F \to GL(d,\C)$, respectively.
Define $W^\imath \in \W_{\psi^{-1}}$ by $W^\imath(g) = W(w_d{}^tg^{-1}w_{1,d-1})$.
The $G_d$-module $\pi^\imath = \{W^\imath \mid W \in \W_\psi(\pi)\}$ is equivalent to $\pi^\vee$ (c.f. \cite{G-Ka}).
The functional equation for $\pi \t \1$ given in \cite{J-PS-S2} is 
\begin{eqnarray}
\frac{Z_0(1-s, W^\imath)}{L(1-s,\pi^\imath)} = \ep(s,\pi,\psi) \frac{Z_{d-2}(s, W)}{L(s,\pi)}. \label{eqn:FEGL}
\end{eqnarray}
It holds that $\ep(s,\pi,\psi)\ep(1-s,\pi^\imath,\psi^{-1}) = 1$.
Now fix a $\psi$ with conductor $\o$.
We define the root number $\ep_\pi$ and conductor $n_\pi$ by 
\begin{eqnarray*}
\ep(s,\pi,\psi) = \ep_\pi q^{-n_\pi(s-\frac{1}{2})}.
\end{eqnarray*}
Let $V(m)$ denote the subspace consisting of $K_1(m)$-invariant vectors in $\W_\psi(\pi)$.
Let $\w_\pi$ denote the central character of $\pi$, and $\e = \c(\w_\pi)$ its (order of) conductor.
Since $K(m)/K_1(m) \simeq \o^\t/(1+ \p^m)$, 
\begin{eqnarray}
V(m)  = \{W \in \W_\psi(\pi) \mid \pi(k)W = \w_\pi(k_{dd})W, \ \ k \in K(m) \}.  \label{eqn: Vmgl}
\end{eqnarray}
It is obvious that $\{0\} = \cdots \subset V(\e) \subset \cdots \subset V(\infty) := \cup_m V(m)$.
Observe that $V(\e-1) = \{ 0\}$ in case of $\e > 0$.
The smallest integer $m$ such that $V(m) \neq \{0\}$ is called {\bf the minimal level} of $\pi$, and denoted by $m_\pi$.
Then, $V(m_\pi)$ is one-dimensional, and spanned by a $W$ such that $W(1_d) = 1$, which is called the {\bf newform} of $\pi$ and denoted by $W_\pi$(called the essential vector in their original paper \cite{J-PS-S}).
The following identity was showed in \cite{Ma}, \cite{Mi}:
\begin{eqnarray}
Z_0(s,W_\pi)= L(s,\pi). \label{eqn: ZLcoinGL}
\end{eqnarray}
However, by the method of Lemma 4.1.1. of \cite{R-S}, $Z_r(s,W)$ are same for all nonnegative $r \le d-2$, if $W \in V(\infty)$ as is showed below.
\begin{prop}\label{prop:simp-zetagl}
With notations as above, if $W \in V(\infty)$, then $Z_0(s,W) = Z_r(s,W)$ for any $0 \le r \le d-2$.
\end{prop}
\begin{proof}
It suffices to show that, for $x = {}^t(x_1, \ldots, x_{r}) \not\in \o^r$, 
\begin{eqnarray}
W(\begin{bmatrix}
t & \\
x &1_r 
\end{bmatrix}') = 0. \label{eqn:Zsimpgl}
\end{eqnarray}
Let $x_l$ be the last element such that $x_l \not\in \o$.
Let $\check{x_l} = {}^t(x_1, \ldots, x_{l-1})$.
By the $K_1(m)$-invariance property, 
\begin{eqnarray*}
W(\begin{bmatrix}
t & \\
x &1_r
\end{bmatrix}') 
= W(\begin{bmatrix}
t & & \\
\check{x}_l&1_{l-1} &\\
x_l & & 1 
\end{bmatrix}').
\end{eqnarray*}
By (\ref{eqn:usid}) and the $K_1(m)$-invariance property, this equals 
\begin{eqnarray*}
&& W(\l(\begin{bmatrix}
t & & \\
\check{x_l} &1_{l-1} & \\
 & & 1 
\end{bmatrix}
\begin{bmatrix}
1 & & x_l^{-1} \\
 &1_{l-1} & \\
 & & 1 
\end{bmatrix}
\begin{bmatrix}
 & & -x_l^{-1}\\
 &1_{l-1} & \\
x_l  & &  
\end{bmatrix}
\begin{bmatrix}
1 & & x_l^{-1} \\
 &1_{l-1} & \\
 & & 1 
\end{bmatrix}\r)') \\
&=& W(\l(\begin{bmatrix}
1 & & x_l^{-1}t\\
 &1_{l-1} & x_l^{-1}\check{x_l}\\
 & & 1 
\end{bmatrix}
\begin{bmatrix}
t & & \\
\check{x_l} &1_{l-1} &\\
 & & 1 
\end{bmatrix}
\begin{bmatrix}
 & & -x_l^{-1}\\
 &1_{l-1} & \\
x_l  & &  
\end{bmatrix}\r)') \\
&=& 
\psi(x_{l-1}x_l^{-1})W(\l(\begin{bmatrix}
t & & \\
\check{x_l} &1_{l-1} &\\
&&1
\end{bmatrix}
\begin{bmatrix}
 & & -x_l^{-1}\\
 &1_{l-1} & \\
x_l  & &  
\end{bmatrix}\r)'). 
\end{eqnarray*}
Since $x_l \not\in \o$ and $\psi(\p^{-1}) \neq 1$,  there exists a $y \in \o$ such that $\psi(x_ly) \neq 1$.
Now (\ref{eqn:Zsimpgl}) follows from Lemma \ref{lem:vanishlem} below combined with 
\begin{eqnarray*}
&& \Int^{-1}(\l(\begin{bmatrix}
t & & \\
\check{x_l} &1_{l-1} & \\
 & & 1
\end{bmatrix}
\begin{bmatrix}
 & & -x_l^{-1}\\
 &1_{l-1} & \\
x_l  & &  
\end{bmatrix}\r)')\begin{bmatrix}
1_{l} & & \\
 &1 & x_ly\\
& &1  
\end{bmatrix}'\\
&=& \Int^{-1}(
\begin{bmatrix}
 & & -x_l^{-1}\\
 &1_{l-1} & \\
x_l  & &  
\end{bmatrix}')\begin{bmatrix}
1_{l} & & \\
 &1 & x_ly\\
& &1  
\end{bmatrix}' = \begin{bmatrix}
1 & & -y\\
 &1_l & \\
& &1  
\end{bmatrix}' \in K_1(m).
\end{eqnarray*}
\end{proof}
\begin{lem}\label{lem:vanishlem}
Let $G$ be a group and $H, K$ be subgroups of $G$.
Let $\xi: H \to \C^\t$ and $\chi: K \to \C^\t$ be homomorphisms.
Let $f: G \to \C$ such that $L(h^{-1}) R(k)f = \xi(h)\chi(k) f$ for $h \in H, k \in K$.
Let $g \in G$.
If there exists an $h \in H$ such that $\Int^{-1}(g)h \in K$ and $\xi(h) \neq \chi(\Int^{-1}(g) h)$, then $f(g) = 0$.
\end{lem} 
By \cite{K-Y} it was showed that $n_\pi = m_\pi$.
Taking into account above results, we obtain the following characterization for the newforms.
\begin{thm}\label{thm:main}
Let $\pi \in \Ir^{gn}(G_d)$.
An integer $m$ equals $m_\pi$, if and only if there exists a $W \in V(m)$ such that $Z_0(s,W) = L(s,\pi)$ and $W(w_d a_{m}w_{1,d-1}) \neq 0$.
\end{thm}
\begin{proof}
We show only the if-part.
Let $W' = \pi^\imath(a_{-m})W^\imath$.
By (\ref{eqn:FEGL}) and Proposition \ref{prop:simp-zetagl}, 
\begin{eqnarray*}
\frac{Z_0(1-s,W')}{L(1-s,\pi^\imath)} = \ep_\pi q^{(m- m_\pi)(s -1/2)}.
\end{eqnarray*}
Since $W'$ is invariant under the subgroup 
\begin{eqnarray*}
\{n(x)' \in G_2' \mid x \in \o \},
\end{eqnarray*}
and $W'(1) (= W(w_d a_{m}w_{1,d-1}))\neq 0$ is assumed, by Lemma \ref{lem:vanishlem}, $Z_0(1-s,W')$ is a power series in $q^s$ with a nonzero constant term, and so is the left hand side of the above equation (recall $L(s,\pi^\imath)^{-1} = L(s,\pi^\vee)^{-1}$ is a polynomial in $q^{-s}$ with constant term $1$.).
However, the right side is a monomial in $q^{s}$.
Hence, both sides are constant, and $m = m_\pi$.
\end{proof}
From now on, we concentrate on the argument for the case that $d =2$ and the central character is ramified, which is an archetype for $GSp(4)$, and will be used repeatedly.
\begin{prop}\label{prop:mpicr}
Let $\pi \in \Ir^{gn}(G_2)$.
If $L(s,\pi) = 1$, then $m_\pi > \c(\w_\pi)$.
\end{prop}
Let $\e = \c(\w_\pi)$.
In case of $\e = 0$, the assertion is obvious, since an unramified representation is a principal series representation.
Assume $\e > 0$.
Let $m \ge \e$.
Consider the Hecke action $\Tc: V(m) \to V(m)$ defined by 
\begin{eqnarray}
\Tc W := 
\sum_{x \in \o/\p} \pi(\bar{n}(x\v^{m})a_{-1})W = \sum_{x \in \o/\p} \pi(a_{-1}\bar{n}(x\v^{m-1}))W. \label{eqn:expfmTGL2}
\end{eqnarray}
Observe that $\{\bar{n}(x\v^{m})\mid x \in \o/\p\}$ is representatives for $K_1(m) / K_1(m) \cap \Int(a_{-1})K_1(m)$.
In case of $m = m_\pi > \e$, we have $\Tc W_\pi = 0$ since $\Tc W_\pi$ is a constant multiple of $
\pi(a_{-1}) \int_{K_1(m_\pi-1)} \pi(k) W_\pi d k$.
But, this argument does not work in case of $m_\pi = \e$.
To observe $K_1(\e)$-invariant vectors in $\W_\psi$, we need the following Gauss sum and its partial sum.
Let $\chi \in \Xs(\o^\t)$ with $\c(\chi) = \e >0$.
Let $1 \le m \le \e$.
For $u \in \o^\t$, define 
\begin{eqnarray*}
\Gb(\chi, u) &=& \int_{\o}\psi\l(\frac{x}{\v^{\e}u}\r) \chi(x) d x, \\
\Sb_m(\chi, u) &=& \int_{\p^m} \psi\l(\frac{x}{\v^{\e}u}\r)\chi(1+x) dx, 
\end{eqnarray*}
where $dx$ is chosen so that $\vol(\o) = 1$.
Since $\p^{\lceil \e/2 \rceil}/\p^\e \simeq (1+\p^{\lceil \e/2 \rceil})/(1 + \p^\e)$, there is a continuous homomorphism $\psi_\chi: \p^{\lceil \e/2 \rceil} \to \C^1$ such that $\psi_\chi(x) = \chi(1+x)$.
If $u \in \o^\t$, then $\ker(\psi(*/u\v^\e) \psi_\chi(*)) = \p^{m(u)}$ for some integer $m(u) \le \e$.
Let $n_\chi = \min\{m(u) \mid u \in \o^\t \}$. 
Of course, $\lceil \e/2 \rceil \le n_\chi \le \e$.
By definition, $m(u) = m(u') = n_\chi$, if and only if $u \equiv u' \pmod{\p^{\e- n_\chi}}$.
So, we can define $u_\chi \in \o^\t$ uniquely modulo $\p^{\e- n_\chi}$ such that $m(u_\chi) = n_\chi$.
Let $1 \le m < \e$ and $z \in \p^{*m}$.
In case of $n_\chi > \max\{m,\e-m \}$, there exists an $x \in \p^{\max\{m,\e-m\}}$ such that $\psi(\frac{x}{\v^\e})\psi_\chi(\frac{zx}{\v^m}) \neq 1$ by definition.
In case of $n_\chi \le \max\{m,\e-m \}$, $\ker(\psi(\frac{*}{\v^\e})\psi_\chi(\frac{z*}{\v^m})) = \p^{\max\{m,\e-m \}}$, if and only if $z \equiv \v^m u_\chi \pmod{\p^{\min\{\e,2m \}}}$.
\begin{lem}\label{lem:Gauss2}
If $n_\chi \le \max\{m,\e-m \}$, then, $\Sb_m(\chi,u_\chi) \neq 0$.
\end{lem}
\begin{proof}
By definition of $u_\chi$, 
\begin{eqnarray}
\Sb_{n_\chi}(\chi,u) = 
\begin{cases}
q^{-n_\chi} & \mbox{if $u \equiv u_\chi \pmod{\p^{\e- n_\chi}}$,} \\
0 & \mbox{otherwise.}
\end{cases}
\label{eqn: Sce/2}
\end{eqnarray}
If $m \ge \e/2$, then $\max\{m,\e-m \} = m \ge n_\chi$, and the assertion is obvious.
Assume that $m < \e/2$.
Then, $m \le \e - n_\chi$.
Let $1 \le l < n \le \e- n_\chi + 1$, and $Y = \{y \}$ be representatives for $\p^l/\p^n$.
Consider the following decompositions:
\begin{eqnarray*}
\Sb_l(\chi,u) &=& \sum_{y \in Y}\int_{\p^{n}} \psi\l(\frac{y + x + y x}{\v^\e u}\r)\chi\l((1+y)(1+ x)\r) dx \\
&=& \sum_{y \in Y} \psi\l(\frac{y}{\v^{\e} u}\r) \chi(1+ y) \int_{\p^{n}} \psi\l(\frac{(1+ y)x}{\v^\e u}\r)\chi(1+ x)  dx \\
&=& \sum_{y \in Y} \psi\l(\frac{y}{\v^{\e}u}\r) \chi(1+y) \Sb_{n}(\chi,\frac{u}{1+ y}),
\end{eqnarray*}
and 
\begin{eqnarray*}
(0 \neq ) \ \Gb(\chi,1) &=& \sum_{z \in (\o/\p)^\t} \int_{\p} \psi\l(\frac{z+x}{\v^\e}\r) \chi(z + x) dx \\
&=&\sum_{z \in (\o/\p)^\t} \psi\l(\frac{z}{\v^\e}\r)\chi(z)  \int_\p \psi\l(\frac{x}{\v^\e}\r)\chi(1+ z^{-1} x) dx \\
&=& \sum_{z \in (\o/\p)^\t} \psi\l(\frac{z}{\v^\e}\r)\chi(z) \Sb_1(\chi, z^{-1}).
\end{eqnarray*}
If $u \not\equiv u_\chi \pmod{\p^m}$, then $\Sb_m(\chi,u) =0$ by (\ref{eqn: Sce/2}) and the former decomposition in case of $l = m, n =n_\chi$.
Hence by the latter decomposition, $\Sb_1(\chi,z^{-1})$ is not zero for a unique $z \in (\o/\p)^\t$ such that $z^{-1} \equiv u_\chi \pmod{\p}$.
Hence $\Sb_m(\chi,u) \neq 0$ if $u \equiv u_\chi \pmod{\p^m}$ by the former decomposition in case of $l = 1, n =m$.
This completes the proof.
\end{proof}
For $t \in F^\t$ and $z \in F$, let
\begin{eqnarray}
[t;z] = \begin{bmatrix}
t & \\
z & 1
\end{bmatrix} \in G_2. \label{eqn:[;]}
\end{eqnarray}
\begin{lem}\label{lem:e-lem}
With the preceeding assumption, let $W \in \W_\psi$ such that $R(k)W = \chi(k_{22})W$ for $k \in K(\e)$.
\begin{enumerate}[i)]
\item 
Assume $\e > 1$.
Let $0 < m < \e$, and $z \in \p^{*m}$.
We have $W([\v^i;z]) = 0$ unless $n_\chi \le \max\{m,\e-m\}$, $i = m-\e$ and $z \equiv \v^m u_\chi \pmod{\p^{\min\{\e,2m\}}}$.
In case of $m = i + \e$, it holds that, for $y \in \p^m$, 
\begin{eqnarray}
W([\v^i;(1+ y)^{-1}\v^{m}u_\chi]) =
\psi(\frac{y}{\v^\e u_\chi})\chi(1+ y) W([\v^i;\v^m u_\chi])\label{eqn:e-lemv}
\end{eqnarray}
and 
\begin{eqnarray*}
\int_{\p^{*m}} W([\v^i;z])dz = 
\Sb_m(\chi,u_\chi)W([\v^i;\v^m u_\chi]) 
\end{eqnarray*}
where $dz$ is chosen so that $\vol(\o) = 1$.
\item Assume $\e = 1$.
For $i \ge 0$, 
\begin{eqnarray*}
\int_{\o^\t} W([\v^i;z])dz = 0.
\end{eqnarray*}
\end{enumerate}
\end{lem}
\begin{proof}
i) 
(\ref{eqn:e-lemv}) follows from the identity
\begin{eqnarray*}
n(x)[\v^i;\v^m u_\chi]=  [\v^i;\v^m u_\chi - \frac{\v^{\e+m}u_\chi^2x}{1+ \v^{\e}u_\chi x}] \begin{bmatrix}
1+ \v^{\e}u_\chi x &  \\
 & (1+ \v^{\e}u_\chi x)^{-1}
\end{bmatrix} n\l(\frac{\v^{-i}x}{1+ \v^{\e}u_\chi x}\r).
\end{eqnarray*}
The last assertion is obvious.
For the remained assertion, we will use repeatedly Lemma \ref{lem:vanishlem}, and the identity
\begin{eqnarray}
\Int^{-1}([\v^i;z])n(x) = 
\begin{bmatrix}
1+ \v^{-i}zx & \v^{-i}x \\
-\v^{-i}xz^2 & 1- \v^{-i}zx
\end{bmatrix}, \label{eqn: elemad}
\end{eqnarray}
which lies in $K(\e)$ if $x \in \p^{\max\{i, i -2m + \e\}}$.
Suppose that $i < m- \e$.
Then $m-i-1 \ge \e$, and there is an $x \in \p^{-1} (\subset \p^{\max\{i, i -2m + \e\}})$ such that $\psi(x) \neq 1 = \chi(1- \v^{-i} zx)$.
Hence $W([\v^i;z]) = 0$.
Suppose that $i > m- \e$.
Then there is an $x \in \p^{\e-m+i-1}( \subset \p^{\max\{i, i -2m + \e\}})$ such that $\psi(x) = 1 \neq \chi(1- \v^{-i} z x)$. 
Hence $W([\v^i;z]) = 0$.
Suppose that $i = m- \e$.
If $z \not\equiv \v^{m} u_\chi \pmod{\p^{\min\{\e,2m \}}}$ or $n_\chi > \max\{m,\e-m \}$, then there is an $x \in \p^{\max\{m- \e, -m\}}$ such that $\psi(x) \neq \chi(1- \v^{-i}z x)$, and hence $W([\v^i;z]) = 0$.

ii) follows from the computation:
\begin{eqnarray*}
\int_{\o^\t} W(a_i \bar{n}(z)) dz &=& \int_{\o^\t} W(a_i n(z^{-1})j(z)n(z^{-1})) dz \\
&=& \int_{\o^\t} \psi(z^{-1} \v^{i}) W(a_ij(z)) dz \\
&=& W(a_ij(1)) \int_{\o^\t} \chi(z)^{-1} dz = 0.
\end{eqnarray*}
\end{proof}
Now, we can prove Proposition \ref{prop:mpicr}. 
By (\ref{eqn: ZLcoinGL}), $W_\pi(1) = 1$ and $W_\pi(a_i) = 0$ for $i \neq 0$.
Since $\dim V(m_\pi) = 1$, there is a constant $\lambda$ such that $\Tc W_\pi = \lambda W_\pi$.
From (\ref{eqn:expfmTGL2}), and the above Lemma, it follows that 
\begin{eqnarray*}
\lambda W_\pi(a_1) = W_\pi(1).
\end{eqnarray*}
This is a contradiction.
This completes the proof of the proposition.
\section{Representations of $P_3$}\label{sec:P3}
Let $P_3$ be the subgroup of $G_3$ of matrices of the form of 
\begin{eqnarray*}
\begin{bmatrix}
g & \b\\
 & 1
\end{bmatrix}, \ \ g \in G_2.
\end{eqnarray*}
We need the following notations for subgroups and elements in $P_3$. 
\begin{eqnarray*}
N_2 &=& \{n_2(x) = \begin{bmatrix}
1 & & \\
& 1& x \\
& & 1
\end{bmatrix} \}, \ 
N_3 = \{n_3(x) = \begin{bmatrix}
1 & & x \\
 & 1& \\
& & 1
\end{bmatrix} \},  \\
N' &=& \{n'(x) = \begin{bmatrix}
1 & x \\
 & 1
\end{bmatrix}' \}, \\
Z_2' &=& \{z_2'(t) =(t1_2)' \mid t \in F^\t  \}, \\
M &=&N'N_3, \ 
M^\flat = \bar{N}' N_2 (\simeq M).
\end{eqnarray*}
For $\xi \in \Xs(F^\t)$, let $\xi_\psi \in \Xs(NA)$ defined by 
\begin{eqnarray*}
\xi_\psi(a(t)n) = \xi(t)\psi(n_{2,3}), \ n \in N.
\end{eqnarray*} 
For $\rho \in \Ir(G_2)$, let $\rho'$ denote the representation of $P_3$ sending elements $g' n \in G_2' N_2N_3$ to $\rho(g)$, whose representation space is same as $\rho$.
Every irreducible smooth representation of $P_3$ is isomorphic to 
\begin{eqnarray*}
\tau_0 := \mathrm{ind}_{N}^{P_3} \tilde{\psi}, \ \tau_1(\xi) := \mathrm{ind}_{NA}^{P_3} \xi_\psi \ \mbox{or} \ \tau_2(\rho) := \rho',
\end{eqnarray*}
where $\mathrm{ind}$ indicates the compact induction.
For $\chi \in \Xs(F^\t)$, and $b \in F^\t$, let $\eta_b(\chi) \in \Xs(MZ_2'), \s_b(\chi) \in \Xs(M^\flat A)$, and $\s_0(\chi) \in \Xs(\bar{N}' Z_2')$ defined by  
\begin{eqnarray*}
\eta_b(\chi)(z_2'(t)m) &=& \chi(t)\psi(b m_{1,2}),  \ \ m \in M \\
\s_b(\chi)(a(t)m) &=& \chi(t)\psi(b m_{2,3}), \ \ m \in M^\flat \\
\s_0(\chi)(z_2'(t)n) &=& \chi(t), \ \ n \in \bar{N}'.
\end{eqnarray*}
For an $l$-group $G$, we say a distribution $D$ on $G$ left (resp. right) quasi-invariant with $\chi \in \Xs(G)$, if $\chi(g) D$ equals $D \circ L(g)$ (resp. $D \circ R(g)$) for all $g \in G$.
By the proof of Proposition 1.18 of \cite{B-Z} (taking the family of neighborhoods of $1$ in $\ker(\chi)$), the space of quasi-invariant distributions is one-dimensional. 
Indeed, there is a constant $c$ such that $D(\vp) = c \int_G \vp(g) \chi(g)^{-1}dg$ for $\vp \in \Ss(G)$, where $dg$ is a left (resp. right) Haar measure on $G$.
Following propositions are verified by Bruhat's distributional technique for induced representations (c.f. section 5 of \cite{Wa}).
\begin{prop}\label{prop: onedimP}
With the above notation, 
\begin{enumerate}[i)]
\item The space $\Hom_{MZ_2'}(\tau_0, \eta_b(\chi))$ is spanned by the nontrivial functional $\mu_\chi^b:\tau_0 \to \C$ defined by 
\begin{eqnarray*}
\mu_\chi^b(f) = \int_{F^\t} \chi^{-1}(t)f(z_2'(t)a(b)) d^\t t.  \label{eqn:linfun}
\end{eqnarray*}
\item For any $\xi \in \Xs(F^\t)$, $\Hom_{MZ_2'}(\tau_1(\xi), \eta_b(\chi)) = \{ 0 \}$.
\item Let $\rho \in \Ir(G_2)$.
Then,
\begin{eqnarray*}
\Hom_{MZ_2'}(\tau_2(\rho), \eta_b(\chi)) = \begin{cases}
\C \mu_2^b (\neq \{0\}) & \mbox{if $\chi = \w_\rho$, and $\rho \in \Ir^{gn}(G_2)$,} \\
\{0 \}& \mbox{otherwise},
\end{cases}
\end{eqnarray*}
where $\mu_2^b: \tau_2(\rho) \to \C$ is defined by $\mu_2^b(f) = f(a(b))$.
\end{enumerate}
\end{prop}
\begin{proof}
It suffices to show for the case of the conductor of $\psi$ is $\o$ and $b =1$.
i) 
Let $\vp \in \Ss(P_3)$.
Define $f_\vp \in \tau_0$ by
\begin{eqnarray}
f_\vp(p) = \int_{N} \tilde{\psi}(n)^{-1} \vp(n p) d n. \label{eqn:consttau0}
\end{eqnarray}
We claim that the linear mapping $\Ss(P_3) \ni \vp \mapsto f_\vp \in \tau_0$ is surjective.
Let $f \in \tau_0$.
We will use the following compact subgroups: 
\begin{eqnarray*}
\Gamma(m) &=& \{k \in G_2(\o) \mid k \equiv 1_2 \pmod{\p^m}\} \subset G_2, \\
\Upsilon(m) &=& \{p \in P_3(\o) \mid p \equiv 1_3 \pmod{\p^m}\} \subset P_3.
\end{eqnarray*}
Take $m$ so that $f$ is right $\Upsilon(m)$-invariant.
By the Iwasawa decomposition of $G_2$, we have $P_3 = \bigsqcup_{l \in \Z^2} N \v^{l}G_2(\o)'$.
Hence, by a finite subset $\mathfrak{T}$ of representatives for $G_2(\o)/\Gamma(m)$, we have 
\begin{eqnarray}
P_3= \bigsqcup_{l \in \Z^2, t \in \mathfrak{T}} N \v^{l} t' \Upsilon(m) \label{eqn:IwP3}.
\end{eqnarray}
Let $\vp^{l}_{t} \in \Ss(P_3)$ be the characteristic function of the compact orbit $N(\o)\v^{l} t' \Upsilon(m)$.
The function $f_{\vp_t^l}$ vanishes outside of $N\v^{l} t' \Upsilon(m)$, and takes a constant value $c_t^l$ on $N(\o)\v^{l} t' \Upsilon(m)$.
Any $\Upsilon(m)$-invariant $f' \in \tau_0$ with $\supp(f')=N \v^{l} t' \Upsilon(m)$ is a constant multiple of $f_{\vp_t^l}$.
In particular, $f'(\w^lt') = 0$ if $c_t^l = 0$.
Therefore, for the $\Upsilon(m)$-invariant $f \in \tau_0$, setting 
\begin{eqnarray*}
\vp = \sum (c_t^l)^{-1} f(\v^{l} t') \vp^{l}_{t}
\end{eqnarray*}
with the sum (finite since $f \in \tau_0$) being over $l \in \Z^2$ and $t \in \mathfrak{T}$ such that $c_t^l \neq 0$, we have $f(\v^{l} t') = f_\vp(\v^{l} t')$ for all $l \in \Z^2,t \in \mathfrak{T}$. 
By the disjoint union (\ref{eqn:IwP3}), $f = f_\vp$.
This proves the claim.
Let $\mu \in \Hom_{MZ_2'}(\tau_0, \eta_b(\chi))$ correspond to the distribution $D_\mu$ on $P_3$ defined by 
\begin{eqnarray*}
D_\mu(\vp) = \mu(f_\vp).
\end{eqnarray*}
Since $\Ss(P_3) \ni \vp \mapsto f_\vp \in \tau_0$ is surjective, the linear mapping $\mu \mapsto D_\mu$ to the space of distributions on $P_3$ is injective.
By definition, if $D= D_\mu$, then 
\begin{eqnarray}
D \circ R(h) &=& \eta_1(\chi)(h) D  \ \ \ \ ( h \in MZ_2'),  \label{eqn:trlawdis}\\
D \circ L(n) &=& \tilde{\psi}(n)^{-1} D \ \ \ \ (n \in N). \label{eqn:trlawdis2}
\end{eqnarray}
Now we observe the support of $D_\mu$ in the sense of 1.10 of \cite{B-Z}.
Take representatives for the double coset space $N \bs P_3/ MZ_2'$, for example, $\{a(s) \mid s \in F^\t \} \sqcup \{a(s) w_2' \mid s \in F^\t\}$.
Let $\vp_{s,m} \in \Ss(P_3)$ be the characteristic function of $a(s)\Upsilon(m)$.
For $k \in \Upsilon(m)$, 
\begin{eqnarray*}
L(n'(x))\vp_{s,m}(a(s)k) &=& \vp_{s,m}(n'(-x)a(s)k) \\
&=& \vp_{s,m}(a(s)\Int(n'(-s x))k n'(-sx)) \\
&=& R(n'(-sx))\vp_{s,m}(a(s)\Int(n'(-s x))k).
\end{eqnarray*}
If $s \neq 1$, then we may take a sufficiently large $m$ so that there exists an $x \in F$ such that $\psi((1-s)x) \neq 1$ and $\Int(n'(s x)) \Upsilon(m) \subset \Upsilon(m)$, and therefore, by (\ref{eqn:trlawdis}), (\ref{eqn:trlawdis2}), 
\begin{eqnarray*}
\psi(-x) D(\vp_{s,m}) &=& D(L(n'(x))(\vp_{s,m})) \\
&=& D(R(n'(-s x))(\vp_{s,m})) \\
&=& \psi(-sx) D(\vp_{s,m}).
\end{eqnarray*}
Hence $D(\vp_{s,m}) = 0$ and $a(s) \not\in \supp(D)$, unless $s =1$.
Similarly, one can see that $a(s)w_2' \not\in \supp(D)$ by using the identity $a(s)w_2' n_3(x) = n_2(sx)a(s)w_2'$. 
Therefore, $\supp(D) \subset N Z_2'$.
By the exact sequence in 1.9 of loc. cit., we may regard $D$ as a distribution on the closed subgroup $NZ_2'$ of $P_3$ such that $D \circ L(n) R(z_2'(t)) = \chi(t) \tilde{\psi}(n)^{-1} D$.
Since $N \cap Z_2' = \{1 \}$, and $Z_2' \simeq F^\t$, $\Ss(NZ_2') \simeq \Ss(N) \ot \Ss(F^\t)$.
Therefore, such a $D$ lies in the space $\Hom_{N \t F^\t}(\Ss(N) \ot \Ss(F^\t),\C_{\tilde{\psi}^{-1}} \ot \C_{\chi})$ where $\C_{\tilde{\psi}^{-1}}$ and $\C_{\chi}$ indicate the representation spaces of $\tilde{\psi}^{-1}$ and $\chi$ respectively.
Since the spaces of quasi-invariant distributions on $N, F^\t$ are one-dimensional, so is $\Hom_{N \t F^\t}(\Ss(N) \ot \Ss(F^\t),\C_{\tilde{\psi}^{-1}} \ot \C_{\chi})$ which is isomorphic to 
\begin{eqnarray*}
\Hom_{N}\l(\Ss(N),  \C_{\tilde{\psi}^{-1}} \ot \Hom_{F^\t}(\Ss(F^\t),\C_{\chi})\r) \simeq \Hom_{N}(\Ss(N), \C_{\tilde{\psi}^{-1}}).
\end{eqnarray*}
Hence, $\Hom_{MZ_2'}(\tau_0, \eta_1(\chi))$ is $1$-dimensional at most.
By (\ref{eqn:IwP3}) we can define the right $\Upsilon(\c(\chi))$-invariant $f_\chi \in \tau_0$ by 
\begin{eqnarray*}
f_\chi(n\v^{l}g') =
\begin{cases}
\tilde{\psi}(n) & \mbox{if $l = (0,0)$ and $ng' \in N\Upsilon(\c(\chi))$,} \\
0 & \mbox{otherwise.}
\end{cases}
\end{eqnarray*}
Obviously $\mu_\chi^1(f_\chi) \neq 0$, and $\mu_\chi^1$ spans $\Hom_{MZ_2'}(\tau_0, \eta_1(\chi))$. \\
ii) Similar to i).
Replace the condition (\ref{eqn:trlawdis2}) with 
\begin{eqnarray*}
D \circ L(h) =  \eta_1(\chi)(h)^{-1} D \ \ \ \ (h \in NA). 
\end{eqnarray*}
By using this condition, (\ref{eqn:trlawdis}), and representatives for $NA \bs P_3/ MZ_2'$, say $\{1_3,w_2' \}$, one can see that the supports of corresponding distributions are emptysets.\\
iii) follows from the fact that $\Hom_{G_2}(\rho, \W_\psi) \simeq \Hom_{N (\subset G_2)}(\rho, \tilde{\psi})$ is one-dimensional, if $\rho$ is generic.
\end{proof}
The following proposition is proved similarly (c.f. Lemma 2.5.4., 2.5.5., 2.5.6. of \cite{R-S}). 
\begin{prop}\label{prop: onedimP2} 
\begin{enumerate}[i)]
\item  The space $\Hom_{AM^\flat}(\tau_0, \s_b(\chi))$ is spanned by the nontrivial functional $\lambda_\chi^b:\tau_0 \to \C$ defined by  
\begin{eqnarray}
\lambda_\chi^b(f) = \int_{F^\t} \int_{\bar{N}'} \chi^{-1}(t)\nu_{-1}(t)f(a(t) n z_2'(b)) d n d^\t t.   \label{eqn:linfun2}
\end{eqnarray}
\item Let $\chi, \xi \in \Xs(F^\t)$.
Then, 
\begin{eqnarray*}
\Hom_{AM^\flat}(\tau_1(\xi), \s_b(\chi)) = \begin{cases}
\C \lambda_1^b \neq \{0\} & \mbox{if $\xi = \nu_1\chi$,} \\
\{ 0\} & \mbox{otherwise,}
\end{cases}
\end{eqnarray*}
where $\lambda_1^b: \tau_1(\xi) \to \C$ is defined by
\begin{eqnarray*}
\lambda_1^b(f) = \int_{\bar{N}'} f(n z_2'(b)) dn.
\end{eqnarray*}
\item  For any $\rho \in \Ir(G_2)$, $\Hom_{AM^\flat}(\tau_2(\rho), \s_b(\chi)) = \{ 0 \}$.
\end{enumerate}
\end{prop}
Since both of $\tau_2(\xi \circ \det)$ and $\s_0(\chi)$ are one-dimensional, the following is obvious.
\begin{prop}\label{prop: onedimP3}
Let $\chi, \xi \in \Xs(F^\t)$.
Then, 
\begin{eqnarray*}
\Hom_{Z_2'\bar{N}'}(\tau_2(\xi \circ \det), \s_0(\chi)) = \begin{cases}
\C \mu_2' \neq \{0\} & \mbox{if $\chi = \xi^2$,} \\
\{ 0\} & \mbox{otherwise,}
\end{cases}
\end{eqnarray*}
where $\mu_2': \tau_2(\xi \circ \det) \to \C$ is the nontrivial functional defined by $\mu_2'(f) = f(1_3)$.
\end{prop}
\section{Representations of Whittaker types}\label{sec:Q} 
Let $\G = GSp(4,F)$.
Subgroups of $\G$ will be written in capital boldface.
The center of $\G$ is isomorphic to $F^\t$, and we identify them.
Let $\Qb^\circ \subset \G$ be the subgroup consisting of matrices of the form of 
\begin{eqnarray*}
\begin{bmatrix}
* &* &* &* \\
 & a& b& x\\
 & c& d & y\\
 & & & 1
\end{bmatrix}. \label{eqn:elmQ}
\end{eqnarray*}
The Klingen parabolic subgroup $\Qb$ is generated by $\Qb^\circ$ and $F^\t$.
The Jacobi subgroup of $\Qb$ consists of the above matrices such that $ad -bc = 1$, and its center is 
\begin{eqnarray*}
\Zb^J = \{\zb(x) = \begin{bmatrix}
1 & & &x \\
 & 1& & \\
& & 1& \\
& & & 1
\end{bmatrix} \mid x \in F \}.
\end{eqnarray*}
Let $\mathrm{pr}: \Qb^\circ \to P_3$ be the projection sending the above matrices in $\Qb^\circ$ to 
\begin{eqnarray*}
\begin{bmatrix}
a & b& x \\
 c& d & y \\
 & & 1
\end{bmatrix}.
\end{eqnarray*}
Then, $\mathrm{pr}$ is a homomorphism with $\ker(\mathrm{pr}) = \Zb^J$, and thus $\Qb^\circ /\Zb^J \simeq P_3$.
We will argue about the representations of $P_3$ and $\Qb^\circ$.
In \cite{R-S}, they use the projection sending $q = zq_0$ with $z \in F^\t, q_0 \in \Qb^\circ$ to $\mathrm{pr}(q_0)$, and relate the representations of $P_3$ to those of $\Qb/F^\t$.
By using $\mathrm{pr}$, many of their arguments for the representations of $PGSp(4)$ also work for those of $\G$ having unramified central characters.
The following subgroups of $\Qb^\circ$ correspond to those of $P_3$ in the previous section. 
\begin{eqnarray*}
&& \G_2' = \{g' = \begin{bmatrix}
\det(g) & &  \\
 & g&  \\
& & 1& 
\end{bmatrix} \mid g \in G_2 \},  \N' = \{\nb'(x) = \begin{bmatrix}
1 & & & \\
 & 1&x & \\
& & 1& \\
& & & 1
\end{bmatrix} \mid x \in F \}, \\
&& \N_2 = \{\nb_2(x) = \begin{bmatrix}
1 &-x & & \\
 & 1& & \\
& & 1&x \\
& & & 1
\end{bmatrix}\mid x \in F \}, 
\N_3 = \{\nb_3(x) = \begin{bmatrix}
1 & &x & \\
 & 1& & x\\
& & 1& \\
& & & 1
\end{bmatrix} \mid x \in F \}, \\
&& \N = \N' \N_2 \N_3 \Zb^J.
\end{eqnarray*}
Define $\tilde{\psi} \in \Xs(\N)$ by $\tilde{\psi}(\nb_2(x)\nb'(y)\nb_3(*)\zb(*)) = \psi(x+y)$.
Define $\W_\psi = {\rm Ind}_{\N}^{\G} \tilde{\psi}$, $\Ir^{gn}(\G)$ and $\W_\psi(\pi)$ for $\pi \in \Ir(\G)$, similar to the $G_d$-case. 
Let $W \in \W_\psi$.
Via $\mathrm{pr}$, and the embedding 
\begin{eqnarray}
G_2 \ni h \mapsto h^\natural:= \begin{bmatrix}
w_2{}^th^{-1}w_2 & \\
 & h
\end{bmatrix} \in \G,  \label{eqn:defsharp}
\end{eqnarray}
we define the function on $P_3$ and that on $G_2$ by
\begin{eqnarray}
f_W(p) = W(\mathrm{pr}^{-1} (p)), \ \xi_W(h) = W(h^\natural). \label{eqn:gaugedef}
\end{eqnarray}
They are called {\bf the first and second gauge} of $W$, respectively.
Note that $f_W$ is well-defined since $W$ is left $\Zb^J$-invariant.
For the torus subgroups, we will use the following notations:
\begin{eqnarray*}
\Tb &=& \{\tb(x,y;z) = diag(xz,yz,y^{-1},x^{-1}) \mid x,y,z \in F^\t \},\\
\Tb_1 &=& \{\tb_1(x) = \tb(x,1;1) \mid x \in F^\t \}, \\
\Ab' &=& \{\ab'(z) = \tb(1,1;z) \mid z \in F^\t \}. 
\end{eqnarray*}
In particular, 
\begin{eqnarray*}
\ab_j^i &=& \tf(\v^i,1;\v^j), \ \  \eta = \tb_1(\v^{-1}),  \ \ \ab_j = \ab'(\v^j).
\end{eqnarray*}
The following Weyl elements are important to our arguments.
\begin{eqnarray}
&& \jb'(x) = \begin{bmatrix}
1 & & & \\
 & &-x^{-1} & \\
& x& & \\
& & &1 
\end{bmatrix}, \ \jb''_m = \begin{bmatrix}
 & & & -\v^{-m} \\
 &1 & & \\
& &1 & \\
\v^m & & & 
\end{bmatrix}, \nonumber \\
&& \j_{m} = \begin{bmatrix}
 & & 1& \\
 & & &-1 \\
\v^m& & & \\
& -\v^m & & 
\end{bmatrix}, \label{eqn:weylelms}
\end{eqnarray}
where $x \in F^\t, m \in \Z$.
For an admissible $(\pi, V) \in \Alg(\G)$, let $V(\Zb^J)$ denote the $\C$-subspace spanned by $v -\pi(z) v, v \in V, z \in \Zb^J$.
Let $V_{\Zb^J} = V/ V(\Zb^J)$. 
Via the isomorphism $\Qb^\circ/\Zb^J \simeq P_3$, we may regard $V_{\Zb^J}$ as a smooth $P_3$-module. 
We denote also by $\mathrm{pr}$ the projection $V \to V_{\Zb^J}$.
Following to \cite{J-PS-S2}, we refer to an admissible $\pi \in \Alg(\G)$ of finite length such that $\dim_\C \Hom_{\N}(\pi, \wt{\psi}) = 1$, as a representation {\bf of Whittaker type}.
By the proof of Lemma 2.5.2, Theorem 2.5.3 of \cite{R-S}, 
\begin{thm}[\cite{R-S}]\label{thm: RSfilt}
With notations as above, if $\pi$ is of Whittaker type, then the $P_3$-module $V_{\Zb^J}$ has a finite Jordan-H\"older sequence of smooth $P_3$-modules $
0 \subset V_0 \subset V_1 \subset \cdots \subset V_n = V_{\Zb^J}$ such that $
V_0 \simeq \tau_0$ and, for some $I \le n-1$, 
\begin{eqnarray*}
V_{i+1}/V_i \simeq 
\begin{cases}
\tau_1(\xi_i), \ \xi_i \in \Xs(F^\t)& \mbox{($i \le I$),} \\
\tau_2(\rho_i), \ \ \rho_i \in \Ir(G_2) & \mbox{($i >I$).}
\end{cases}
\end{eqnarray*}
We have $V_{\Zb^J} = V_0$, if and only if $\pi$ is supercuspidal.
\end{thm} 
\begin{prop}\label{prop:ondim}
Fix $\psi$ and $b \in F^\t$.
Let $(\pi,V)$ be of Whittaker type.
Except for finitely many $\chi \in \Xs(F^\t)$, the space of functionals $\mu: \pi \to \C$ such that 
\begin{eqnarray*}
\mu(\pi\l(\tb_1(t) \nb'(x) \nb_3(*)\zb(*)\r)v) = \psi(b x)\chi(t) \mu(v), 
\end{eqnarray*}
and the space of functionals $\lambda: \pi \to \C$ such that 
\begin{eqnarray*}
\lambda(\pi\l(\ab'(t)\bnb'(*) \nb_2(x)\zb(*)\r)v) = \psi(b x)\chi(t) \lambda(v)
\end{eqnarray*}
are both one-dimensional.
\end{prop}
\begin{proof}
Note that $\mathrm{pr}(\Tb_1\N'\N_3\Zb^J) = MZ_2' \subset P_3$, and the character $\tb_1(t) \nb'(x) \nb_3(*)\zb(*) \mapsto \psi(b x)\chi(t)$ corresponds to $\eta_b(\chi)$ defined in previous section.
By Theorem \ref{thm: RSfilt}, $V_{i+1}/V_i \simeq \tau_{j}(\s_i)$ for some $\s_i \in \Ir(G_j)$ for $j \in \{1,2\}$.
Therefore, the following sequence is exact: 
\begin{eqnarray}
\Hom_{MZ_2' }(\tau_{j}(\s_i), \eta_b(\chi)) \to \Hom_{MZ_2' }(V_{i+1},\eta_b(\chi)) \to \Hom_{MZ_2' }(V_i,\eta_b(\chi)) \to 0. \label{eqn:rses}
\end{eqnarray}
By Proposition \ref{prop: onedimP}, $\dim_\C \Hom_{MZ_2' }(V_0,\eta_b(\chi))= 1$, and $\Hom_{MZ_2' }(\tau_{j}(\s_i),\eta_b(\chi)) = \{0\}$ for all $i$ except for finitely many $\chi \in \Xs(F^\t)$.
By (\ref{eqn:rses}) and induction, $\dim_\C \Hom_{MZ_2' }(V_{1},\eta_b(\chi)) = \cdots = \dim_\C \Hom_{MZ_2' }(V_n,\eta_b(\chi)) = \dim_\C \Hom_{MZ_2' }(V_{\Zb^J},\eta_b(\chi)) = 1$ except for finitely many $\chi \in \Xs(F^\t)$.
This proves the assertion for the space of $\mu$.
For $\lambda$, use Proposition \ref{prop: onedimP2}.
\end{proof}
For $\pi \in \Ir(\G)$, let $\pi^\imath = \pi \ot (\w_\pi^{-1}\circ \mu)$, which is equivalent to $\pi^\vee$ by Proposition 2.3 of \cite{T}, where $\mu$ indicates the similitude factor. 
For $W \in \W_\psi$, define $W^\imath \in \W_\psi$ by
\begin{eqnarray*}
W^\imath(g) = \w_\pi(\mu(g))^{-1} W(g),
\end{eqnarray*}
and the zeta integrals:  
\begin{eqnarray*}
&&\Xi(s,W) = \int_{F^{\t}} W(\ab'(t)) \nu_{s-\frac{3}{2}}(t) d^\t t, \\
&&Z(s,W) = \Xi(s, \int_{\bar{\N}'} \pi(\nb) W d \nb),
\end{eqnarray*}
where $d^\t t$ and $d \nb$ is chosen so that $\vol(\o^\t) = 1$ and $\vol(\bar{\N}'(\o)) = 1$ respectively.
Now, let $\pi \in \Ir^{gn}(\G)$.
Fix $\psi$.
For $W \in \W_\psi(\pi)$, $Z(s,W)$ converges absolutely to an element in $\C(q^{-s})$ if $s \in \C$ lies in some right half complex plane, and the $\C$-vector subspace $I(\pi) \subset \C(q^{-s})$ spanned by all $Z(s,W)$ is a fractional ideal of the principal ideal domain $\C[q^{\pm s}] := \C[q^s,q^{-s}]$.
Therefore, $I(\pi)$ admits a generator of the form $P(q^{-s})^{-1}$ with $P(X) \in \C[X]$ such that $P(0) = 1$.
Set $L(s,\pi) = P(q^{-s})^{-1}$.
From Proposition \ref{prop:ondim}, we obtain the following functional equation by the standard argument (c.f. \cite{R-S}, \cite{J-PS-S2}).
We omit the proof.
\begin{thm}\label{thm:fesp}
Let $\pi \in \Ir^{gn}(\G)$.
There exists a monomial $\ep(s,\pi,\psi)$ in $q^{-s}$ such that
\begin{eqnarray*}
\frac{Z(1-s, \pi^\imath(\j_0) W^\imath)}{L(1-s,\pi^\imath)} = \ep(s,\pi,\psi) \frac{Z(s,W)}{L(s,\pi)}\label{eqn: locfe}
\end{eqnarray*}
for any $W \in \W_\psi(\pi)$.
It holds that $\ep(s,\pi,\psi)\ep(1-s,\pi^\imath,\psi) =1$.
\end{thm}
For $\psi$ with conductor $\o$, define the analytic root number $\ep_\pi'$ and conductor $n_\pi' \in \Z$ by 
\begin{eqnarray}
\ep(s,\pi,\psi) = \ep_\pi' q^{-n_\pi'(s-\frac{1}{2})}. \label{eqn:defepGSp4}
\end{eqnarray}
\section{Quasi-paramodular forms}\label{sec:QP}
In this and next section, we fix a $\psi$ with conductor $\o$.
Let $\pi \in \Ir^{gn}(\G)$.
Let $\w_\pi$ be the central character of $\pi$, and $\e = \c(w_\pi)$.
For $m \ge 2\e$, define the quasi-paramodular groups $\Kb(m;\e), \Kb_1(m;\e)$ as in introduction.
Define $\Kb^c(m;\e) = \Int(\j_m) \Kb(m;\e)$ and $\Kb_1^c(m;\e) = \{k \in \Kb^c(m;\e) \mid k_{33} \in 1+ \p^{\e}\}$.
Explicitly, $\Kb^c(m;\e)$ consists of $k \in \G$ such that $\det(k) \in \o^\t$ and  
\begin{eqnarray*}
k \in \begin{bmatrix}
\o & \o & \p^{-\e} & \p^{-m} \\
\p^l & \o & \o &\p^{-\e} \\
\p^m & \p^\e & \o & \o \\
\p^{m} & \p^m & \p^l & \o 
\end{bmatrix}.
\end{eqnarray*}
In case of $\e = 0$, these open compact subgroups coincide with the paramodular group $\Kb(m)$, as well as the quasi-paramodular groups.
Let $V(m)$ denote the space of quasi-paramodular forms of level $m$ in $\W_\psi(\pi)$.
For each $W \in V(m)$, define the conjugate $W^c$ by
\begin{eqnarray*}
W^c = \pi^\imath(\j_m)W^\imath\in \W_\psi(\pi^\imath).
\end{eqnarray*}  
Observe that $\pi^\imath(k)W^c = \w_\pi(k_{33})^{-1}W^c$ for $k \in \Kb^c(m;\e)$.
The image of $V(m)$ by $c$ is denoted by $V^c(m)$.
In case of $\e = 0$, $V(m) = V^c(m)$, and we have a decomposition
\begin{eqnarray}
V(m) = V(m)_+ \op V(m)_-, \label{eqn:ALdec}
\end{eqnarray}
where $V(m)_\pm = \{W \in V(m) \mid \pi(\j_m)W = \pm W \}$.
In case of $\e > 0$, $V(m) \neq V^c(m)$, and we call $\Kb_1^c(m;\e)$-invariant Whittaker functions in $\W_\psi$ {\bf coquasi-paramodular forms} of level $m$. 
We call $\Kb_1(m;\e)$-invariant Whittaker functions quasi-paramodular forms of level $m$ including the case of $\e = 0$. 
But, whenever we call $\Kb_1^c(m;\e)$-invariant Whittaker functions coquasi-paramodular forms, we assume $\e > 0$.  
The proof for the existence of nontrivial quasi-paramodular forms (and thus that of coquasi-ones) for the case of $\e > 0$ is easier than that by \cite{R-S} for the case of $\e = 0$.
As in Theorem 4.4.1 of loc. cit., one can show that there is a quasi-$\mathrm{Kl}(\p^n)$-invariant $W \in \W_\psi(\pi)$ such that $W(1) \neq 0$, for a sufficiently large $n$. 
Obviously 
\begin{eqnarray*}
\int_{\Kb(n;\e)/\mathrm{Kl}(\p^n)}\w_\pi(k_{44})^{-1} \pi(k)W dk \in V(n)
\end{eqnarray*}
is not zero at $1$.
Quasi- and coquasi-paramodular forms have the following fine property.
\begin{prop}\label{prop:zeta-simp}
Let $W \in \W_\psi$.
If $W$ is $\N_3(\p^{-r})$-invariant, then 
\begin{eqnarray*}
Z(s,W) = \Xi\l(s, \int_{\bar{\N}'(\p^{r})} \pi(\nb)W d \nb \r).
\end{eqnarray*}
In particular, $Z(s,W) = q^{-r} \Xi(s,W)$, if $W$ is $\bar{\N}'(\p^{r})$-invariant additionally.
\end{prop}
\begin{proof}
Let $f = f_W$ be the first gauge of $W$ (c.f. (\ref{eqn:gaugedef})).
Then, $f$ is right $ N_3(\p^{-r})$-invaraint, and 
\begin{eqnarray*}
Z(s,W) = \int_{F^\t} \int_{F} f(a(t)\bar{n}'(x)) \nu_{s-3/2}(t)dx d^\t t.
\end{eqnarray*}
For $x \not\in \p^r$, $f(a(t)\bar{n}'(x)) = 0$ is verified similar to Proposition \ref{prop:simp-zetagl}.
Hence the assertion.
\end{proof}
We call $\N_3(\p^{-r})$ and $\bar{\N}'(\p^r)$-invariant Whittaker functions {\bf $r$-balanced}.
If $W$ is quasi-paramodular, then
\begin{eqnarray}
Z(s,W) = \Xi(s,W), \ \ \ Z(s,W^c) = q^{-\e} \Xi(s,W^c). \label{eqn:ZXi}
\end{eqnarray}
Additionally if $\pi \in \Ir^{gn}(\G)$ and $W \in \W_\psi(\pi)$, the functional equation is simplified to 
\begin{eqnarray}
q^{-\e}\frac{\Xi(1-s,W^c)}{L(1-s,\pi^\imath)} =\ep'_\pi q^{(m -n_\pi')(s-1/2)}\frac{\Xi(s,W)}{L(s,\pi)}.\label{eqn:FEstd}
\end{eqnarray}
In case of $\e > 0$, we will show that other balanced forms are obtained from quasi-paramodular forms and coquasi-ones of level $m$ by the linear operators $\Gamma_r, \Gamma_{r}^\imath$ defined by  
\begin{eqnarray}
\Gamma_r: W \mapsto \int_{\v^r \Cs_{\e}} \pi(\bar{n}(z))W d z, \ \ \Gamma_{r}^\imath: W \mapsto \int_{\v^{r}\Bs_{l}} \pi(\bar{n}(z))W dz, \label{eqn: Gamma}
\end{eqnarray}
where $l = m-\e$, and  
\begin{eqnarray}
\Cs_a = \{\begin{bmatrix}
x & \\
y & x
\end{bmatrix} \mid x \in \o, y \in \p^{a} \}, 
\Bs_{a} = \{\begin{bmatrix}
x & y \\
 & x
\end{bmatrix} \mid x \in \o, y \in \p^{-a} \},  a \in \Z. \label{eqn:CrBrl}
\end{eqnarray}
\begin{lem}\label{lem:Glop}
With notations as above, 
\begin{enumerate}[i)]
\item If $W$ is quasi-paramodular of level $m$, then $\Gamma_{r}(W)$ is $0$-balanced, and $\pi(\j_0)\Gamma_{r}(W)$ is $(-r)$-balanced for $r \ge \max\{m-2\e, \e \}$.
\item If $W$ is coquasi-paramodular of level $m$, then $\Gamma_{r}^\imath(W)$ is $(r-l)$-balanced, and $\pi^\imath(\j_0)\Gamma_{r}^\imath(W)$ is $(-r)$-balanced for any $r \in \Z$.
\end{enumerate}
\end{lem}
\begin{proof}
i) 
The assertion for $r = l$ is obvious, since $\Gamma_l W$ is a constant multiple of $W$.
It suffices to show that $\Gamma_r(W)$ is $\bar{\N}'(\o), \bar{\N}_3(\p^r), \N_3(\o)$ and $\Zb^J(\p^{-r})$-invariant.
The $\bar{\N}'(\o), \bar{\N}_3(\p^r)$-invariance property is obvious.
We will show the $\N_3(\o), \Zb^J(\p^{-r})$-invariance property by induction.
We also use identities (\ref{eqn: nion}), and 
\begin{eqnarray}
\Int(A^\natural)\bar{n}(C) = \bar{n}(w_2{}^tAw_2 CA). \label{eqn: Atra}
\end{eqnarray}
Let $H^\natural = \{A^\natural \mid A \in H\}$ be the subgroup of $\Kb_1(m;\e)$, where 
\begin{eqnarray*}
H = \begin{bmatrix}
\o^\t & \o \\
\p^{l} & 1+\p^\e 
\end{bmatrix} \subset G_2.
\end{eqnarray*}
If $A \in H$, then the mapping $C \mapsto w_2{}^tAw_2 CA$ induces a translation in the quotient of modules $\v^r\Cs_{\e}/\v^{r+1}\Cs_{\e}$.
Therefore, if $\Gamma_{r+1} (W)$ is invariant under $H^\natural$, then so is $\Gamma_{r}(W)$.
Therefore $\Gamma_r(W)$ is $H^\natural$-invariant.
Now the $\N_3(\o), \Zb^J(\p^{-r})$-invariance property follows from (\ref{eqn: nion}), induction hypothesis and the calculation
\begin{eqnarray*}
&& BC = \begin{bmatrix}
ax +by& bx\\
ay & ax
\end{bmatrix} \in \begin{bmatrix}
\p^r + \p^\e & \o \\
\p^{r+\e} & \p^r
\end{bmatrix}
\subset 
\begin{bmatrix}
\p^\e & \o \\
\p^l & \o
\end{bmatrix}, \\
&& CBC = \begin{bmatrix}
ax^2 + bxy & bx^2  \\
by^2+2axy & ax^2 + bxy
\end{bmatrix} 
\in \begin{bmatrix}
\p^{r+\e} & \p^r \\
\p^{r+2\e}+ \p^{2r+\e} & \p^{r+\e}
\end{bmatrix} \subset 
\begin{bmatrix}
\p^{r+1} & \o \\
\p^{r+\e+1} & \p^{r+1}
\end{bmatrix}
\end{eqnarray*}
for
\begin{eqnarray*}
B = \begin{bmatrix}
a& b \\
&a  
\end{bmatrix} \in \Bs_r,\ \
C = \begin{bmatrix}
x &  \\
y & x
\end{bmatrix} \in \v^r \Cs_\e.
\end{eqnarray*}
ii) Similar to i).
We only check that 
\begin{eqnarray*}
&& BC = \begin{bmatrix}
ax & ay + bx \\
 & ax
\end{bmatrix} \in 
\begin{bmatrix}
\p^l & \o \\
 & \p^l 
\end{bmatrix} \subset 
\begin{bmatrix}
\o & \o \\
 & \p^e
\end{bmatrix}, \\
&& CBC = \begin{bmatrix}
ax^2 & 2axy + bx^2 \\
 & ax^2
\end{bmatrix} \in \begin{bmatrix}
\p^{r+l} & \p^r \\
 & \p^{r+l}
\end{bmatrix} \subset 
\begin{bmatrix}
\p^{r+1} & \p^{r+1-l} \\
 & \p^{r+1}
\end{bmatrix}
\end{eqnarray*}
for
\begin{eqnarray*}
B = \begin{bmatrix}
a & b \\
 & a 
\end{bmatrix} \in \v^{l-r}\Bs_{l}, C = \begin{bmatrix}
x &y \\
 & x
\end{bmatrix} \in \v^r \Bs_{l}.
\end{eqnarray*}
\end{proof} 
The proof of the next is similar to that of Theorem 3.1.3 of \cite{R-S}, and omitted.
\begin{thm}\label{thm: liq}
Let $(\pi, V) \in \Alg(\G)$.
Assume that $V^{Sp(4,F)}$, the subspace of $Sp(4,F)$-invariant vectors in $V$, is $\{0\}$.
Let $\{m_0 < \cdots < m_r \}$ be a finite set of nonnegative integers, and $v_i (\neq 0)$ be $\Kb_1^c(m_i)$-invariant vectors in $V$. 
Then, $v_i$ are linearly independent.  
\end{thm}
The next is the main theorem of this section.
\begin{thm}\label{thm:vnthm}
Let $\pi \in \Ir^{gn}(\G)$, and $W \in \W_\psi(\pi)$ be quasi-paramodular.
If $W(\Tb) = 0$, then $W$ is identically zero.
\end{thm}
In case of $\e = 0$, this is Corollary 4.3.8. of the `$\eta$-principle' of loc. cit.
Although they assumed $\w_\pi = \1$, their argument works as far as $\e = 0$.
We will consider the case of $\e > 0$, mainly.
We need some preparations.
Let $W \in \W_{\psi}(\pi)$.
For $r \in \Z$, set 
\begin{eqnarray*}
W_r = \pi(\eta^{-r}) W.
\end{eqnarray*}
If $W$ is quasi-paramodular, then $W_r$ is $\N_2(\p^{r})$-invariant.
By using Lemma \ref{lem:vanishlem} one can show that $Z(s,W_r) = 0$ if $r < 0$.
We will compute $Z(s,W_r)$ for $r \ge 0$.
By Proposition \ref{prop:zeta-simp},
\begin{eqnarray*}
Z(s,W_r) &=& \Xi\l(s, \int_{\p^{-r}} \pi(\bnf'(z))W_r dz \r)\\
&=& \Xi(s,W_r) + \sum_{m=1}^{r} \Xi\l(s, \int_{\p^{*-m}} \pi(\bnf'(z))W_r dz\r).
\end{eqnarray*}
For $j \in \Z$ and a Laurent series $D(X) = \sum c_n X^{n}$, let 
\begin{eqnarray*}
D(X)_j = q^{-1}\l(-c_{j-1} X^{(j-1)} + (q-1)\sum_{n =j}^\infty c_n X^{n}\r).
\end{eqnarray*}
\begin{lem}\label{lem:NovZ}
With notation as above, if $W \in \W_\psi(\pi)$ is quasi-paramodular, then  
\begin{eqnarray*}
Z(s,W_r) = \sum_{m = 0}^{r} \w_\pi(\v)^{-m} q^{2m(s-1)} \Xi(s,W_{r-m})_m, 
\end{eqnarray*}
\end{lem}
\begin{proof} 
Let $m$ be a negative integer.
Let $z \in \p^{*m}$.
By using (\ref{eqn:usid}), and the $\G_2(\o)'$-invariance property of $W_r$, we compute 
\begin{eqnarray} 
W_r \l(\af(t)\bnf'(z)\r) &=& W_r\l(\ab'(t)\nf'(z^{-1})\jb'(z)\nf'(z^{-1})\r) \nonumber \\
&=& \psi(tz^{-1}) W_r \l(\ab'(t)\jb'(z)\r)  \label{eqn:W^az} \\
&=& \psi(tz^{-1}) W_r \l(\ab'(t)\jb'(z)\jb'\r) \nonumber \\
&=& \psi(tz^{-1}) \w_\pi(\v^{m}) W_{r+ m}\l(\ab'(tz^{-2})\r). \nonumber
\end{eqnarray}
Therefore, 
\begin{eqnarray*}
\int_{\o^\t} \int_{\p^{*m}} W_r(\af'(\v^iu)\bnf'(z)) dz d u = c_m \w_\pi(\v)^{m} W_{r+m}\l(\ab'(\v^{i-2m})\r), \\
\mbox{where} \ \ 
c_m = \begin{cases}
q^{-m-1} (q-1) & \mbox{if $i \ge m$,} \\
-q^{-m-1} & \mbox{if $i = m -1$,} \\
0 & \mbox{otherwise.}
\end{cases}
\end{eqnarray*}
From this, the assertion follows.
\end{proof}
\begin{lem}\label{lem:NovZ2}
If $W \in \W_\psi(\pi^\imath)$ is coquasi-paramodular, then 
\begin{eqnarray*}
Z(s,W_r) = q^{- \e} \Xi\l(s,W_r \r) + \sum_{\e-r \le m < \e} c_m q^{(\e-m)(s-\frac{3}{2})},
\end{eqnarray*}
where 
\begin{eqnarray*}
c_m = \begin{cases}
\Sb_m(\w_\pi^{-1},u_{\w_\pi^{-1}}) W_r \l([\v^{m- \e}; \v^{m}u_{\w_\pi^{-1}}]'\r) & \mbox{if $m > 0$,} \\
 \w_\pi^{-1}( u_{\w_\pi}\v^{m}) \Gb(\w_\pi, u_{\w_\pi})W_{r + m}(\ab_{-m- \e} \jb'(1)) & \mbox{if $m \le 0$}.
\end{cases}
\end{eqnarray*}
The notation $[*;*]$, $\Gb(\w_\pi, u_{\w_\pi})$ and $\Sb_m(\w_\pi^{-1},u_{\w_\pi^{-1}})$ are defined in sect. \ref{sec:GLd}.
\end{lem}
\begin{proof}
By Proposition \ref{prop:zeta-simp}, $Z(s,W_r) = q^{- \e} \Xi(s,W_r) + \sum_{m=1- \e}^{r- \e} \Xi_m$, where 
\begin{eqnarray*}
\Xi_m =  \Xi\l(s, \int_{\p^{*-m}} \pi(\bnf'(z))W_r dz\r).
\end{eqnarray*}
We will show $\Xi_m = c_m q^{(\e-m)(s-\frac{3}{2})}$ for the constant $c_m$ as in the assertion.
Let $m = ord(z)$.
Suppose $1 \le m \le \e -1$. 
By the $K_1(\e)'(\subset G_2')$-invariance property of $W_r$, for $u \in \o^\t$,
\begin{eqnarray*}
W_r\l(\af'(\v^iu)\bnf'(z)\r) = W_r\l(\af'(\v^iu)\bnf'(z) \af'(u)^{-1}\r) = W_r\l(\af'(\v^i)\bnf'(u^{-1}z) \r).
\end{eqnarray*}
By Lemma \ref{lem:e-lem}, this is zero unless $i = m - \e$.
Therefore, $\Xi_m = c_m q^{(\e-m)(s-\frac{3}{2})}$.
Suppose $\e-r \le m \le 1$.
By (\ref{eqn:W^az}),
\begin{eqnarray*}
\iint W_r(\af'(\v^i u)\bnf'(\v^m u_\w v )) du dv 
&=& \iint \psi(\v^i u (\v^m u_\w v)^{-1}) W_r \l(\af'(\v^i u)\jb'(\v^m u_\w v)\r) du dv \\
&=& W_r \l(\ab_i\jb'(\v^m u_{\w_\pi})\r)
\iint \psi\l(\frac{\v^{i-m}}{ vu^{-1}u_{\w_\pi}}\r) \w_\pi(vu^{-1}) du dv
\end{eqnarray*}
where integrations are over $u \in \o^\t$ and $v \in \o^\t$.
The last double integral equals the Gauss sum, which is zero unless $i = m- \e$.
Since 
\begin{eqnarray*}
W_r \l(\ab_i \jb'(\v^m u_{\w_\pi})\r) = \w_\pi^{-1}(u_{\w_\pi} \v^m) W_{r+m}\l(\ab_{i-2m} \jb'(1)\r),
\end{eqnarray*}
we have $\Xi_m = c_m q^{(\e-m)(s-\frac{3}{2})}$.
\end{proof}
\begin{prop}\label{prop:dvqv}
Let $\pi \in \Ir^{gn}(\G)$, and $W \in \W_\psi(\pi)$ be quasi-paramodular.
\begin{enumerate}[$i)$]
\item If $W(\Tb) = 0$, then $W(\Qb) = 0$.
\item If $W(\Tb) = W^c(\Tb) = 0$, then $W^c(\Qb) = 0$.
\end{enumerate}
\end{prop}
\begin{proof}
$i)$ follows from the decomposition $\Qb = \N \Tb \G_2'(\o)$.
In case of $\e = 0$, $W^c$ is paramodular, and $ii)$ follows from $i)$.
Hence, we may assume $\e > 0$.
Let $m$ be the level of $W$.
Let $u = u_{\w_\pi^{-1}}$ be the element in $\o^\t$ in defined in sect. \ref{sec:GLd}.
We take $\mathfrak{N}= \{ \bnf'(\v^i u) \mid 0 \le i \le \e\}$ for the representatives of $F^\t \N \Ab' \Tb_1 \bs \Qb/(\Qb \cap \Kb^c(m;\e))$.
It suffices to show that $W^c(\ab^r_{s} \bnb) = 0$ for $r,s \in \Z, \bnb \in \mathfrak{N}$.
By assumption, $W^c(\ab^r_{s}) = 0$.
From Lemma \ref{lem:vanishlem}, it follows that $W^c(\ab^r_{s}\bnb) = 0$ if $r < 0$, and that $W^c(\ab^r_{s} \jb'(1)) = 0$ if $s < \e$.
By Lemma \ref{lem:e-lem}, (\ref{eqn:usid}), our remained task is to show that 
\begin{eqnarray}
W^c(\ab^r_{s-\e} \bnf'(\v^{s}u)) = W^c(\ab^r_{i} \jb'(1)) = 0 \label{eqn: dvqv}
\end{eqnarray}
for $r \ge 0$, $i \ge - \e$ and positive $s \le \e-1$.
By i) and Lemma \ref{lem:NovZ}, $Z(s,W_r) = 0$. 
By the functional equation, $Z(s,(W^c)_r) = 0$.
Since $W^c(\Tb) = 0$, we have $\Xi(s,(W^c)_r) = 0$.
Now (\ref{eqn: dvqv}) follows from Lemma \ref{lem:Gauss2}, \ref{lem:e-lem}, and \ref{lem:NovZ2}.
\end{proof}
\begin{prop}\label{prop: wvzj}
Let $(\pi,V) \in \Ir^{gn}(\G)$.
If $W \in V^c(m)$ for some $m$ vanishes on $\Qb$, then $W \in V(\Zb^J)$.
\end{prop}
\begin{proof}
We have constructed a (unique up to a constant multiple) nontrivial functional for each $\tau_j(\s)$ with $\s \in \Ir(G_j), j = 1,2$ in section \ref{sec:P3}.
In case of $j = 1$, the functional is $\lambda_1^b$ with $\chi = \s\nu_{-1}$, and corresponds to the functional 
\begin{eqnarray*}
\W_\psi(\pi^\imath) \ni W \mapsto \int_{F^\t} \int_{\bar{\N}'} W(\nb \ab'(t)\tb_1(b))\s^{-1} \nu_{s}(t) d \nb d^\t t 
\end{eqnarray*}
where $s = 1$.
In case that $j = 2$, and $\s \in \Ir^{gn}(G_2)$ (hence infinite-dimensional), it is $\mu_2^b$ and corresponds to the functional 
\begin{eqnarray*}
\W_\psi(\pi^\imath) \ni W \mapsto \int_{F^\t} W(\tb_1(u)\ab'(b))\w_\s^{-1} \nu_{s}(u) d^\t u,
\end{eqnarray*}
where $s = 0$.
In case that $j = 2$, and $\s = \xi \circ \det$ with $\xi \in \Xs(F^\t)$, it is $\mu_2'$ and corresponds to the functional 
\begin{eqnarray*}
\W_\psi(\pi^\imath) \ni W \mapsto \int_{F^\t}\int_{\bar{\N'}} W(\tb_1(u)\nb)\xi^{-2}\nu_{s}(u) d \nb d^\t u, 
\end{eqnarray*}
where $s=0$.
Since $W(\Qb) = 0$, all these functionals send $W$ to $0$.
Now, let $\s \in \Ir(G_j)$ and $f \in \tau_j(\s)$.
If $f$ is sent to $0$ by the corresponding functional, and satisfies
\begin{eqnarray}
f(pk) = \w_\pi(k_{22})^{-1}f(p) \ \ \mbox{for $k \in  
\begin{bmatrix}
\o & \o & \p^{-\e} \\
\p^\e & \o & \o \\
 & & 1
\end{bmatrix} (=\mathrm{pr}(\Kb^c(m;\e)))$}, \label{eqn:P3esub}
\end{eqnarray}
then we have $f = 0$.
Indeed, it follows from Lemma \ref{lem:e-lem2} below in the first case, from the newform theory for $G_2$  with $m_\s = \e$ in the second case, and from the one-dimensionality of $\tau_2(\xi \circ \det)$ in the third case.
Therefore, $W \in V_0$ by Theorem \ref{thm: RSfilt}. 
Let $\W_0 = \{W \in V \mid W(\Qb) = 0 \}$.
By the proof of Theorem 4.3.5 of \cite{R-S}, $\W_0 \subset V(\Zb^J)$.
This completes the proof.
\end{proof}
\begin{Rem}
The last two integrals are absolutely convergent if $\Re(s) >>0$, and analytically continued to the whole complex plane.
They are related to the so-called degree five $L$-function of $\pi$.
We will discuss them in a forthcoming paper.
\end{Rem}
\begin{lem}\label{lem:e-lem2}
Let $\xi \in \Xs(F^\t)$, and $f \in \tau_1(\xi)$.
If $f$ satisfies (\ref{eqn:P3esub}), and $\lambda_1^b(f) = 0$ for $\chi = \xi \nu_{-1}$ and any $b \in F^\t$, then $f$ is identically zero.
\end{lem}
\begin{proof}
By (\ref{eqn:P3esub}) and the decomposition $P_3 = N A Z_2' G_2'(\o)$, it suffices to show that 
$f(z_2'(b) \bar{n}'(z)) = 0$ for any $b \in F^\t, z \in F$.
In case of $\e = 0$, the assertion follows immediately from the decomposition.
Assume that $\e > 0$.
Let $z \not\in \p^\e$.
For $l = ord(z) -\e +1$, and $x \in \o$,  
\begin{eqnarray*}
f(z_2'(b) \bar{n}'(z)) &=& \xi(\v)^{-l}f(a_lz_2'(b) \bar{n}(z)) \\
&=& \xi(\v)^{-l}f(n(x)a_lz_2'(b) \bar{n}(z)) \\
&=& \xi(\v)^{-l}f(a_lz_2'(b) \bar{n}(z)k') \\
&=& \w_\pi(1 -\v^{-l} zx)^{-1}\xi(\v)^{-l}f(a_lz_2'(b) \bar{n}(z)),
\end{eqnarray*}
where $k = \Int^{-1}([\v^l;z])n(x) \in K(\e) \subset G_2$(c.f. (\ref{eqn: elemad})).
There is an $x \in \o$ such that $\w_\pi(1 -\v^{-l} zx) \neq 1$.
From Lemma \ref{lem:vanishlem}, $f(z_2'( b ) \bar{n}'(z)) = 0$ follows in this case.
Consequently, $\lambda_1^b(f) = \vol(\bar{N}(\p^\e)) f(z_2'(b))$, and the assertion follows.
\end{proof}
Let $(\pi,V) \in \Alg(\G)$.
For a moment, by abuse of notation, we denote by $V^c(m)$ the subspace of $\Kb_1^c(m;\e)$-invariant vectors in $V$.
We will use the following level $+2$ raising operator $\eta$, and level $+1$ one $\a_m$ for $V^c(m)$ ($\e$ may be zero): 
\begin{eqnarray*}
&& \eta: V^c(m) \ni v \longmapsto \pi^\imath(\eta) v \in V^c(m+2), \\ 
&& \a_m: V^c(m) \ni v \longmapsto \sum_{k \in \Kb^c(m+1;\e)/\Kb^c(m;\e) \cap \Kb^c(m+1;\e)} \pi^\imath(k) v \in V^c(m+1).
\end{eqnarray*}
Computing the coset space, we have
\begin{eqnarray*}
\a_m v - \sum_{x \in \o/\p} \pi^\imath(\nb_3(\v^{-m-1}x))v &=& \pi^\imath(\jb''_{m+1}) v  \\
 &=& \pi^\imath(\jb''_{m+1})\pi^\imath(\jb''_{m}) v \\
&=& \eta v.
\end{eqnarray*}
\begin{prop}\label{prop: vzjv0}
Let $(\pi,V) \in \Alg(\G)$.
Assume that $V^{Sp(4,F)} = \{0\}$.  
If $v \in V(\Zb^J)$ is $\Kb_1^c(m;\e)$-invariant ($\e$ may be zero), then $v = 0$.
\end{prop}
\begin{proof}
Write the level raising operator $\a_m = \eta + q z_{m+1}$, where $z_{m+1}$ is the linear operator defined by $\vol(\p^{-m-1})^{-1}\int \pi(\bnb) d \bnb$ with integration over $\bnb \in \bar{\N}(\p^{-m-1})$.
At first, we will show by induction that there exist certain linear operators $\b_r: V^c(m) \to V^c(r+m+1)$ and $\g_r: V^c(m) \to V^c(r+m)$ such that
\begin{eqnarray}
z_{m+r} = \b_r+ \g_r. \label{eqn: abceta}
\end{eqnarray}
For $r = 1$, this holds obviously.
Assume (\ref{eqn: abceta}) for $r \ge 1$.
Since $v$ lies in $V^c(m)$, 
\begin{eqnarray*}
z_{r+m+1}v &=& z_{r+m+1} \circ z_{r+m}v \\
&=& q^{-1}(\a_{r+m} -\eta) \circ (\b_r+ \g_r)v \\
&=& q^{-1}\l(-\eta \g_r v + (\a_{r+m} \circ \g_r + z_{r+m+1}\circ \b_r)v \r)\\
&=& -q^{-1}\eta \g_r v +q^{-1} (\a_{r+m} \circ \g_r + \b_r)v
\end{eqnarray*}
where the assumption $\b_r v \in V^c(r+m+1)$, and $(z_{r+m+1}\circ \b_r)v = \beta_r v$ are used at the last equality.
Therefore,  
\begin{eqnarray}
\b_{r+1} &:=& -q^{-1}\eta \g_r:  V^c(m) \to V^c(r+m+2) \label{eqn:recz} \\
\g_{r+1} &:=& q^{-1}(\a_{r+m} \circ \g_r + \b_r): V^c(m) \to V^c(r+m+1)
\label{eqn:recz2}
\end{eqnarray}
are the desired linear operators.
This proves (\ref{eqn: abceta}).
Next, we will show $v = 0$.
Since $v \in V(\Zb^J)$, 
\begin{eqnarray*}
(\a_{r+m-1} -\eta)\circ \cdots \circ (\a_m-\eta) v = z_{r+m} v = 0
\end{eqnarray*}
for a sufficiently large $r$.
Since $\b_{r} v$ and $\g_r v$ have different levels and are linearly independent, $\b_{r} v = \g_r v = 0$ by Theorem \ref{thm: liq}.
By (\ref{eqn:recz}), $\eta \g_{r-1} v = 0$.
Since $\eta$ is obviously injective, $\g_{r-1}v = 0$.
Therefore $\a_{r+m-1} \circ \g_{r-1}v = 0$.
By (\ref{eqn:recz2}), $\b_{r-1} v = 0$.
Thus $z_{r+m-1} v = \b_{r-1}v + \g_{r-1}v = 0$.
Hence, $v = 0$.
\end{proof}
Now, we can prove Theorem \ref{thm:vnthm}.
Suppose that $W \in V(m)$ with $W(\Tb) = 0$.
By Proposition \ref{prop:dvqv}, $W(\Qb) = 0$.
Let $i$ be an arbitrary nonnegative integer.
By Lemma \ref{lem:NovZ}, $Z(s,W_i) = 0$. 
By the functional equation, $Z(s,(W^c)_i) = 0$.
By Lemma \ref{lem:NovZ2}, $\Xi(s,(W^c)_i) = 0$.
Thus, $W^c(\Tb) = 0$.
By Proposition \ref{prop:dvqv} again, $W^c(\Qb) = 0$.
By Proposition \ref{prop: wvzj}, $W^c \in V(\Zb^J)$.
By Proposition \ref{prop: vzjv0}, $W^c$ is identically zero, and so is $W$.
\begin{lem}\label{lem:m>n}
If $\pi \in \Ir^{gn}(\G)$, then $m_\pi \ge n_\pi'$.
\end{lem}
\begin{proof}
Let $W \in V(m_\pi)$.
Let $i_0$ be the minimal nonnegative integer such that $\Xi(s,W_{i_0}) \neq 0$.
By Theorem \ref{thm:vnthm}, such $i_0$ exists.
By Lemma \ref{lem:NovZ}, $Z(s,W_{i_0}) = \Xi(s,W_{i_0}) \neq 0$.  
The functional equation for $W_{i_0}$ is 
\begin{eqnarray*}
\frac{Z(1-s, \l(W^c\r)_{i_0})}{L(1-s,\pi^\imath)} = \ep'_\pi q^{(m_\pi- 2i_0-n_\pi')(s-\tfrac{1}{2})} \frac{\Xi(s, W_{i_0})}{L(s,\pi)}. \label{eqn: zcoindfe}
\end{eqnarray*}
The right hand side lies in $q^{(m_\pi- 2i_0-n_\pi')s}\C[q^{-s}]$.
By Lemma \ref{lem:NovZ2}, the left hand side lies in $q^{-i_0s}\C[q^{s}]$.
In case of $\e = 0$, we may assume $W^c = \pm W$ by (\ref{eqn:ALdec}) and the left hand side lies in $\C[q^{s}]$.
Therefore, $m_\pi - n_\pi' -i_0 \ge 0$ in any case.
Thus the assertion. 
\end{proof}
\begin{thm}\label{thm: coinLN}
Let $\pi \in \Ir^{gn}(\G)$.
Assume that $V(m)$ contains a $W_0$ such that 
\begin{eqnarray}
L(s,\pi) = Z(s,W_0), \ \ \mbox{and} \ \ \ c_0 L(s,\pi^\imath) = Z(s, W_0^c) \label{eqn: coinLN}
\end{eqnarray}
for a constant $c_0$.
Then, $c_0= \ep'_\pi$, and $m = n_\pi' = m_\pi$.
Further, $V(m_\pi)$ is spanned by $W_0$. 
\end{thm}
\begin{proof}
From (\ref{eqn: coinLN}), and the functional equation (\ref{eqn:FEstd}) for $W = W_0$, it follows that $\ep'_\pi = c_0$ and $m = n_\pi'$.
By Lemma \ref{lem:m>n}, $m= m_\pi = n_\pi'$.
For the last assertion, we will show that an arbitrary $W \in V(m)$ is a constant multiple of $W_0$.
Since $\Xi(s,W)$ is in $\C \jump{q^{-s}}$, the ring of formal power series in $q^{-s}$, and $\Xi(1-s, W^c)$ is in $\C \jump{q^s}$, 
\begin{eqnarray*}
\C[q^s] \ni q^{-\e}\frac{\Xi(1-s, W^c)}{L(1-s,\pi^\imath)} = \ep'_\pi \frac{\Xi(s, W)}{L(s,\pi)} \in \C[q^{-s}].
\end{eqnarray*}
Therefore, these quotients are constants, and there exists a constant $c_W$ such that 
\begin{eqnarray*}
\Xi(s,W -c_WW_0) = \Xi(1-s, \l(W - c_WW_0\r)^c) = 0. \label{eqn: DsW0cW}
\end{eqnarray*}
Set $W' = W - c_W W_0$.
We will claim by induction that $\Xi(s, W_r') = 0$ for any $r \ge 0$.
Assume that $\Xi(s, W_{i}') = \Xi(1-s, (W'^c)_{i}) = 0$ for all $i < r$.
Then, $Z(s,W_r') = \Xi(s,W_r')$ by Lemma \ref{lem:NovZ}.
The functional equation for $W_r'$ is 
\begin{eqnarray*}
\frac{Z(1-s, (W'^c)_r)}{L(1-s,\pi^\imath)} = \ep'_\pi q^{- 2r(s-\tfrac{1}{2})} \frac{\Xi(s, W_r')}{L(s,\pi)}. \label{eqn: zcoindfe}
\end{eqnarray*}
(note that $\l(W'^c\r)_r = \pi^\imath(\j_{m-2r})W_r^\imath$.)
The right hand side lies in $q^{- 2rs}\C[q^{-s}]$.
In case of $\e > 0$, the left hand side lies in $q^{-rs}\C[q^s]$ by Lemma \ref{lem:NovZ2}.
In case of $\e = 0$, $W^c$ is also paramodular, and the left hand side lies in $\C[q^s]$ by Lemma \ref{lem:NovZ}, again.
Hence both sides are zero, and the claim is verified.
Now, $W = c_W W_0$ by Theorem \ref{thm:vnthm}.
This completes the proof.
\end{proof}
We will call $W_0$ as in this thoerem the {\bf newform} of $\pi$, and denote by $W_\pi$.
\section{Hecke operators}\label{sec:Hecke}
Let $\chi \in \Xs(F^\t)$, and $(\s,V) \in \Ir(G_2)$.
The Klingen parabolic induction $\chi \rtimes \s$ consists of smooth $V$-valued functions $f$ on $\G$ such that
\begin{eqnarray}
L(q^{-1})f &=& |t\det(g)^{-1}| \chi(t)\s(g)f, \nonumber \\
& &\mbox{where} \ \ 
q = \begin{bmatrix}
t & * & * \\
 & g& * \\
 & & t^{-1} \cdot\det(g)
\end{bmatrix} \in \Qb, \ \ t \in F^\t, g \in G_2. \label{eqn:Kldef}
\end{eqnarray}
A Klingen parabolic induction has a unique generic constituent (submodule, c.f. sect. 2.4. of \cite{R-S}).
We call $\chi \rtimes \s$ a Klingen parabolic induction from supercuspidal when $\s$ is supercuspidal.
By the work of \cite{T}, when $\pi \in \Ir^{gn}(\G)$ is supercuspidal, or the constituent of a Klingen parabolic induction from supercuspidal, $L(s,\pi)$ equals $1$.
In this section, we devote to prove 
\begin{thm}\label{thm:Ls1}
Let $\pi \in \Ir^{gn}(\G)$ be supercuspidal, or a constituent of the Klingen parabolic induction from supercuspidal.
Assume $\e > 0$.
Then, there exists the newform $W_\pi$ in $V(m_\pi)$ ($n_\pi' = m_\pi$ and $\dim V(m_\pi) = 1$ by Theorem \ref{thm: coinLN}). 
The newform $W_\pi$ and its conjugate $W_\pi^c$ take the following values on $\Tb$: 
\begin{eqnarray*}
W_\pi(\ab^i_{r}) 
&=& \begin{cases}
1 & \mbox{if $(i,r) = (0,0)$,} \\
0& \mbox{otherwise}.
\end{cases} \label{eqn: WTv}  \\ 
W_\pi^c(\ab^i_{r}) 
&=& \begin{cases}
\ep'_\pi & \mbox{if $(i,r) = (0,0)$,} \\
-q^2\ep'_\pi & \mbox{if $(i,r) = (1,0)$,} \\
0& \mbox{otherwise}.
\end{cases} \label{eqn:W*Tv}
\end{eqnarray*} 
\end{thm}
See Corollary 7.4.6. \cite{R-S} for the case of $\e = 0$.
In this section, we assume 
\begin{eqnarray*}
\e > 0.
\end{eqnarray*}
Our proof consists of four steps. \\

{\bf Step 1.}
For a nontrivial polynomial $\sum c_n X^{n} \in \C[X^\pm]:=\C[X, X^{-1}]$ with $X = q^{-s}$, we call its {\bf range} the pair of the minimal and maximal integers $n$ such that $c_n \neq 0$.
In case of $L(s,\pi) = 1$, for any $W \in \W_\psi(\pi),i \in \Z$, $Z(s,W_i)$ lies in $\C[X^\pm]$ by definition.
In this step, we show that $\Xi(s,W_i) \in \C[X^\pm]$, if $W$ is quasi-paramodular.
\begin{lem}\label{lem:havev}
Let $\pi \in \Ir^{gn}(\G)$, and $W \in V(m)$.
If $L(s,\pi) =1$, then $W(\ab^i_r) = W^c(\ab^i_r) = 0$ for sufficiently large $i,r$.
\end{lem}
\begin{proof}
Let $i$ be a nonnegative integer.
Since $L(s,\pi)$ equals $1$, so does $L(s,\pi^\imath)$.
Therefore, both $Z(s,W_i)$ and $Z(s, (W^c)_i)$ are polynomials in $q^{\pm s}$.
Let $(c_i,d_i)$ and $(c^*_i,d^*_i)$ be their ranges, respectively.
By Lemma \ref{lem:NovZ}, \ref{lem:NovZ2}, both $\Xi(s,W_i)$ and $\Xi(s, (W^c)_i)$ are polynomials in $q^{- s}$.
Let $(a_i,b_i)$ and $(a_i^*,b_i^*)$ be their ranges, respectively. 
From the functional equation for $W_i$, 
\begin{eqnarray}
(c_{i} + n_\pi' - m +2 i, d_{i} + n_\pi' - m + 2i) = (-d^*_{i}, -c^*_{i}). \label{eqn:idrange}
\end{eqnarray}
Now, assume that $\Xi(s,W_i) \neq 0$ for infinitely many $i$'s. 
Then, we may take an $i_1 \ge m -n_\pi'$ so that $\Xi(s,W_{i_1}) \neq 0$, and $b_{i_1}+2i_1 \ge b_{n} + 2n$ for all $n < i_1$.
By Lemma \ref{lem:NovZ}, $d_{i_1} = b_{i_1}$.
By Lemma \ref{lem:NovZ2}, $c^*_{i_1} \ge  -i_1$.
By (\ref{eqn:idrange}), $b_{i_1} + n_\pi' -m + 2i_1\le i_1$.
Since $b_{i_1} \ge a_{i_1} \ge 0$ by Lemma \ref{lem:vanishlem}, $i_1 \le m -n_\pi' -b_{i_1} < m -n_\pi'$.
This is a contradiction.
Hence, $\Xi(s,W_i) = 0$ for sufficiently large $i$.
Therefore, there is an integer $I$ such that $\Xi(s,W_i) = 0$ for all $i > I$ and $b_i < I$ for all $i < I$.
If $i > 2I$, then $Z(s,W_i) = 0$ by Lemma \ref{lem:NovZ}, $Z(s,(W^c)_i) = 0$ by the functional equation, and $\Xi(s,(W^c)_i) = 0$ by Lemma \ref{lem:NovZ2}. 
This completes the proof.
\end{proof}
{\bf Step 2.} 
In this and next steps, we assume that $(\pi,V) \in \Ir(\G)$ is unitary, and use several Hecke operators.
For $h \in \G$, let $\Tc_K(h)$ denote the Hecke operator acting on $V^K$ defined by 
\begin{eqnarray*}
\Tc_K(h) v = \sum_{t \in K/K \cap \Int(h)K} \pi(t h) v.
\end{eqnarray*}
\begin{lem}\label{lem:selfad}
With notations as above, if $(\pi,V)$ is unitary, then the Hecke operator $\Tc_K(h)+ \Tc_K(h^{-1})$ on $V^K$ is diagonalizable.
In particular, if $\Tc_K(h) = \Tc_K(h^{-1})$, then $\Tc_K(h)$ is diagonalizable.
\end{lem}
\begin{proof}
Let $\langle *, * \rangle$ denote the inner product in $V$.
In general, $\langle \Tc_K(h) v, w\rangle = \langle v, \Tc_K(h^{-1})w \rangle$ for $v,w \in V^K$ (see Lemma 6.5.1.of \cite{R-S}.). 
From this, the assertion follows immediately.
\end{proof}
We use the diagonalities of Hecke operators repeatedly, and therefore the unitarity assumption is needed.
In this step, we show the next basic inequality $m_\pi > 2\e$, which is an analogue of Proposition \ref{prop:mpicr}.
This inequality is essentially important for the comparison of Hecke operators and level descending $V(m_\pi) \to V(m_\pi -1)$.
Since the quasi-paramodular group $\Kb_1(m;\e)$ is defined for $m \ge 2\e$, one cannot consider the level descending $V(m_\pi) \to V(m_\pi -1)$ in case of $m_\pi = 2\e$.
In \cite{R-S}, for the $PGSp(4)$ case, to compute some Hecke operators, the condition $m_\pi \ge 2$ for supercuspidal $\pi$ was used.
\begin{prop}\label{prop: cricase}
Let $\pi \in \Ir^{gn}(\G)$ be unitary.
If $L(s,\pi) = 1$, then $m_\pi > 2\e$. 
\end{prop}
Let $\Sc = \Tc_K(\eta) + \Tc_K(\eta^{-1}) \curvearrowright V(m_\pi)$ with $K = \Kb_1(m_\pi;\e)$.
This Hecke operator is diagonalizable by Lemma \ref{lem:selfad}.
Let with $l = m_\pi - \e$.
We compute   
\begin{eqnarray}
\Sc &=& \sum_{} \pi\l(\nb_2(x) \nb_{3}(y) \zb(z\v^{- \e}) \eta \r) \label{eqn: Sop}\\
&& + \sum_{} \pi\l(\eta^{-1} \bnf_2(x\v^{l-1}) \bnf_{3}(y \v^{l-1}) \bz(z\v^{m_\pi-2})\r), \nonumber  \end{eqnarray}
where both sums are over $x,y \in \o/\p,z \in \o/\p^2$.
In case of $m_\pi > 2 \e$, one can find that the latter sum is zero, by comparing with the level descending $\int_{\Kb_1(m_\pi-1;\e)} \pi(k) dk: V(m_\pi) \to V(m_\pi-1)$.
In case of $m_\pi = 2 \e$, the sum is not zero, as follows.
For $W \in V(2\e)$, we set 
\begin{eqnarray*}
W' = \sum_{z \in \v^{\e-1}\Cs_{\e-1}/\v^\e\Cs_\e} \pi(\bar{n}(z))W, \ \ 
W'' = \sum_{x \in \o/\p} \pi(\bnb_2(x\v^{l-1}))W', 
\end{eqnarray*}
where $\Cs_a$ is defined in (\ref{eqn:CrBrl}). 
\begin{lem}\label{lem:cricase}
With notation as above, $W''(\ab^i_r) = W'(\ab^i_r) = W(\ab^i_r)$ for $i,r \ge 0$.
\end{lem}
\begin{proof}
For the first identity, consider the second gauge $\xi$ of $\pi(\ab_r)W$, which is quasi-invariant on $K(\e)$. 
The mapping $C \mapsto w_2 {}^tA w_2 CA$ induces a translation in $\v^{\e-1}\Cs_{\e-1}/\v^\e\Cs_{\e}$ if $A \in K(\e)$.
Hence $R(k)\xi = \w_\pi(k_{22}) \xi$ for $k \in K(\e)$ by the identity (\ref{eqn: Atra}). 
It suffices to show that 
\begin{eqnarray*}
0 = \sum_{x \in \o^\t/\p} \xi (a_i \bar{n}(x \v^{\e-1})) 
\end{eqnarray*}
for $i \ge 0$.
This follows from Lemma \ref{lem:e-lem}.
For the second identity in case of $\e=1$, we compute, for $i,r \ge 0$,
\begin{eqnarray*}
\sum_{} W(\ab^i_r \bar{n}(C)) &=&\sum_{} W\l(\ab^i_r n(C^{-1})j(C)n(C^{-1})\r) \\
&=& \sum_{}  W\l(\ab^i_r j(C)\r)\\
&=&  \sum_{} \w_\pi(x)W\l(\ab^i_r \bnf_2(y) j(1_2)\r) = 0
\end{eqnarray*}
with the sums being over $x \in (\o/\p)^\t, y \in \o/\p^2$, where 
\begin{eqnarray*}
C = \begin{bmatrix}
x & \\
y &x
\end{bmatrix}.
\end{eqnarray*}
By a similar argument, 
\begin{eqnarray*}
\sum_{y \in (\o/\p^{2})^\t} \pi(\bz(y)) W(\ab^i_r)= \sum_{y \in (\o/\p)^\t} \pi(\bz(y\v)) W(\ab^i_r) = 0.
\end{eqnarray*}
Now the second identity in case of $\e = 1$ follows immediately.
For the case of $\e >1$, if $c \in \v^{\e-1}\Cs_{\e -1} \setminus \v^\e \Cs_{\e}$, then $W(\ab^i_r \bar{n}(c)) = 0$ follows from Lemma \ref{lem:vanishlem}, and the identity (\ref{eqn: nion}) with 
\begin{eqnarray*}
b \in \begin{bmatrix}
\o^\t & \\
 & \o^\t
\end{bmatrix} \ \cup \begin{bmatrix}
 &\p^{*(1- \e)} \\
 &
\end{bmatrix}.
\end{eqnarray*}
The second identity follows in this case.
This completes the proof.
\end{proof}
Now, we can prove Proposition \ref{prop: cricase}.
By Lemma \ref{lem:selfad}, it suffices to show that each eigen form $W \in V(2\e)$ of $\Sc$ is identically zero.
Let $\lambda_\Sc$ be the eigenvalue of $W$.
From (\ref{eqn: Sop}) and Lemma \ref{lem:cricase}, it follows that  
\begin{eqnarray*}
\lambda_\Sc W(\ab^i_r) = q^4 W(\ab^{i+1}_r) + W(\ab^{i-1}_r), \ \ i \ge 0.
\end{eqnarray*}
Fix $r \ge 0$. 
By Theorem \ref{thm:vnthm} and Lemma \ref{lem:havev}, we can take the maximal integer $i_0$ such that $W(\ab^{i_0}_r) \neq 0$ for some $r$, if $W$ is not identically zero.
By this recursion formula with $i = i_0+1$, $W(\ab^{i_0}_r) = 0$.
This is a contradiction.
Hence, $W$ is identically zero.

An immediate consequence of this proposition is the next: 
\begin{prop}\label{prop:htW}
Let $\pi \in \Ir^{gn}(\G)$ be unitary.
Assume $L(s,\pi) = 1$.
If $W \in V(m_\pi)$ is a nontrivial form, then, $\Xi(s,W_i) = 0, i \ge 1$ and $\Xi(s,W)$ and $\Xi(s,W^c)$ are nonzero.
\end{prop}
\begin{proof}
Let $l = m_\pi- \e$.
Since $L(s,\pi) = 1$, $l -1 = m_\pi - \e -1 > \e -1 \ge 0$ by Proposition \ref{prop: cricase}.
Then, $W_1 := \eta W$ is invariant under the subgroup  
\begin{eqnarray*}
\Int(\eta)\Kb_1(m_\pi;\e) = 
\begin{bmatrix}
\o & \p & \p & \p^{2-l} \\
\p^{l-1} & \o & \o & \p \\
\p^{l-1} & \o & \o & \p \\
\p^{m_\pi-2} & \p^{l-1}& \p^{l-1} & 1 + \p^\e
\end{bmatrix},
\end{eqnarray*}
and $W_1' := \sum_{x,y,z \in \o/\p} \pi\l(\nb_{2}(x) \nb_{3}(y)\zb(z\v^{1-l})\r)W_1$ lies in $V(m_\pi -1)$.
Assume that $W(\ab^i_r) \neq 0$ for some $i \ge 1$.
Then $W_1(\ab^{i-1}_r) \neq 0$, and $W_1'(\ab^{i-1}_r) \neq 0$.
This contradicts to the level minimality.
Hence $W(\ab^i_r) = 0$ for all $i \ge 1$.
By Theorem \ref{thm:vnthm}, $W(\ab_r) \neq 0$ for some $r$.
Therefore, $\Xi(s,W)$ is nonzero, and so is $\Xi(s,W^c)$ by the functional equation (\ref{eqn:FEstd}).
This completes the proof.
\end{proof}
{\bf Step 3.}
To show Theorem \ref{thm:Ls1}, we need the level descending operator 
\begin{eqnarray*}
\mathcal{D}:= \int_{\Kb^c(m_\pi-1;\e)} \pi^\imath(k) dk: V^c(m_\pi) \to V^c(m_\pi-1)
\end{eqnarray*}
for $m_\pi > 2\e$, and the Hecke operators
\begin{eqnarray*}
\Tc &=& \Tc_{\Kb_1(m;\e)}(\ab_1) + \Tc_{\Kb_1(m;\e)}(\ab_{-1}), \\
\Tc^\imath &=& \Tc_{\Kb_1^c(m;\e)}(\ab_1) + \Tc_{\Kb_1^c(m;\e)}(\ab_{-1}), \\
\Sc^\imath &=& \Tc_{\Kb_1^c(m;\e)}(\eta), \\
\Tc_+^\imath &=& \Tc_{\Kb_1^c(m;\e)}(\ab_1).
\end{eqnarray*}
The first three Hecke operators are self-adjoint and diagonalizable by Lemma \ref{lem:selfad}. 
It is not hard to show that $c \circ \Tc = \Tc^\imath \circ c$.
First, compare the actions of $\Sc^\imath$ and $\mathcal{D}$.
We compute   
\begin{eqnarray*}
\Sc^\imath &=& \sum_{x,y \in \o/\p, z \in \o/\p^2} \pi^\imath\l(\nf_2(x) \nb_{3}(y \v^{-e}) \zb(z\v^{-n}) \eta^{-1}\r) \\
&&+ \sum_{x,y \in \o/\p, u \in \o/\p} \pi^\imath\l(\eta \bnb_2(x\v^{l-1}) \bnf_{3}(y \v^{m-1}) \bz(u\v^{m-1})\r), \label{eqn: S*op} \\
\mathcal{D} &=& \sum_{x,y \in \o/\p} \pi^\imath\l(\nb_2(x)\nb_3(y)\eta^{-1}\r) + \sum_{x,y \in \o/\p, u \in \o/\p} \pi^\imath\l(\bnb_2(x\v^{l-1}) \bnf_{3}(y \v^{m-1}) \bz(u\v^{m-1})\r)
\end{eqnarray*}
(c.f. Lemma 3.3.7., 6.1.2. of \cite{R-S}). 
Comparing their latter sums, we have
\begin{eqnarray*}
\lambda_{\Sc^\imath}W^c(\ab^i_r) = 
\begin{cases}
q^4 W^c(\ab^{1}_r) & \mbox{if $i= 0$,} \\
q^4 W^c(\ab^{i+1}_r) - q^2 W(\ab^i_r) & \mbox{if $i > 0$,}
\end{cases} \label{eqn: pi(S)exp0}
\end{eqnarray*}
for an eigenvector $W^c \in V^c(m_\pi)$ with eigenvalue $\lambda_{\Sc^\imath}$.
By Proposition \ref{prop:htW}, there is a nonnegative integer $r$ such that $W^c(\ab_r) \neq 0$.
By Lemma \ref{lem:havev} and the above recursion formula, $\lambda_{\Sc^\imath}$ is equal to $0$ or $-q^2$.
Assume that $\lambda_{\Sc^\imath} = 0$. 
Then, $\Xi(s,(W^c)_1) = 0$.
By Proposition \ref{prop:htW}, $\Xi(s,W_1) =0$.
By Lemma \ref{lem:NovZ}, $Z(s,W_1) = \w_\pi(\v)^{-1} q^{2s-4}\Xi(s,W)_1$.
Therefore, the functional equation for $W$, and that for $W_1$ are 
\begin{eqnarray}
q^{-\e} \Xi(1-s,W^c) &=& \ep'_\pi q^{(m_\pi - n_\pi')(s-1/2)} \Xi(s,W), \label{eqn: FED1}\\
Z(1-s,(W^c)_1) &=& \ep'_\pi  \w_\pi(\v)^{-1}q^{(m_\pi - n_\pi')(s-1/2)} \Xi(s,W)_1 \label{eqn: FED2}.
\end{eqnarray}
Since $\Xi(s,W)$ and $\Xi(s,W)_1$ are polynoimals in $q^{-s}$ with a same range, $\Xi(1-s,W^c)$ and $Z(1-s,(W^c)_1)$ have a same range.
By Lemma \ref{lem:NovZ2}, $Z(1-s,(W^c)_1)$ is a constant multiple of $q^{-s}$, since $\Xi(1-s, (W^c)_1) = 0$.
But, $\Xi(1-s, W^c)$ is a polynomial in $q^s$.
This is a contradiction. 
Hence, 
\begin{eqnarray}
\lambda_{\Sc^\imath} = -q^2, \ W^c(\ab^1_r) = -q^2W^c(\ab_r) \ \ \mbox{for all $W^c \in V^c(m_\pi)$.} \label{eqn: Wv1jvj}
\end{eqnarray}
Next, for $\Tc, \Tc^\imath$, letting $\Cs_a, \Bs_{a}$ be the lattice defined in (\ref{eqn:CrBrl}), and 
\begin{eqnarray*}
\wt{\Bs}_{a} = \Bs_{a} \op \begin{bmatrix}
 &  \\
\o & 
\end{bmatrix}, \ 
\wt{\Cs}_a = \Cs_a \op \begin{bmatrix}
 & \o \\
 &
\end{bmatrix}, 
\end{eqnarray*}
we compute 
\begin{eqnarray*}
\Tc_{\Kb_1(m;\e)}(\ab_1) = \sum_{B \in \wt{\Bs}_{l}/\v \wt{\Bs}_{l}} \pi(n(B)
 \ab_1) + 
 \sum_{B \in \Bs_{l}/\v\Bs_{l}} \pi(\jb'(1) n(B) \ab_1) 
\end{eqnarray*}
Since $\jb'(1) \in \Kb_1(m;\e)$, the latter sum equals
\begin{eqnarray*}
\sum_{B \in \Bs_{l}/\v\Bs_{l}} \pi(\jb'(1)n(B) \ab_1\jb'(1)) = \w_\pi(\v) \sum_{x,y \in \o/\p} \pi(\nf_{2}(x)\zb(y \v^{-l})\ab_{-1}^1).
\end{eqnarray*}
Similarly, 
\begin{eqnarray*}
\Tc_{\Kb_1(m;\e)}\l(\ab_{-1}\r) &=& \sum_{C \in \v^l \Cs_\e/\v^{l+1}\Cs_{\e}} \pi(\jb'(1)\bar{n}(C) \ab_{-1}\jb'(1)) + \sum_{C \in  \v^l\wt{\Cs}_\e/\v^{l+1}\wt{\Cs}_{\e}} \pi(\bar{n}(C) \ab_{-1}) \\
&=& \w_\pi(\v)^{-1} \sum_{x,y \in \o/\p} \pi(\ab_1^{-1}\bnb_{2}(x \v^{l-1})  \bar{\zb}(y\v^{m-1})) \\
&& + \sum_{C \in \v^l \Cs_\e/\v^{l+1}\Cs_{\e}} \pi(\bar{n}(C) \ab_{-1}) + 
\sum_{z\in (\o/\p)^\t} \l(\sum_{C \in \v^l \Cs_\e/\v^{l+1}\Cs_{\e}} \pi(\bnb'(z)\bar{n}(C)\ab_{-1})\r).
\end{eqnarray*}
Using (\ref{eqn:usid}) and the $\Kb_1(m;\e)$-invariance property, the sum in the bracket is transformed to 
\begin{eqnarray*}
&& \sum_{C \in \v^l \Cs_\e/\v^{l+1}\Cs_{\e}} \pi(\nf'(z^{-1}) \jb'(z) \nf'(z^{-1})\bar{n}(C)\ab_{-1}) \\
&=& \sum_{x,y\in \o/\p} \pi(\nb'(z^{-1}) \jb'(z) \nb'(z^{-1})\ab_{-1} \bar{\nf}_{3}(x \v^{l-1}) \bar{\zb}(y\v^{m-1})) \\
&=& 
 \sum_{x,y\in \o/\p} \pi(\nb'(z^{-1}) \jb'(z) \ab_{-1} \bnb_{3}(x \v^{l-1})\bar{\zb}(y\v^{m-1}) \bnb_{2}(z^{-1}y \v^{m})\bar{\zb}(-z^{-1} y^2 \v^{2m-1})).
\end{eqnarray*}
Since $\Kb_1(m;\e)$ contains $\bar{\nf}_{2}(z^{-1}y \v^{m}), \bar{\zb}(-z^{-1} y^2 \v^{2m-1})$, this sum equals 
\begin{eqnarray*}
\sum_{C \in \v^{l-1}\Cs_{\e}/\v^l\Cs_{\e}} \pi(\nb'(z^{-1}) \jb'(z) \ab_{-1} \bar{n}(C)) 
&=& 
\sum_{C \in \v^{l-1}\Cs_{\e}/\v^l\Cs_{\e}} \pi(\nb'(z^{-1}) \jb'(z) \ab_{-1}\bar{n}(C)\jb'(z)) \\
&=& 
\w_\pi(\v)^{-1}\sum_{x,y\in \o/\p} \pi(\nb'(z^{-1})\ab_1^{-1}\bar{\nf}_{2}(x \v^{l-1})\bar{\zb}(y\v^{m-1})).
\end{eqnarray*}
Therefore, for $W \in V(m)$, 
\begin{eqnarray}
\Tc W\l(\ab^i_r\r) &=& q^3 W\l(\ab^i_{r+1}\r) + q^2 \w_\pi(\v) W\l(\ab^{i+1}_{r-1}\r) + \sum_{C \in  \v^{l-1}\Cs_{\e}/\v^l\Cs_{\e}} W\l(\ab^i_{r-1}\bar{n}(C)\r) \nonumber \\
&+& q \w_\pi(\v)^{-1} \sum_{x,y \in \o/\p} W\l(\ab_{r+1}^{i-1} \bar{\nf}_{2}(x \v^{l-1}) \bar{\zb}(y\v^{m-1})\r),  \ \ i,r \ge 0. \label{eqn: Heckeexp}
\end{eqnarray}
By a similar computation, for $W^c \in V^c(m)$,
\begin{eqnarray}
\Tc^\imath W^c\l(\ab^i_r\r) &=& q^3 W^c\l(\ab^i_{r+1}\r) + q^3 \sum_{z \in \o/\p} W^c\l(\ab^{i+1}_{r-1}\bnb'(z \v^{\e-1})\r) \nonumber \\
&+& \sum_{C \in \v^{m-1}\Bs_{l}} W^c\l(\ab^i_{r-1}\bar{n}(C)\r) + q \sum_{x \in \o/\p} W^c\l(\ab_{r+1}^{i-1}\bnb_{2}(x \v^{l-1})\r),\ \  i,r \ge 0.  \label{eqn: Heckeexp2}
\end{eqnarray} 
Observing the first gauge of $W^c$ ($K_1(\e)' (\subset P_3)$-invariant!), one can find by Lemma \ref{lem:e-lem} that the second term equals $q^3W^c(\ab^{i+1}_{r-1})$ if $r \ge 1$.
We choose the Haar measures in (\ref{eqn: Gamma}) so that the third terms of (\ref{eqn: Heckeexp}) and (\ref{eqn: Heckeexp2}) are equal to the values at $\ab^i_{r-1}$ of $\Gamma_{l-1}W$ and $\Gamma_{m-1}^\imath W^c$, respectively.
\begin{lem}\label{lem:bnb2324}
Let $\pi \in \Ir^{gn}(\G)$ be unitary.
Assume that $L(s,\pi) = 1$. 
Let $W \in V(m)$ such that $\Xi(s,W) \neq 0$.
Then, for the Haar measures as above, we have the following identities.
\begin{eqnarray*}
\Gamma_{l-1}W\l(\ab_{r}\r) &=& 
\begin{cases}
0 & \mbox{if $r \ge r_0$} \\
qW(\ab_{r}) & \mbox{otherwise}.
\end{cases}, \\
\Gamma_{m-1}^\imath W^c(\ab_{r}) &=& qW^c(\ab_{r}),
\end{eqnarray*}
where $l = m-\e$, and $r_0$ is the maximal integer such that $W(\ab_{r_0}) \neq 0$ (such $r_0$ exists by Lemma \ref{lem:havev}).
\end{lem}
\begin{proof}
Since the arguments are similar, we only prove the first identity.
We observe the both sides of the functional equation (\ref{eqn:FEstd}) and 
\begin{eqnarray}
Z(s,\pi^\imath(\j_{m})(\Gamma_{l-1}W)^\imath) &=& q^{(m- n_\pi')(s-1/2)}Z(s,\Gamma_{l-1}W). \label{eqn:Gammafeq}
\end{eqnarray}
By Lemma \ref{lem:Glop}, Proposition \ref{prop: cricase}, $\Gamma_{l-1}W$ is $0$-balanced, and $\pi^\imath(\j_{m})(\Gamma_{l-1}W)^\imath$ is $(\e+1)$-balanced.
By Proposition \ref{prop:zeta-simp}, 
\begin{eqnarray*}
Z(s,\Gamma_{l-1}W) = \Xi(s,\Gamma_{l-1}W), \ Z(s,\pi^\imath(\j_{m})(\Gamma_{l-1}W)^\imath) = q^{- \e-1}\Xi(s,\pi^\imath(\j_{m})(\Gamma_{l-1}W)^\imath).
\end{eqnarray*} 
Since
\begin{eqnarray*}
\Xi(s,\pi^\imath(\j_{m})(\Gamma_{l-1}W)^\imath) &=& \Xi\l(s,\sum_{y,z \in \o/\p} \pi(\nf_3(y \v^{- \e-1}) \nf'(z\v^{-1}))\pi(\j_{m})W^c\r) \\
&=& q^2 \l(\Xi(s,W^c) - W^c(1)\r),
\end{eqnarray*}
we have 
\begin{eqnarray*}
q^{1-\e}\l(\Xi(1-s,W^c) - W^c(1)\r) = q^{(m- n_\pi')(s-1/2)}\Xi(s,\Gamma_{l-1}W)
\end{eqnarray*}
by (\ref{eqn:Gammafeq}).
By (\ref{eqn:FEstd}), we have 
\begin{eqnarray*}
-q^{1-\e}W^c(1) = q^{(m- n_\pi')(s-1/2)}\Xi(s,\Gamma_{l-1}W-qW),
\end{eqnarray*}
from which the identity follows.
\end{proof}
By this Lemma and (\ref{eqn: Wv1jvj}), the second term and third term of (\ref{eqn: Heckeexp2}) cancel if $i =0$ and $r \ge 1$.
The last terms of (\ref{eqn: Heckeexp}) and (\ref{eqn: Heckeexp2}) vanish if $i = 0$, by the following lemma.
\begin{lem}\label{lem:Heckevnsh}
Let $\pi \in \Ir^{gn}(\G)$.
Let $\e = \c(\w_\pi)$.
If $m > 2\e$, and $W \in V(m)$, then, for $x,y \in \o$,
\begin{eqnarray*}
\pi\l(\ab_{r} \bnb_{2}(x \v^{l-1}) \bar{\zb}(y\v^{m-1})\r)W(\eta) = \pi^\imath\l(\ab_{r} \bnb_2(x \v^{l-1})\r)W^c(\eta) = 0.
\end{eqnarray*}
\end{lem}
\begin{proof}
By Lemma \ref{lem:vanishlem}, it suffices to see that the second gauges of the above Whittaker functions are $N(\o)$-invariant.
For the latter Whittaker function, the $N(\o)$-invariance property follows from that of the second gauge of $\pi^\imath(\ab_r)W^c$, and the identity (\ref{eqn: nion}).
For the former one, use the $\jb'(1)$-conjugation of the identity (\ref{eqn: nion}) with
\begin{eqnarray*}
B \in \begin{bmatrix}
\o & \\
 & \o
\end{bmatrix}, C \in \begin{bmatrix}
\p^{l-1} & \\
\p^{m-1} & \p^{l-1}
\end{bmatrix}
\end{eqnarray*}
and the calculation in the proof of Lemma \ref{lem:Glop}.
\end{proof}
Now suppose that $W (\neq 0) \in V(m_\pi)$ is an eigenvector of $\Tc$ with eigenvalue $\lambda = \lambda_{\Tc}$.
Then $W^c \in V^c(m_\pi)$ is also an eigenvector of $\Tc^\imath$ with eigenvalue $\lambda$ since $c \circ \Tc = \Tc^\imath \circ c$.
By the above argument, we have 
\begin{eqnarray}
\lambda W(\ab_{r}) &=& q^3 W(\ab_{r+1}) + q W(\ab_{r-1}), \ \ r \le r_0, \label{eqn:Texpfin} \\
\lambda W^c(\ab_{r}) &=& q^3 W^c(\ab_{r+1}), \ \ r \ge 1, \label{eqn: T^*2}
\end{eqnarray}
where $r_0$ is as in Lemma \ref{lem:bnb2324}.
By Lemma \ref{lem:vanishlem}, $W(\ab_{-1}) = 0$.
If we assume that $W(1) = 0$, then, by (\ref{eqn:Texpfin}), $\Xi(s,W) = 0$, which contradicts to Proposition \ref{prop:htW}.
Hence,
\begin{eqnarray}
W(1) \neq 0. \label{eqn:W1not0}
\end{eqnarray}
Next, assume that $\lambda \neq 0$.
Then, we conclude $W^c(\ab_1) = 0$ by considering (\ref{eqn: T^*2}) and Lemma \ref{lem:havev}.
By (\ref{eqn: T^*2}) again, $\Xi(1-s,W^c)$ is a constant.
By the functional equation (\ref{eqn:FEstd}), $\Xi(s,W)$ is a monomial.
Since $W(1) \neq 0$, $\Xi(s,W)$ is a constant.
In particular, $W(\ab_1) = 0$.
By (\ref{eqn:Texpfin}), $\lambda = 0$.
This is a contradiction.
Hence 
\begin{eqnarray*}
\lambda = 0.
\end{eqnarray*}
By (\ref{eqn:Texpfin}) again, $\Xi(s,W)$ is a constant.
Therefore, $W(\ab^i_r) = 0$ for $(i,r) \neq (0,0)$ by Proposition \ref{prop:htW}.
Since $\Tc$ is diagonalizable, by Theorem \ref{thm:vnthm}, 
\begin{eqnarray*}
\dim V(m_\pi) = 1.
\end{eqnarray*}
Finally, we consider the action $\Tc_+^\imath$.
Let $W^c (\neq 0) \in V^c(m_\pi)$.
Since $\dim V^c(m_\pi) = 1$, $W^c$ is an eigenvector of $\Tc_+^\imath$. 
Let $\lambda_{+}$ be the eigenvalue.
We compute
\begin{eqnarray*}
\Tc_+^\imath W^c(\ab_r^i) = q^3 W^c\l(\ab^i_{r+1}\r) + q \sum_{x \in \o/\p} W^c\l(\ab_{r+1}^{i-1}\bnb_{2}(x \v^{l-1})\r),\ \  i,r \ge 0.
\end{eqnarray*}
By Lemma \ref{lem:Heckevnsh}, $\lambda_{+} W^c(\ab_r) = q^3 W^c\l(\ab_{r+1}\r)$.
By Proposition \ref{prop:htW}, $W^c(\ab_r) \neq 0$ for some $r$.
By Lemma \ref{lem:havev},   
\begin{eqnarray*}
\lambda_{+} = 0, \ \ W^c(1) \neq 0,\ \ W^c(\ab_r) = 0, \ r \ge 1.
\end{eqnarray*}
Now, Theorem \ref{thm:Ls1} for the unitary case follows from Theorem \ref{thm: coinLN}.
The values of $W^c$ on $\Tb$ are determined by the recursion formula (\ref{eqn: Wv1jvj}).

{\bf Step 4.}
Finally, we discuss for the non-unitary case.
For a supercuspidal representation, twisting it by a character $\nu_a, a \in \R$, we obtain a unitary supercuspidal representation (\cite{C2}), where $a$ is unique and called the exponent of the representation.
Applying the above argument to the twist, one can show the theorem for this case.
For a generic constituent $\pi$ of the Klingen induction $\chi \rtimes \s$ from supercuspidal, we apply the argument in sect. 5 of \cite{R-S}.
Consider the following facts (c.f. Table A.3., A.4. of \cite{R-S}, p. 93-94 \cite{S-T}): 
\begin{itemize}
\item The Jacquet module of $\chi \rtimes \s$ with respect to the unipotent radical $\Ub_\mathbf{P}$ of the Siegel parabolic subgroup $\mathbf{P}$ vanishes (see p. 29 of \cite{R-S} for the definition of $\mathbf{P}$).
\item The semisimplification of the Jacquet module of $\chi \rtimes \s$ with respect to the unipotent radical $\Ub_\Qb$ of $\Qb$ is $\chi \t \s + \chi^{-1} \t \chi \s$.
\end{itemize}
Since $\mathrm{pr}(\Ub_\mathbf{P}) = N' N_3$, $\tau_1$-type does not appear in the Jordan-H\"older sequence in Theorem \ref{thm: RSfilt}.
Since $\mathrm{pr}(\Ub_\Qb) = N_2N_3$ and $\s$ is ramified, $\tau_2$-types in the sequence have no $P_3(\o) (= \mathrm{pr}(\Kb_1(m;\e)))$-invariant vector.
Hence, non-generic constituents of $\chi \rtimes \s$ have no quasi-paramodular vector.
Since the generic constituent is a unique constituent, $V(m_\pi)$ and the subspace of $\Kb_1(m_\pi,\e)$-invariant vectors $f$ in $\chi \rtimes \s$ have the same dimension.
Let $R = \{r\}$ be representatives for $\Qb \bs \G /\Kb_1(m_\pi;\e)$.
Since $f$ is determined by the values $f(r)$, we should have $\chi(t) \s(g) f(r) = f(r)$ for all $r \in R$ and $k \in \Qb \cap \Int(r)\Kb_1(m_\pi; \e)$, where we write $k = q \in \Qb$ of the form of (\ref{eqn:Kldef}).
Then $\det(g) = \mu(k)$ lies in $\o^\t$.
Since any power of $k$ lies in the compact subgroup $\Int(r)\Kb_1(m_\pi;\e)$, $t$ lies in $\o^\t$.
Let $a,b$ be the exponents of $\chi, \s$, respectively. 
The generic constituent of $\nu_{-a}\chi \rtimes \nu_{-b}\s$ is unitary (c.f. Table A.1. of \cite{R-S}).
Denote it by $\pi_1$.
We have showed that $\dim V(m_{\pi_1}) =1$.
Since $\det(g), t \in \o^\t$, the above condition on $f \in \nu_{-a}\chi \rtimes \nu_{-b}\s$ is same for various $a,b$.
Hence $\dim V(m_{\pi})= \dim V(m_{\pi_1}) = 1$.
Now the above argument of Hecke operators for unitary representations works, and therefore, Theorem \ref{thm:Ls1} is true also for non unitary generic constituents.
This completes the proof.

For $W \in V(m)$, define
\begin{eqnarray}
W^- = \pi(\jb_{m -\e}'') W, \ \ W^{-c} = \pi^\imath(\j_{m}) (W^-)^\imath = \pi^\imath(\jb'(\v^\e))W^c. \label{eqn:defW^-}
\end{eqnarray}
Since $W_\pi^-$ is $0$-balanced, by Propositon \ref{prop:zeta-simp}, the functional equation is 
\begin{eqnarray}
\frac{Z(1-s,W_\pi^{-c})}{L(1-s,\pi^\imath)} = \ep(s,\pi,\psi)\frac{\Xi(s,W_\pi^-)}{L(s,\pi)}. \label{eqn:FE-}
\end{eqnarray}
\begin{Cor}\label{cor:L1}
For $\pi \in \Ir^{gn}(\G)$ as in the theorem, 
\begin{eqnarray*}
\ep_\pi'^{-1} Z(s,W_\pi^{-c})  = \Xi(s,W_\pi^-) =W_\pi^-(1) = \Gb(\w,1).
\end{eqnarray*}
\end{Cor}
\begin{proof}
Let $f$ be the first gauge of $W_\pi^{-c}$.
We have $Z(s,W_\pi^{-c}) = \int_{F^\t} \wt{f}(a(t)) \nu_{s-3/2}(t) d^\t t$ by Proposition \ref{prop:zeta-simp},  where $\wt{f}(p)= \int_{\bar{N}'(\p^\e)}f(p n) dn$.
By using the identity (\ref{eqn: nion}) and the invariance property of $f$ under
\begin{eqnarray*}
\begin{bmatrix}
\o^\t & \p^{-\e}\\
\p^{2\e} & 1 + \p^\e
\end{bmatrix}' \subset P_3,
\end{eqnarray*}
one can show that $\wt{f}$ is $N'(\o)$-invariant.
Therefore, $Z(s,W_\pi^{-c}) \in \C\jump{q^{-s}}$ by Lemma \ref{lem:vanishlem}.
Now, we compute $\wt{f}(1) = \int_{\bar{N}'(\p^\e)} f(n) dn$ which is the constant term of $Z(s,W_\pi^{-c})$.
By using Lemma \ref{lem:vanishlem} and the identity (\ref{eqn: nion}) again, one can show that $f(\bar{n}'(x)) = 0$ for $x \in \p^{\e+1}$. 
Therefore, 
\begin{eqnarray*}
\int_{\p^\e} f(\bar{n}'(x)) dx &=& \int_{\p^{*\e}} f(\bar{n}'(x)) dx \\
&=& q^{-\e}\int_{\o^\t} f(n'(\v^{-\e}u^{-1})j(\v^\e u)'n'(\v^{-\e}u^{-1})) du \\
&=& q^{-\e}\int_{\o^\t} \psi(\v^{-\e}u^{-1})\w_\pi(u)f(j(\v^\e)') du \\
&=& q^{-\e}\Gb(\w,1)f(j(\v^\e)'),
\end{eqnarray*}
where (\ref{eqn:usid}) and the $N'(\p^{-\e})$-invariance property of $f$ are used.
Now the assertion follows from (\ref{eqn:FE-}) and the identities $f(j(\v^\e)')= W_\pi^{-c}(\jb'(\v^\e))= W_\pi^{c}(1) = q^\e \ep'_\pi$.
\end{proof}
Now, let $(\pi,V) \in \Ir(\G)$ be tempered, non-generic.
Then $\pi$ is the representation of VIb or VIIIb listed in the Table A. 1. of \cite{R-S}.
By the proof of Theorem 2.5.3., and Table A.6., A.7. of loc. cit., $V_{\Zb^J}$ is irreducible and a $\tau_2$-type.
But, any $\tau_2$-type does not have a $\mathrm{pr}(\Kb_1(m;\e))$-invariant vector by the above argument for the case of $\e > 0$, and by Lemma 3.4.4 of loc. cit. for the case of $\e = 0$.
Hence, we have: 
\begin{thm}\label{thm:temp}
A tempered $\pi \in \Ir(\G)$ has a quasi-paramoular vector, if and only if $\pi$ is generic.
\end{thm}
\section{Construction of quasi-paramodular forms}\label{sec:CQF}
In this section, by local $\th$-lift from $GSO(2,2)$ to $\G$, we show the existence of the newform (c.f. Theorem \ref{thm: coinLN}) for generic constituents of Borel and Siegel parabolic inductions, respectively.
The proof of the main theorem will be complete, except for the coincidences of root numbers and conductors.
Let $X = M_{2 \t 2}(F)$, equipped with the nondegenerate symmetric split form $\frac{1}{2}Tr (x_1^* x_2)$, where $x^*$ indicates the main involution of $x \in X$.
Let $GO_X$ denote the generalized orthogonal group of $X$ and $\mu_X$ the similitude factor.
Let $H = GSO_X := \ker(\mu_X^{-2} \det) \subset GO_X$.
Letting $G_2 \t G_2$ act on $X$ by $(g_1,g_2) \cdot x = g_1 x g_2^{*}$, we have the isomorphisms 
\begin{eqnarray*}
H \simeq G_2 \t G_2/\{(z, z^{-1}) \mid z \in F^\t \}, 
SO_X \simeq \{(g_1,g_2) \mid \det(g_1g_2) = 1\}/\{(z, z^{-1}) \mid z \in F^\t \}.
\end{eqnarray*}
Via these isomorphisms, we will represent elements and subgroups of $H$ by those of $G_2 \t G_2$, and objects in $\Ir(H)$ by those in $\{\tau_1 \bt \tau_2 \mid \w_{\tau_1} = \w_{\tau_2}\} \subset \Ir(G_2) \bt \Ir(G_2)$.
Let $B, T$ denote the upper triangular and diagonal matrices in $G_2$ respectively.
Let $N_X = N \t N \subset H$ and $B_X,T_X$, similarly.
Define $\psi_X \in \Xs(N_X)$ by $\psi_X((n,m)) = \psi(n m^{-1})$.
We say $\tau = \tau_1 \bt \tau_2 \in \Ir(H)$ is generic, if $\Hom_{N_X}(\tau,\psi_X) \neq \{0\}$, or equivalently if both $\tau_1, \tau_2 \in \Ir(G_2)$ are generic.
Let $Y$ denote the $4$-dimensional space equipped with the symplectic form defined by the matrix (\ref{eqn:defmat}).
Set $Z = X \ot Y$.
Let $Y = Y^+ + Y^-, Z^\pm = X \ot Y^\pm = X \op X$ be the polarizations.
For $z = x_1 \op x_2, z' = x_1' \op x_2' \in Z^+$, we write $(z,z') = (\frac{1}{2}Tr(x_i^* x_j'))$.
For $\vp \in \Ss(Z^+)$, let $\vp^\sharp$ denote the Fourier transform defined by $\vp^\sharp(z') = \int_{X \op X}\psi(Tr(z' ,z)) \vp(z) d z$ where $d z$ is chosen so that $(\vp^{\sharp})^\sharp(z) =\vp(-z)$.
The Weil representation $w_\psi$ of the dual pair $Sp(4) \t O_X$ can be realized on the Schwartz space $\Ss(Z^+)$.
The action is given by the following formulas: 
\begin{eqnarray*}
&&w_\psi(1,h) \vp(z) = \vp(h^{-1}\cdot z), \ \ h \in O_X, \\
&&w_\psi(A^\natural,1) \vp(z) = |\det(A)|^{-2}\vp(zw_2{}^tA^{-1}w_2), \ A \in G_2 \\
&&w_\psi(n(B),1) \vp(z) = \psi(\frac{1}{2}Tr (B w_2(z,z))) \vp(z), \\
&&w_\psi(j(-w_2),1) \vp(z) = \vp^\sharp(z).
\end{eqnarray*}
Let $R= G \t H$, and $R_0= \ker(\mu^{-1}\mu_X) \subset R$.
For our convenience of the computation below, we adopt the following extension $w_\psi$ to $R_0$ as in \cite{R2}
\begin{eqnarray*}
w_\psi(g,h)\vp(z) = |\mu_X(h)|^{-2} w_\psi(g_1,1)\vp(h^{-1} \cdot z), 
\end{eqnarray*}
where 
\begin{eqnarray*}
g_1 = g\begin{bmatrix}
1_2 & \\
 & \mu(g)^{-1}1_2
\end{bmatrix}.
\end{eqnarray*}
Note that this differs from the normalization used in \cite{G-T}.
Observe that the central elements $(u,u) \in R_0$ act on $\Ss(Z^+)$ trivially.
Let $\Om = \mathrm{ind}_{R_0}^R w_\psi$ be the compact induction, which can be realized on the Schwartz space $\Ss(Z^+ \t F^\t)$ (c.f. \cite{R}, \cite{S}).
For $\tau^1 \in \Ir(SO_X)$ define $w_\psi(\tau^1) = w_\psi/ \displaystyle \cap_{\lambda \in \Hom_{SO_X}(w_\psi,\tau^1)} \ker(\lambda)$, and for $\tau \in \Ir(H)$ define $\Om(\tau)$ similarly.
By Lemme 2. III. 4. of \cite{MVW}, there exist $\Th_\psi(\tau^1) \in \Alg(Sp(4,F))$ and $\Th(\tau) \in \Alg(\G)$, such that 
\begin{eqnarray}
w_\psi(\tau^1) \simeq \Th_\psi(\tau^1) \ot \tau^1, \ \ 
\Om(\tau) \simeq \Th(\tau) \ot \tau. \label{eqn: defThtau}
\end{eqnarray}
It is known that $\Th_\psi(\tau^1)$ and $\Th(\tau)$ are admissible of finite length.
The maximal semi-simple quotients of $\Th_\psi(\tau)$ and $\Th(\tau)$ are denoted by $\th_\psi(\tau)$ and $\th(\tau)$ respectively.
Let $\Om_{\N,\psi}$ be the $\psi$-twisted $\N$-Jacquet module of $\Om$.
By the Frobenius reciprocity,
\begin{eqnarray*}
\Hom_{R}(\Om, \mathrm{Ind}_\N^\G \psi \ot \tau) &\simeq& \Hom_{\N \t H}(\Om|_{\N \t H}, \psi \ot \tau) \\
&\simeq& \Hom_{\N \t H}(\Om_{\N,\psi}, \psi \ot \tau).
\end{eqnarray*}
As in the proof of Proposition 2.4 and its Corollary in \cite{G-R-S}, one can prove 
\begin{eqnarray*}
\Om_{\N,\psi} \simeq \psi \ot \mathrm{ind}_{N_X}^H \psi_X 
\end{eqnarray*}
as $\N \t H$-modules (c.f. Proposition 4.1. of \cite{M-S}).
Therefore, 
\begin{eqnarray*}
\Hom_{\N \t H}(\Om_{\N,\psi}, \psi \ot \tau) &\simeq& \Hom_{H}(\mathrm{ind}_{N_X}^H \psi_X,\tau) \\
&\simeq& \Hom_{H}(\tau,\mathrm{Ind}_{N_X}^H \psi_X^{-1}).
\end{eqnarray*}
Now suppose that $\tau$ is generic.
Then, $\dim \Hom_{R}(\Om, \mathrm{Ind}_\N^G \psi \ot \tau) = \dim \Hom_{H}(\tau,\mathrm{Ind}_{N_X}^H \psi_X^{-1})= 1$.
Since the Jacquet module $(\tau \ot \tau^\vee)_H$ is isomorphic to $\C$, 
\begin{eqnarray}
\dim \Hom_{\G}(\Th(\tau), \mathrm{Ind}_\N^\G \psi) = 1 \label{eqn:mult1thtau}
\end{eqnarray}
by (\ref{eqn: defThtau}).
Hence, $\Th(\tau)$ has a generic irreducible constituent.
By the work of W. T. Gan and S. Takeda \cite{G-T}, this constituent is $\th(\tau)$.
The (unique) generic constituent of the parabolic induction $\chi_1 \t \chi_2 \rtimes \chi$ (resp. $\rho \rtimes \chi$) (c.f. Table A.1. of \cite{R-S}) for $\chi_1,\chi_2, \chi \in \Xs(F^\t)$ coincides with the small $\th$-lift of $\pi(\chi,\chi\chi_1\chi_2)^\vee \bt \pi(\chi\chi_1,\chi\chi_2)^\vee$ (resp. $\pi(\chi,\chi\w_\rho)^\vee \bt \chi^{-1} \rho^\vee$) (c.f. \cite{G-T}.), where $\rho \in \Ir(G_2)$, $\chi,\chi_i \in \Xs(F^\t)$.
Here, $\pi(\a,\b)$ for $\a, \b \in \Xs(F^\t)$ indicates the principal series induced from the representation of $B \subset G_2$ sending $b \to \nu_{1/2}(b_{11}/b_{22}) \a(b_{11})\b(b_{22})$.
Since all $L$-functions $L(s,\pi)$ for $\pi \in \Ir^{gn}(\G)$ are computed by R. Takloo-Bighash \cite{T}, and it is known by \cite{G-T} when the $\th$-lift is a constituent of a parabolic induction, one can show that
\begin{eqnarray}
L(s,\th(\tau_1^\vee \bt \tau_2^\vee)) = L(s,\tau_1)L(s,\tau_2) \label{eqn:Lpitau}
\end{eqnarray}
by case-by-case argument.
Let
\begin{eqnarray*}
z_ 0 = e_0 \op 1_2 \in Z^+, \ \ e_0 = \begin{bmatrix}
 & 1 \\
 &
\end{bmatrix} \in X.
\end{eqnarray*}
The stabilizer subgroup of $z_0$ by $SO_X$ is $N_{\Delta} := \{(n, n) \mid n \in N \}$.
Let $\xi_1 \in \W_\psi(\tau_1), \xi_2 \in \W_{\psi^{-1}}(\tau_2)$ and $\xi = \xi_1 \bt \xi_2$.
Let $\vp \in \Ss(X \op X)$.
We choose the Haar measure $d h$ on $H$ (resp.  $dn$ on $N_\Delta$) such that $\vol(G_2(\o) \t G_2(\o))$ (resp. $\vol(N_\Delta(\o))$) equals $1$.
Let $d\dot{h} = dh /dn$ denote the Haar measure on $N_\Delta  \bs H$.
Consider the function $\xi_\vp$ on $\G$ defined by  
\begin{eqnarray}
\xi_\vp(g) = \int_{N_{\Delta} \bs SO_X} w_{\psi}(g, h h_g)\vp(z_0)\xi(hh_g) d \dot{h}, \label{eqn:locWh}
\end{eqnarray}
where $h_g \in H$ is chosen so that $\mu(g) = \mu_X(h_g)$.
This integral is independent from the choice of $h_g$, and converges since the function $h \to \vp(h^{-1}\cdot z_0)$ has a compact support modulo $N_{\Delta}$.
By using the above formulas of $w_\psi$, one can see that 
\begin{eqnarray*}
w_{\psi}\l(\nf'(b)\nf_3(*)\zb(*),1\r)\vp(z_0) &=& \psi(b)\vp(z_0), \\
w_\psi(\nf_2(b),h)\vp(z_0) &=& w_\psi(1,h)\vp(z_0 - \ep_2 \ot b e_0)) \\
&=& w_\psi(1,(n(-b),1)h)\vp(z_0),
\end{eqnarray*}
and that $\xi_\vp$ is a Whittaker function with respect to $\psi$.
Now, let $\pi$ denote the (generic) $\G$-module generated by these $\xi_\vp$.
We will show that there is a $\G$-surjection
\begin{eqnarray}
\Th(\tau^\vee) \to \pi. \label{eqn:Thsurjpi}
\end{eqnarray}
Since the central elements $(u,u) \in R_0$ act on $\Ss(Z^+)$ trivially, $\xi_\vp$ and $\tau$ have the same central character.
Write $\w = \w_\tau = \w_\pi$.
By $\w$ and Lemma 2.9 of \cite{B-Z}, there is an irreducible $SO_X$-submodule $\tau_0$ of $\tau$ and finite subset $h_1,\ldots, h_r$ of representatives for $H/ F^\t SO_X$ such that $\tau|_{SO_X} = \op_{i=1}^r \tau_i$ where $\tau_i$ denotes the $h_i$-translation of $\tau_0$.
For $0 \le i \le r$, let $\pi_i$ denote the $Sp_4$-module generated by $\xi_\vp$ for $\xi \in \tau_i$.
Let $g_i \in \G$ such that $\mu(g_i) = \mu_X(h_i)$.
By definition, $\pi = \op_{i=1}^r \pi_i$ as $Sp_4$-modules, where $\pi_i$ denotes the $g_i$-translation of $\pi_0$.
Let $\lambda_i \in \Hom_{Sp_4 \t SO_X}(w_\psi, \Hom_\C(\tau_i,\pi_i))$ denote the mapping $\vp \mapsto (\xi \mapsto \xi_\vp)$.
Since $\Hom_\C(\tau_i,\pi_i)^K \simeq (\tau_i^K)^* \ot \pi_i \simeq (\tau_i^*)^K \ot \pi_i$ for any open subgroup $K \subset SO_X$, $\lambda_i$ factors through a $\lambda_i' \in \Hom_{Sp_4 \t SO_X}(w_\psi, \tau_i^\vee \ot \pi_i)$.
By (\ref{eqn: defThtau}), we have an $Sp_4$-homomorphism $\Th_\psi(\tau_i^\vee) \to \pi_i$, which is surjective by construction.
Therefore, $\pi_i$ is admissible, and so is $\pi$.
Let $\lambda \in \Hom_{R_0}(w_\psi, \Hom_\C(\tau,\pi))$ denote the mapping $\vp \mapsto (\xi \mapsto \xi_\vp)$.
Then, $\lambda$ also factors through a $\lambda' \in \Hom_{R_0}(w_\psi, \tau^\vee \ot \pi)$.
Since $\tau^\vee \ot \pi$ is $R$-admissible, by Proposition 2.15. of \cite{B-Z} and Lemma \ref{lem:Gref} i) below, $((\tau^\vee \ot \pi)^\vee|_{R_0})^\vee \simeq \tau^\vee \ot \pi$.
By the Frobenius reciprocity, 
\begin{eqnarray*}
\Hom_{R_0}(w_\psi, \tau^\vee \ot \pi)) 
&\simeq& \Hom_{R_0}(w_\psi, (\tau^\vee \ot \pi)^\vee|_{R_0})^\vee) \\
&\simeq& \Hom_{R}(\Om, \tau^\vee \ot \pi).
\end{eqnarray*}
For any $\xi_\vp \in \pi$, there exists a $\xi^* \in \tau^\vee$ such that $\lambda(\vp) = \xi^* \ot \xi_{\vp}$ by construction.
Let $\wt{\lambda} \in \Hom_{R}(\Om, \tau^\vee \ot \pi)$ correspond to $\lambda$.
By Lemma \ref{lem:Gref} ii), $\xi^* \ot \xi_{\vp} \in \mathrm{Im}(\wt{\lambda})$.
Since $(\tau \ot \tau^\vee)_H \simeq \C$, $\wt{\lambda}$ induces a surjection (\ref{eqn:Thsurjpi}). 

For an $l$-group, let $\Delta_G$ denote the modulus of $G$.
\begin{lem}\label{lem:Gref}
Let $G$ be an $l$-group, and $G_0$ a closed subgroup of $G$.
Let $(\pi,V) \in \Alg(G)$.
Assume that $G$ has a system of neighbourhoods $N = \{K \}$ of the unit consisting of open compact subgroups such that $V^K = V^{K \cap G_0}$. 
Then 
\begin{enumerate}[i)]
\item $(\pi|_{G_0})^\vee = \pi^\vee$. 
\item Let $\rho \in \Alg(G_0)$ and $\lambda \in \Hom_H((\Delta_{G_0}/\Delta_G) \rho,(\pi|_{G_0})^\vee)$.
Let $\wt{\lambda} \in \Hom_G(\mathrm{ind}_{G_0}^G \rho,\pi^\vee)$ induced by the Frobenius reciprocity.
Then $\mathrm{Im}(\lambda) \subset \mathrm{Im}(\wt{\lambda})$.
\end{enumerate}
If $G_0 \lhd G$, then for any $G_0$-admissible $(\pi,V) \in \Alg(G)$, there is a system of neighbourhoods as above. 
\end{lem}
\begin{proof}
i) Let $V^*$ denote the dual of $V$.
$\pi|_{G_0}$ and $\pi$ have the same dual $V^*$.
By Lemma 2.14 of \cite{B-Z}, $(V^*)^{K \cap G_0} = (V^{K \cap G_0})^* = (V^K)^* = (V^*)^K$ for any $K \in N$.
Therefore, $(\pi|_{G_0})^\vee = \cup_{K \in N} (V^*)^{K \cap G_0} = \cup_{K \in N} (V^*)^{K} = \pi^\vee$.
\\
ii)
For $\xi \in \Delta \rho$, take a $K \in N$ so that $\xi$ is $K \cap G_0$-invariant.
Then, $\lambda(\xi) \in (V^*)^{K \cap G_0} = (V^{K \cap G_0})^* = (V^K)^*$. 
Let $\la, \ra$ denote the natural pairing for $V,V^*$.
By 2.29 of loc. cit., $\wt{\lambda}$ is given by 
 \begin{eqnarray*}
\la \wt{\lambda}(f), v \ra = \int_{G_0 \bs G} \la \lambda(f(g)), \pi(g)v \ra d g, \ \ v \in V, f \in \mathrm{ind}_{G_0}^G \rho.
\end{eqnarray*}
Since $\xi$ is $K \cap G_0$-invariant, we can define $f_K \in \mathrm{ind}_{G_0}^G \rho$ by $f_K(hk) = \Delta_{G_0}/ \Delta_{G}(h)\xi(h)$ for $h \in G_0, k \in K$.
By definition, $f_K$ is $K$-invariant, and therefore $\wt{\lambda}(f_K)$ lies in $(V^*)^K = (V^K)^*$.
For $v \in V^K$, 
\begin{eqnarray*}
\la \wt{\lambda}(f_K), v \ra &=& \int_{G_0 \bs G_0K} \la \lambda(f_K(g)), \pi(g)v \ra d g \\
&=& \int_{G_0 \bs G_0K} \la \lambda(\xi), v \ra d g \\
&=& \vol(G_0 \bs G_0K)\la \lambda(\xi), v \ra.
\end{eqnarray*} 
Hence $\wt{\lambda}(\vol(G_0 \bs G_0K)^{-1}f_K) = \lambda(\xi)$.
For the last assertion, let $L \subset G$ be an open compact subgroup.
Fix an isomorphism $\mu: L /L \cap G_0 \simeq A$ for a compact group $A$.
Since $\pi$ is $G_0$-admissible, $V^{L \cap G_0}$ is finite dimensional.
Therefore, there is an open subgroup $B \subset A$ such that $V^{L \cap G_0} \subset V^{L_B}$ for $L_B := \{k \in L \mid \mu(k) \in B \}$.
Then, $L_B \cap G_0 = \{k \in L \mid \mu(k) = 1 \} = L \cap G_0$, and hence $V^{L_B \cap G_0} = V^{L_B}$.
Then $N := \{L_B\}$ is the desired system of neighbourhoods.
\end{proof}
Now, since the generic irreducible $\th(\tau^\vee)$ is the quotient of $\Th(\tau^\vee)$, $\pi$ has a generic irreducible quotient isomorphic to $\th(\tau^\vee)$ by (\ref{eqn:Thsurjpi}).
By (\ref{eqn:mult1thtau}), $\Th(\tau^\vee)$ is of Whittaker type.
Therefore by Proposition \ref{prop:ondim}, for any $\xi$ and $\vp$, there exists a $W \in \W_\psi(\th(\tau^\vee))$ such that 
\begin{eqnarray*}
Z(s,W) = Z(s,\xi_\vp), \ \mbox{and} \ \ Z(1-s,\pi^\imath(\j_0)W^\imath) = Z(1-s, \pi^\imath(\j_0)(\xi_\vp)^\imath)
\end{eqnarray*}
up to constant multiples.
Of course, if $\xi_\vp$ is quasi-paramodular, then so is $W$ and $Z(s,W^c) = Z(s,\xi_\vp^c)$.
So, for the existence of the newform, we will construct a $\Kb_1(m;\e)$-invariant Whittaker function $\xi_\vp$ as in Theorem \ref{thm: coinLN}.
Now, fix a $\psi$ with conductor $\o$.
When $L$ is a subgroup of a similitude group, we will denote by $L^1$ the intersection of $L$ and the isometry group.
Let $m_i \ge n_{\tau_i} = m_{\tau_i}$ and $K = (K(m_1) \t K(m_2))^1$.
Let $m = m_1+m_2, l = m-\e$ where $\e = \c(\w)$.
Let $\xi_i$ be $K_1(m_i)$-invariant, and set $\xi = \xi_1 \bt \xi_2$.
Define $\vp = \vp_{m_1,m_2} \in \Ss(X \op X)$ by 
\begin{eqnarray*}
\vp_{m_1,m_2}(u \op v) = 
\Ch(v;M_{2}(\o)) \t
\begin{cases}
\Ch\l(u; \begin{bmatrix}
\p^{m_2} & \o \\
\p^{m} & \p^{m_1}
\end{bmatrix}\r) & \mbox{if $\e= 0$,} \\
\w(u_{12})\Ch\l(u; \begin{bmatrix}
\p^{m_2} & \o^\t \\
\p^l & \p^{m_1}
\end{bmatrix}\r)& \mbox{if $\e > 0$, } 
\end{cases}
\end{eqnarray*}
so that (\ref{eqn:condvp}).
By using the formulas of $w_\psi$, one can see that this Schwartz function is $\Kb_1(m;\e)^1$-invariant.
Therefore $\xi_{\vp}$ is a quasi-paramodular form of level $m$.
Let 
\begin{eqnarray*}
\vp^c &=& w_\psi(\j_m,(j_{m_1},j_{m_2})) \vp,  \ \ 
j_a = \begin{bmatrix}
 & -1\\
\v^a &
\end{bmatrix} \\
\xi_i^c(g) &=& \w^{-1}(\det(g)) \xi_i(gj_{m_i}).
\end{eqnarray*}
Then, the conjugate $(\xi_{\vp})^c$ of $\xi_{\vp}$ equals $(\xi^c)_{\vp^c}$.
In case of $\e = 0$, $\vp^c = \vp$.
In case of $\e > 0$,  
\begin{eqnarray*}
\vp^c(u \op v) = 
\frac{\Gb(\w,1)}{\w(\v^\e v_{12})}\Ch\l(u; \begin{bmatrix}
\p^{m_2} & \o \\
\p^l & \p^{m_1}
\end{bmatrix}\r) \Ch\l(v; \begin{bmatrix}
\o & \p^{*-\e} \\
\o & \o
\end{bmatrix}\r).
\end{eqnarray*}
For the computation of the zeta integrals, the following lemma and the Bruhat decomoposition $SO_X = \sqcup_{w \in W_X} B_X^1 w N_X$ is useful, where $W_X: = \{(1_2,1_2), (1_2, j_0), (j_0,1_2),(j_0,j_0)\}$ is the Weyl group of $T_X$.
\begin{lem}\label{lem: schnv}
Let $\xi \in \tau$.
Let $\vp = \vp_1 \ot \vp_2 \in \Ss(X \op X)$.
Let $K$ be a open compact subgroup of $SO_X$ such that 
\begin{eqnarray}
\vp(k^{-1} \cdot z) \tau(k) \xi = \vp(z) \xi \label{eqn:condvp}
\end{eqnarray}
for $z \in Z^+, k \in K$.
Assume that  
\begin{eqnarray}
&& \vp_1(h^{-1} \cdot e_0) \neq 0 \ \ \mbox{for} \ \ h \in SO_X \setminus B_X^1K \nonumber \\ 
& &\Longrightarrow \vp_2(h^{-1} \cdot (1_2 - x e_0)) = \vp_2(h^{-1} \cdot 1_2) \ \ \mbox{for any $x \in \p^{-1}$.} \label{eqn: condth2}
\end{eqnarray}
Let $S$ be representatives for the double coset space $N_\Delta \bs B_X^1/B_X^1 \cap K$.
Let $S'$ be the collection of $b \in S$ such that $\vp(b^{-1} \cdot z_0) = 0$, or $\vp_2(b^{-1} \cdot (1_2 - x e_0)) = \vp_2(b^{-1} \cdot 1_2)$ for any $x \in \p^{-1}$.
Then, 
\begin{eqnarray}
\int_{N_{\Delta} \bs SO_X}\vp(h^{-1}\cdot z_0) \xi(h) d \dot{h} = \frac{\vol_{SO_X}(K)}{\vol_{N_\Delta}(\Int(b)K \cap N_\Delta)} \sum_{b \in S \setminus S'}  \xi(b) \vp(b^{-1} \cdot z_0). \label{eqn:thintsim}
\end{eqnarray}
\end{lem}
\begin{proof}
Let $h_1 \in SO_X$.
For $x \in F$, 
\begin{eqnarray*}
\int_{}\vp(h^{-1} \cdot z_0) \xi(h) d \dot{h} &=& \int_{}\vp\l(((n(x),1) h)^{-1} \cdot z_0\r) \xi((n(x),1) h) d \dot{h} \\
&=& \int_{}\vp\l(h^{-1} \cdot \l(e_0 \op (1_2 -x e_0)\r)\r) \psi(x)\xi(h) d\dot{h} \\
&=& \int_{}\vp_1(h^{-1} \cdot e_0) \vp_2\l(h^{-1} \cdot (1_2 - x e_0)\r) \psi(x)\xi(h) d\dot{h},
\end{eqnarray*}
where integrations are over $\dot{h}$ in $N_{\Delta} \bs N_X h_1 K$. 
This integral vanishes, if $\vp_2(h_1^{-1} \cdot(1_2 - x e_0)) = \vp_2(h_1^{-1} \cdot 1_2)$ for any $x \in \p^{-1}$. 
The definition of $S'$ and the condition (\ref{eqn: condth2}) mean the noncontributions to the integral in (\ref{eqn:thintsim}) of the orbits $b'K$ for $b' \in S'$ and the orbits $B_X^1 w N_XK$ for $w \in W_X \setminus \{1_2,1_2 \}$, respectively.
Now the assertion is obvious.
\end{proof}
From the definition of $\xi_\vp$, it follows that
\begin{eqnarray*}
\xi_\vp(\ab_i) = q^{-2i}(\xi^i)_{\vp_i}(1), 
\end{eqnarray*}
where $\xi^i = \tau(1,a_i)\xi$, and 
\begin{eqnarray*} 
\vp_i(u \op v) &=& \vp\l((u \op v)a_i\r) \\
&=& 
\Ch\l(v; \begin{bmatrix}
\p^{-i} & \o \\
\p^{-i} & \o
\end{bmatrix}\r) \t
\begin{cases}
\Ch\l(u; \begin{bmatrix}
\p^{m_2-i} & \o \\
\p^{m-i} & \p^{m_1}
\end{bmatrix}\r) & \mbox{if $\e= 0$,} \\
\w(u_{12})\Ch\l(u; \begin{bmatrix}
\p^{m_2-i} & \o^\t \\
\p^{l-i} & \p^{m_1}
\end{bmatrix}\r)& \mbox{if $\e > 0$.} 
\end{cases}
\end{eqnarray*}
From the Bruhat decomposition, we obtain
\begin{eqnarray}
SO_X = \sqcup_{w \in W_X} N_X T_X^1 w N_w K_i, \label{eqn: dcmpH1}
\end{eqnarray}
where $K_i = \Int((1,a_i))K$, and 
\begin{eqnarray*}
N_w = 
\begin{cases}
\{(1,1)\} & \mbox{if $w = (1_2,1_2)$,} \\
\{(n(b_1),1) \mid  b_1 \in \p^{1-m_1} \}& \mbox{if $w = (j_0,1_2)$,} \\
\{(1,n(b_2))  \mid  b_2 \in \p^{1-m_2+i}\} & \mbox{if $w = (1_2,j_0)$,} \\
\{(n(b_1),n(b_2)) \mid  b_1\in \p^{1-m_1}, b_2 \in \p^{1-m_2+i}\}  & \mbox{if $w = (j_0,j_0)$.}
\end{cases}
\end{eqnarray*}
If $h$ lies in the orbit $N_X T_X^1w N_w$ with $w \neq (1_2,1_2)$, then $h^{-1} \cdot e_0 $ is one of the following forms
\begin{eqnarray*}
\begin{bmatrix}
 & - b_1 \v ^s \\
 & \v^s
\end{bmatrix}, \ \begin{bmatrix}
-\v^s & -b_2 \v ^s \\
 &
\end{bmatrix}, \ \begin{bmatrix}
-b_2 \v^s & - b_1b_2 \v ^s \\
\v^s & b_1\v^s
\end{bmatrix} \ \ (s \in \Z).
\end{eqnarray*}
Now, it is easy to see that (\ref{eqn: condth2}) holds for $\vp_i$.
Let $S = \{h = (\v^r n(x) a_s,a_t) \mid s+2r = -t, x \in F / \p^s \}$.
Then, $S$ is the representatives for $N_\Delta \bs B_X^1 / B_X^1 \cap K_i$, and 
\begin{eqnarray*}
h^{-1} \cdot z_0 = \begin{bmatrix}
 &  \v ^{r} \\
 & \end{bmatrix} \op 
 \begin{bmatrix}
\v^{-s-r} & -\v ^{r}x \\
 &\v^{s+r}
 \end{bmatrix}, \ h \in S.
\end{eqnarray*}
Therefore, $S \setminus S'$ in Lemma \ref{lem: schnv} consists of $(\v^r n(x) a_s,a_t)$ with $r = 0, 0 \le s= -t \le i, x \in \o/\p^s$, and 
\begin{eqnarray}
\xi_{\vp}(\ab_i) &=& q^{-i}\vol(K) \sum_{l=0}^i \xi_1(a_l) \xi_2(a_{i-l}). \label{eqnWij}
\end{eqnarray}
By a similar computation and the identity $\Gb(\w,1)\Gb(\w^{-1},1) = q^{-\e}$ when $\e > 0$, 
\begin{eqnarray}
\xi_\vp^c(\ab_i) =  q^{-i-\e}\vol(K)\sum_{l=0}^i \xi_1^c(a_l) \xi_2^c(a_{i-l}). \label{eqnWcij}
\end{eqnarray}
\begin{thm}\label{thm:thetaGSO22}
Let $\tau_1,\tau_2 \in \Ir^{gn}(G_2)$, and let $n_i = n_{\tau_i} (= m_{\tau_i})$.
Let $\pi = \th(\tau_1^\vee \bt \tau_2^\vee) \in \Ir^{gn}(\G)$.
Let $r \ge 0$.
Then $\pi$ has a quasi-paramodular form $W$ of level $n_1+ n_2 + r$, such that 
\begin{eqnarray*}
\frac{(\ep_{\pi}')^{-1} Z(1-s,W^c)}{L(1-s,\tau_1^\vee)L(1-s,\tau_2^\vee)} = \frac{q^{-r (s- \frac{1}{2})} Z(s,W)}{L(s,\tau_1) L(s,\tau_2)} = 1. \ 
\end{eqnarray*}
In particular, $m_\pi = n_\pi' = n_1 + n_2$, and $\ep'_\pi = \ep_{\tau_1}\ep_{\tau_2}$, and $V(m_\pi)$ is spanned by this $W$.
\end{thm}
\begin{proof}
Set $\vp = \vp_{n_1, r+n_2}$ and $\xi_1 = W_{1}, \xi_2 = \tau_2(a_{r})W_{2}$, where $W_i \in \tau_i$ are the newforms. 
Then, $Z_0(s,\xi_1) = L(s,\tau_1), \ Z_0(s,\xi_2) = q^{-r (s- \frac{1}{2})} L(s,\tau_2)$, and $Z_0(s,\xi_i^c) = \ep_{\tau_i}L(s,\tau_i^\vee)$ for $i = 1,2$.
By (\ref{eqn:Lpitau}), (\ref{eqnWij}), (\ref{eqnWcij}), and these identities, $W = \xi_\vp$ has the desired property.
The last assertion follows from Theorem \ref{thm: coinLN}.
\end{proof}
By the proof of Corollary \ref{cor:L1},
\begin{Cor}\label{cor:thetaGSO22}
With the assumption as in the theorem, assume that $\c(\w_\pi) > 0$ additionally.
Then, there exists a $W^- \in V^-(m_\pi + r)$ such that 
\begin{eqnarray*}
\frac{\ep_{\pi}^{-1} Z(1-s,W^{-c})}{L(1-s,\tau_1^\vee)L(1-s,\tau_2^\vee)} =  \frac{q^{-r (s- \frac{1}{2})}\Xi(s,W^-)}{L(s,\tau_1)L(s,\tau_2)} = \Gb(\w_\pi,1).
\end{eqnarray*}
\end{Cor}
By the work of \cite{G-T2}, the $L$-parameter $\phi_\pi:WD_F \to GSp(4, \C)$ of $\pi = \th(\tau_1^\vee \bt \tau_2^\vee)$ is $\phi_{\tau_1} \op \phi_{\tau_2}$.
Hence, 
\begin{Cor}\label{cor:LPGSO22}
With the assumption as in the theorem, $L(s,\pi) = L(s,\phi_\pi)$ and $\ep(s,\pi,\psi) = \ep(s,\phi_\pi,\psi)$.
\end{Cor}
\section{Construction of newform for GL(4)}\label{sec:CN4}
W. T. Gan and S. Takeda \cite{G-T2} showed the Langlands correspondence for $\G$, comparing the representations of $\G$ and those of $G_4$ by the local $\th$-correspondence for $\G$ and $GSO(3,3) \simeq G_4 \t F^\t/\{(z,z^{-2}) \mid z \in F^\t\}$.
In particular, for $\pi \in \Ir(\G)$, the Langlands parameter $\phi_\pi$ coincides with that of the local $\th$-lift of $\pi$ to $G_4$.
In this section, to show the the coincidences $n_\pi' = n_\pi$ and $\ep_\pi' = \ep_\pi$, we will observe the local $\th$-lift.
Let $U =F^4$.
In this section, let $X = \wedge^2 U$, which is $6$-dimensional.
The bilinear form on $X$ defined by $x \wedge x'$ is symmetric, non-degenerate and splits, where we idenitify $X \wedge X$ with $F$ naturally.
Letting $G_4 \t F^\t$ and $GSO_X := \ker(\mu_X^{-3} \det)$ act on $U$ and $X$ from the left, respectively, we have an isomorphism 
\begin{eqnarray*}
I_a:  G_4 \t F^\t/\{(z,z^{-2}) \mid z \in F^\t\} \simeq GSO_{X}.
\end{eqnarray*}
Let $\{u_1,u_2,u_3,u_4\}$ be the standard basis of $U$, and set 
\begin{eqnarray*}
X^+ &=& \mathrm{Span} \{e_1,e_2,e_3\}; \ \  e_{1} = u_2 \wedge u_3, \ e_{2} = u_3 \wedge u_1, \ e_{3} = u_1 \wedge u_2, \\
X^- &=&  \mathrm{Span} \{e_{-1},e_{-2},e_{-3}\}; \ \   e_{-1} = u_1 \wedge u_4, \ e_{-2} = u_2 \wedge u_4, \ e_{-3} = u_3 \wedge u_4.
\end{eqnarray*}
We will write the elements of $GSO_X$ as matrices according to the basis $\{e_{3}, e_{2}, e_{1}, e_{-1}, e_{-2}, e_{-3} \}$.
Then the isomorphism $I_a$ respects the transpose and sends 
\begin{eqnarray*}
P_4 \ni && 
\begin{bmatrix}
g & s\\
 & 1
\end{bmatrix} \longmapsto 
\begin{bmatrix}
1_3 & b(s)\\
 & 1_3
\end{bmatrix}
\begin{bmatrix}
w_3{}^tg^{-1}w_3 & \\
 & g
\end{bmatrix} \in GSO_{X},
\end{eqnarray*}
where $g$ is an element of $G_3$, and  
\begin{eqnarray*}
b(s) = b(\begin{bmatrix}
s_1  \\
s_2  \\
s_3
\end{bmatrix}) = \begin{bmatrix}
s_{2} & -s_{1}&  \\
-s_{3} & & s_{1}\\
 &s_{3} & -s_{2}
\end{bmatrix}.
\end{eqnarray*}
Let $Z^\pm = X^\pm \ot Y \simeq Y \op Y \op Y$, and identify $Z^+$ with $M_{3 \t 4}(F)$ via the mapping:   
\begin{eqnarray*}
z = \sum e_i \ot y_i \longleftrightarrow \begin{bmatrix}
y_1 \\
y_2 \\
y_3
\end{bmatrix} \in M_{3 \t 4}(F).
\end{eqnarray*}
For $z = \sum_i e_i \ot  y_i,  z' = \sum_i e_i \ot  y_i' \in Z^+$, we write $\la z,z' \ra  = (\la y_i, y_j' j(-w_2) \ra ) \in M_{3}(F)$.
For $\Phi \in \Ss(Z^+)$, let $\Phi^\sharp$ denote the Fourier transform defined by $\Phi^\sharp(z) = \int_{Z^+}\psi^{-1}(Tr(\la z', z \ra)) \Phi(z') d z$ where $d z$ is chosen so that $(\Phi^{\sharp})^\sharp(z) =\Phi(-z)$.
Let $\G$ act on $Y$ from the right.
In this section, we use the Weil representation $w_{\psi^{-1}}$ of $Sp_4 \t O_X$ realized on the space $\Ss(M_{3 \t 4}(F))$ with the following transformation formulas. 
\begin{eqnarray*}
w_{\psi^{-1}}(g,1) \Phi(z) &=& \Phi(z g), \ \ g \in Sp(4), \\
w_{\psi^{-1}}(1, \begin{bmatrix}
w_3{}^ta^{-1}w_3  & \\
 & a
\end{bmatrix}) \Phi(z) &=& |\det(a)|^2 \Phi(a^{-1} z), \\
w_{\psi^{-1}}(1,\begin{bmatrix}
1& b(s)\\
 & 1
\end{bmatrix}) \Phi(z) &=& \psi^{-1}(\frac{1}{2} Tr\l(\la z,z\ra w_3 b(s)\r) ) \Phi(z)\\
&=& \psi(s_1\la y_2,y_3\ra + s_2\la y_3,y_1\ra + s_3\la y_1,y_2\ra)\Phi(z), \\
w_{\psi^{-1}}(1_4,j(-w_3)) \Phi(z) &=& \Phi^\sharp(z).
\end{eqnarray*}
Let $R = \G \t GSO_X$, and $R_0 = \ker(\mu^{-1} \mu_X) \subset R$. 
We extend $w_{\psi^{-1}}$ to $R_0$ via  
\begin{eqnarray*}
w_{\psi^{-1}}(g,h) \Phi(z) = |\mu(g)|^{-3}w_{\psi^{-1}}(1,h_1)\Phi(zg)
\end{eqnarray*}
so that the central elements $(u,u)$ act on trivially, where 
\begin{eqnarray*}
h_1 = h \begin{bmatrix}
\mu(g)^{-1} 1_3 & \\
 & 1_3
\end{bmatrix} \in SO_X.
\end{eqnarray*}
Let $\{\ep_1,\ep_2, \ep_{-2},\ep_{-1} \}$ denote the standard basis of $Y$.
Set
\begin{eqnarray*}
z_0 &=& \ep_{2} \ot e_1 + \ep_{-2} \ot e_2 + \ep_{-1} \ot e_3 = [0,1_3] \in M_{3 \t 4}(F).
\end{eqnarray*}
Let $d g, dz$ be Haar measures on $Sp_4$, and $\Zb^J$, respectively.
We choose $dz$ such that $\vol(\Zb^J(\o)) = 1$.
Let $d\dot{g} = dg /dz$ denote the Haar measure on $\Zb^J \bs Sp_4$.
Let $\pi \in \Ir^{gn}(\G)$.
For $W \in \W_\psi(\pi)$, and $\Phi \in \Ss(M_{3 \t 4}(F))$, we define a function $W_\Phi$ on $G_4$ by 
\begin{eqnarray*}
W_\Phi(h) = \int_{\Zb^J \bs Sp_4}w_{\psi^{-1}}(g_1g_h,h) \Phi(z_0) W(g_1g_h)d\dot{g},
\end{eqnarray*}
where $g_h$ is an element in $\G$ such that $\mu(g_h) =\det(h)$.
By the above formulas of $w_{\psi^{-1}}$, for $n \in N \subset G_4$,
\begin{eqnarray*}
w_{\psi^{-1}}(1,n) \Phi(z_0) = \psi(n_{34})\Phi\l(z_0 \nb_2(-n_{23})\nb_3(-n_{13})\nb(-n_{12})\r), 
\end{eqnarray*}
from which one can find that $W_\Phi$ is a Whittaker function on $G_4$ with respect to $\psi$.
Let $\Pi$ be the $G_4$-module generated by these $W_\Phi$.
Since the central elements $(u,u) \in R_0$ act on $\Ss(M_{3 \t 4}(F))$ trivially, $\w_\Pi = \w_\pi^2$.
Define the big theta $\Th(\pi)$ and the small theta $\th(\pi)$, similar to the previous section.
By the work of \cite{G-T}, $\wt{\pi}: = \th(\pi^\vee)$ is generic.
By the similar argument, and the proof of Lemma 2.10 of \cite{J-PS-S2}, instead of Proposition \ref{prop:ondim}, for any $W \in \pi$ and $\Phi \in \Ss(M_{3 \t 4}(F))$, there exists a $\wt{W} \in \W_\psi(\wt{\pi})$, such that 
\begin{eqnarray*}
Z_2(s,\wt{W}) = Z_2(s,W_\Phi), \ \ \mbox{and} \ \ Z_0(1-s, \wt{W}^\imath) = Z_0(1-s, (W_\Phi)^\imath).
\end{eqnarray*}
Now we will construct a $K_1(m)$-invariant $W_\Phi$ using $W \in V(m)$ for $m \ge 2\e$ where $\e = \c(\w_\pi)$.
Set  
\begin{eqnarray*}
\L_m = \begin{bmatrix}
\p^m & \o &\o& \o \\
\p^m & \o &\o& \o \\
\p^m & \o &\o& \o 
\end{bmatrix},
\end{eqnarray*}
which is a $\Kb(m)$-invariant lattice.
According to $\e$, we define $\Phi_m \in \Ss(M_{3 \t 4}(F))$ by 
\begin{eqnarray*}
\Phi_m(z) = 
\begin{cases}
\Ch(\L_m) & \mbox{if $\e = 0$} \\
\w_\pi(\det(\check{z}))\Ch(a_{m -\e}G_3(\o);\check{z}) \Ch(M_{3\t1}(\o);z_4) & \mbox{ if $\e > 0$}
\end{cases}
\end{eqnarray*}
where $z_4$ and $\check{z}$ indicate the the right $M_{3\t 1}$-part and the left $M_{3\t 3}$-part of $z$, respectively.
Note that the support of $\Phi_m$ is contained in $\L_{m- \e}$.
Let $W \in V(m)$.
By the formulas of $w_{\psi^{-1}}$ and definition, for $k \in K_1(m) \cap SL_4$ and $k' \in \Kb_1(m;\e)^1$, 
\begin{eqnarray}
w_{\psi^{-1}}(k',k)\Phi(z) \pi(k')W = \Phi(z) W, \label{eqn:PhiW}
\end{eqnarray}
and $W_\Phi$ is $K_1(m)$-invariant.
We will compute $W_{\Phi}(a_r)$.
By definition, $W_{\Phi}(a_r) = q^{-r} (W^{(r)})_{\Phi^{(r)}}(1)$, where 
\begin{eqnarray*}
W^{(r)} = \pi(\ab_r)W, \ \mbox{and} \ \Phi^{(r)}(z) = \Phi(a_{-r} z \ab_r).
\end{eqnarray*}
{\bf (Computation for $W_{\Phi}(a_r)$ in the case of $\e = 0$).}
Let $\Kb = \Int(\ab_r)\Kb(m)$.
Let $\Wb' = \{ 1, w_2^\natural, \j_0, j(w_2) \}$.
For $w \in \Wb'$, let $\N_w$ be the following finite subset of $\N$:
\begin{eqnarray*}
\N_w = \begin{cases}
1 & \mbox{if $w= 1$,} \\
\{ \nb_2(\v^i)\}, \ \ 1-m \le i \le 0 & \mbox{if $w= w_2^\natural$,} \\
\{\nb_3(\v^h) \},\ \ 1+r-m \le h \le r-1  & \mbox{if $w=\j_0$,} \\
\{\nb_2(\v^i)\nb_3(\v^h)\} \ \ r-m-h < i < h-r  \le -1, & \mbox{if $w= j(w_2)$}.
\end{cases}
\end{eqnarray*}
Then, $w \nb$ with $w \in \Wb', \nb \in \N_w$ are representatives for $\Bb^1 \bs Sp_4/\Kb^1$ (c.f. Proposition 5.1.2 of \cite{R-S}).
Therefore, we may write for representatives for $\Zb^J \bs Sp_4/ \Kb^1$ of the form of $g = \nb_2(-x)\nb_3(y)\nb'(z) \tb(\a,\b)w \nb$ with $\nb \in N_w$.
By (\ref{eqn:PhiW}), 
\begin{eqnarray*}
\int_{\Zb^J \bs Sp_4}\Phi_m^{(r)}(z_0 g)W^{(r)}(g) dg = q^{-m}\vol(\Kb^1)\sum_{w \in \Wb'} \int_{\Zb^J \bs \Bb^1 w \N_w}\Phi_m^{(r)}(z_0 g)W^{(r)}(g) dg
\end{eqnarray*}
where 
\begin{eqnarray*}
\Phi_m^{(r)} = \Ch(\begin{bmatrix}
\p^m & \o &\p^r &\p^r\\
\p^{m-r} & \p^{-r}& \o & \o \\
\p^{m-r} & \p^{-r}& \o & \o
\end{bmatrix}).
\end{eqnarray*}
Now we observe the integral $\int_{\Zb^J \bs \Bb^1 w \N_w} ... dg$.
In case of $w= w_2^\natural$, 
\begin{eqnarray*}
z_0 g = \begin{bmatrix}
\b & \b \v^i &\a^{-1} y  & \b^{-1}z + \a^{-1} \v^i y \\
0 &0 & \a^{-1}x  & \b^{-1} + \a^{-1}\v^i x \\
0 &0 & \a^{-1} & \a^{-1} \v^i
\end{bmatrix}.
\end{eqnarray*}
If the $(2,3)$-coefficient $\a^{-1}x$ lies in $\o$, and the $(1,1)$-coefficient $\b$ lies in $\p^m$, then $\a^{-1}\v^i x$ lies in $\p^{1-m}$, and the $(2,4)$-coefficient $(\b^{-1} + \a^{-1}\v^i x)$ has order $\le -m$ and is not in $\o$.
Hence, $\Phi(z_0 g)$ is $0$, and so is the integral. 
In case of $w= \j_0$, 
\begin{eqnarray*}
z_0 g = \begin{bmatrix}
\b^{-1} z & \a^{-1}y & \b^{-1}\v^h z & \b + \a^{-1}\v^h y \\
\b^{-1} &\a^{-1} x  & \b^{-1} \v^h & \a^{-1}\v^h x \\
0 &\a^{-1} & 0 & \a^{-1}\v^h \\
\end{bmatrix}.
\end{eqnarray*}
If the $(1,2)$-coefficient $\a^{-1}y$ lies in $\o$, and the $(2,1)$-coefficient $\b^{-1}$ lies in $\p^{m-r}$, then $\a^{-1}\v^h y$ lies in $\p^{1-m+r}$, and the $(1,4)$-coefficient $(\b + \a^{-1}\v^h y)$ has order $\le r-m$, and is not in $\p^r$.
Hence, $\Phi(z_0 g)$ is $0$, and so is the integral. 
In case of $w= j(w_2)$, 
\begin{eqnarray*}
z_0 g = \begin{bmatrix}
\a^{-1}y & \b^{-1}z + \a^{-1} \v^i y &  \b+ \a^{-1} \v^h y & \b \v^i + \b^{-1}\v^h z  \\
\a^{-1} x & \b^{-1}+ \a^{-1}\v^ix& \a^{-1} \v^h x & \b^{-1}\v^h \\
\a^{-1} & \a^{-1} \v^i &\a^{-1} \v^h  & 0 \\
\end{bmatrix}.
\end{eqnarray*}
If the $(1,1)$-coefficient $\a^{-1}y$  lies in $\p^m$, and the $(2,4)$-coefficient $\b^{-1}\v^h$ lies in $\o$, then $\a^{-1}\v^h y$ lies in $\p^{m+h}$, and $\b^{-1}$ lies in $\p^{-h}$, and therefore, the $(1,3)$-coefficient $\b+ \a^{-1} \v^h y$ has order $\le h$ and is not in $\p^r$.
Hence, $\Phi(z_0 g)$ is $0$, and so is the integral.
In case of $w= 1$, 
\begin{eqnarray}
z_0 g = \begin{bmatrix}
0 & \b \ & \b^{-1}z & \a^{-1} y \\
0 & 0&  \b^{-1} & \a^{-1}x\\
0 & 0 & 0&\a^{-1} \\
\end{bmatrix}. \label{eqn:z0g}
\end{eqnarray}
From the $(1,2),(2,3),(3,4)$-coefficients, it follows that $\Phi(z_0 g) = 0$ unless $\b \in \o^\t, \a^{-1} \in \o$.
By Lemma \ref{lem:vanishlem}, $W^{(r)}(\tb(\a,\b)) = 0$ if $\a \not\in \o$.
Therefore $\Phi_{m}^{(r)}(z_0 g)W^{(r)}(g) = 0$ unless $g \in \Zb^J \Kb^1$, and 
\begin{eqnarray}
W_{\Phi}(a_r) = q^{-m}\vol(\Kb)W(\ab_r).  \label{eqn:WPhiaj}
\end{eqnarray}
{\bf (Computation for $W_{\Phi}(a_r)$ in the case of $\e > 0$).}
Let $l = m- \e$.
Noting that $\supp(\Phi_m) \subset \L_{l}$, we find that $\Phi(z_0g) = 0$ unless $g \in \Bb^1 \Int(\ab_r) \Kb(l)^1$ by the above computation.
Therefore, we may assume that $g \in \Zb^J \bs \Bb^1 /\Int(\ab_r)\Kb(m;\e)^1$ is of the form of $\nb_2(x)\nb_3(y) \nb'(z) \tb(\a,\b) \jb''_{l-r}$ or $\nb_2(x)\nb_3(y) \nb'(z)\tb(\a,\b) \bz(\mu)$ with $\mu \in \p^{l-r+1}$.
If $g$ is of the latter form, then
\begin{eqnarray*}
z_0 g = \begin{bmatrix}
\mu \a^{-1} y & \b \ & \b^{-1}z & \a^{-1} y \\
\mu \a^{-1}x & 0&  \b^{-1} & \a^{-1}x\\
\mu \a^{-1} & 0 & 0&\a^{-1} \\
\end{bmatrix}. 
\end{eqnarray*}
Assume  $\Phi(z_0g) \neq 0$.
From the $(1,2), (2,3)$-coefficients, $\b \in \o^\t$.
From the $(3,4)$-coefficient, $\a^{-1} \in \o$, and therefore $\mu \a^{-1} \in \p^{l-r+1}$.
But, $\det(\check{z_0g}) = \mu \a^{-1} \in \p^{*(l-r)}$ by the definition of $\Phi_m^\w$.
Hence, we may assume $g$ is of the first form.
Then
\begin{eqnarray*}
z_0 g = \begin{bmatrix}
\v^{l-r}\a^{-1}y & \b \ & \b^{-1}z & 0 \\
\v^{l-r}\a^{-1}x & 0&  \b^{-1} & 0\\
\v^{l-r}\a^{-1} & 0 & 0& 0 \\
\end{bmatrix},
\end{eqnarray*}
and it is easy to see that $W_\Phi(a_r) = q^{-l}\vol(\Kb(m;\e)^1) W(\ab_r \jb''_l) = q^{-l}\vol(\Kb(m;\e)^1) W^-(\ab_r)$ ($W^-$ is defined at (\ref{eqn:defW^-})).
We have showed: 
\begin{prop}\label{prop:ZWphiW}
With notations as above, if $W \in V(m)$, then 
\begin{eqnarray*}
Z_0(s,W_\Phi) = q^{\e-m}\vol(\Kb(m;\e)^1)Z(s,W^-).
\end{eqnarray*}
\end{prop}
{\bf (Computation for $W_\Phi(w_4 a_{m}w_{1,3})$).}
By (\ref{eqn: Vmgl}),  
\begin{eqnarray*}
W_\Phi(w_4 a_m
\begin{bmatrix}
1 & \\
 & w_3
\end{bmatrix}) &=& \w_\pi^2(-1) W_\Phi(\begin{bmatrix}
1_3 & \\
 & \v^m
\end{bmatrix}
w_4
\begin{bmatrix}
1 & \\
 & w_3
\end{bmatrix}
\begin{bmatrix}
w_3 & \\
 & 1
\end{bmatrix}
\begin{bmatrix}
-1 & \\
 & 1_2 &\\
 & &-1
\end{bmatrix}) \\
&=&
W_\Phi(\begin{bmatrix}
1_3 & \\
 & \v^m
\end{bmatrix}
\begin{bmatrix}
w_2 & \\
 & w_2
\end{bmatrix}
\begin{bmatrix}
-1 & \\
 & 1_2 &\\
 & &-1
\end{bmatrix}).
\end{eqnarray*}
The isomorphism $I_a$ sends 
 \begin{eqnarray*}
\begin{bmatrix}
1_3 & \\
 & \v^m
\end{bmatrix}
\begin{bmatrix}
w_2 & \\
 & w_2
\end{bmatrix}
\begin{bmatrix}
-1 & \\
 & 1_2 &\\
 & &-1
\end{bmatrix} \mapsto  \begin{bmatrix}
 1_3 & \\
 & \v^m1_3
\end{bmatrix}
\begin{bmatrix}
1 & &\\
 &w_4 &\\
 && 1
\end{bmatrix} =: u_m.
\end{eqnarray*}
Set 
\begin{eqnarray*}
\Phi^c = w_{\psi^{-1}}(\j_m,u_m)\Phi.
\end{eqnarray*}
By definition of $W_\Phi$, $W_\Phi(w_4 a_{m}w_{1,3}) = (\pi(\j_m)W)_{\Phi^c}(1)$.
\begin{lem}\label{lem:Phiftr}
With the notations as above, we have the followings.
\begin{enumerate}[i)]
\item In the case of $\e = 0$, 
\begin{eqnarray*}
\Phi_m^c = \Ch(\begin{bmatrix}
 \p^m &\o &\o & \o \\
 \p^m &\o &\o & \o \\
\p^m & \p^m & \p^m & \o
\end{bmatrix}).
\end{eqnarray*}
\item In the case of $\e > 0$, the support $\supp(\Phi_m^c)$ is contained in the lattice 
\begin{eqnarray*} 
\begin{bmatrix}
 \p^l &\o &\o & \p^{-\e}\\
 \p^l &\o &\o & \p^{-\e} \\
\p^m & \p^m & \p^l & \o
\end{bmatrix}.
\end{eqnarray*}
If 
\begin{eqnarray*}
u \in \begin{bmatrix}
\p^m & \p^\e & \o & \o\\
\p^m & \p^\e & \o & \o \\
\p^{m+\e} & \p^{m} & \p^m & \p^\e 
\end{bmatrix}, \label{eqn:phi*in}
\end{eqnarray*}
then $\Phi_{m}^c(z+ u) = \Phi_{m}^c(z)$.
\end{enumerate}
\end{lem}
\begin{proof}
Suppose that $\Phi \in \Ss(Z^+)$ is of the form of $\phi_1 \ot \phi_2 \ot \phi_3$ with $\phi_r \in \Ss(Y \ot e_r)$.
Since 
\begin{eqnarray*}
(\j_m,u_m) = (1,u_0)(\j_m, \begin{bmatrix}
1 & & & \\
 & \v^m 1_2& & \\
& &1_2 & \\
& & & \v^m
\end{bmatrix}),
\end{eqnarray*}
by the formulas of $w_{\psi^{-1}}$, 
\begin{eqnarray*}\Phi^c = (\j_m \cdot \phi_{1})^\sharp \ot (\j_m \cdot \phi_{2})^\sharp \ot (\v^{-m} \j_m \cdot \phi_3),
\end{eqnarray*} 
where $g\cdot \phi$ is defined by $g \cdot \phi(z) = \phi(zg)$, and $\phi^\sharp$ is the Fourier transform defined by $\phi^\sharp(y) = \int_{Y}\psi^{-1}(\la y, y' j(-w_2) \ra) \phi(y') d y$ where $d y$ is chosen so that $(\phi^{\sharp})^\sharp(y) =\phi(-y)$.
Now, i) is a direct calculation.
For ii), we write $\Phi_{m}= \sum_i \phi_1^i \ot \phi_2^i \ot \phi_3^i$ so that  
\begin{eqnarray*}
\supp(\phi_r^i) &\subset& \check{L} \op L_4 := [\p^l,\o,\o,0] \op [0,0,0,\o], \\
\phi_r^i(y + u) &=& \phi_r^i (y) \ \ \mbox{for $u \in \v^\e \check{L} \op L_4$}.
\end{eqnarray*}
Then, 
\begin{eqnarray*}
\supp(\j_m \cdot \phi_r^i) &\subset& \check{L}' \op L_4' := [\o,0,\p^{-\e},\p^{-m}] \op [0,\o,0,0], \\
\j_m \cdot \phi_r^i(y + u) &=& \j_m \cdot \phi_r^i (y) \ \ \mbox{for $u \in \v^\e \check{L}' \op L_4'$},
\end{eqnarray*}
and 
\begin{eqnarray*}
\supp((\j_m \cdot \phi_r^i)^\sharp) &\subset& \check{L}'' \op L_4'' := [\p^l,\o,0, \p^{-\e}] \op[0,0,\o,0], \\
(\j_m \cdot \phi_r^i)^\sharp(y + u) &=& (\j_m \cdot \phi_r^i)^\sharp(y) \ \ \mbox{for $u \in \v^\e\check{L}'' \op L_4''$}.
\end{eqnarray*}
From this, ii) follows.
\end{proof}
At the computation for $W_\Phi(a_{m})$ in the case of $\e = 0$, we do not use the third row of $z_0g$ for the condition $\Phi_m(z_0g) \neq 0$.
Therefore, by Lemma \ref{lem:Phiftr} i), the same argument can be applied for $\Phi_m^c$, and we have $W_\Phi(w_4 a_{m}w_{1,3}) \neq 0$.

Let $\e > 0$.
By Lemma \ref{lem:Phiftr} ii), $\supp(\Phi_{m}^c)$ is contained in the lattice 
\begin{eqnarray*}
\L_l \cdot \ab_{-\e} = \begin{bmatrix}
 \p^l &\o &\p^{-\e} & \p^{-\e}\\
\p^l &\o &\p^{-\e} & \p^{-\e} \\
\p^l &\o &\p^{-\e} & \p^{-\e}
\end{bmatrix},
\end{eqnarray*}
which is $\Int(\ab_{-\e})\Kb(l)$-invariant.
By the above argument for the case of $\e = 0$, 
\begin{eqnarray*}
\Phi_{m}^c(z_0g) = 0, \ \ \mbox{unless $g \in \Bb^1 \Int(\ab_{-\e}) \Kb(l)^1$}.
\end{eqnarray*}
Therefore, since $\Int(\ab_{-\e})\Kb(l) \supset \Kb^c(m;\e)$, we may assume that $g \in \Zb^J \bs Sp_4/\Kb^c(m;\e)^1$ is written of the form of $\nb_2(-x)\nb_3(y) \nb'(u)\tb(\a,\b) \jb'(1) \nb'(\v^i)$ or $\nb_2(-x)\nb_3(y) \nb'(u) \tb(\a,\b)$.
Assume that $g$ is of the former form.
Then, 
\begin{eqnarray*}
z_0 g = \begin{bmatrix}
0 & \b^{-1}z & \a^{-1} y & -\b+\v^i\b^{-1}z \\
0 & 0&\a^{-1} &0 \\
0& \b^{-1}&\a^{-1}x&\v^i\b^{-1}  \\
\end{bmatrix}.
\end{eqnarray*}
By Lemma \ref{lem:Phiftr} ii), for the condition $\Phi(z_0g) \neq 0$, we need the $(3,2)$-coefficient $\b^{-1} \in \p^m$.
Then, $\Phi(z_0g \nb'(u)) = \Phi(z_0g)$ for $u \in \p^{\e-m}$.
Therefore,  the integration $\int_{\Zb^J \bs \{ g\}} \Phi(z_0g)W^c(g) dg$ over the set of the former forms vanishes.
Hence, we may assume that $g$ is of the latter form.
By (\ref{eqn:z0g}), and Lemma \ref{lem:Phiftr}, 
\begin{eqnarray}
W_{\Phi}(w_4 a_m w_{1,3}) = c_{m} \Phi_m^c(z_0) W^c(1), \label{eqn:WPhi*}
\end{eqnarray}
where $c_{m}$ is a constant depending only on $\Phi_m^c$.
\begin{lem} 
In case of $\e > 0$, 
\begin{eqnarray}
c_{m} \Phi_m^c(z_0) = \Gb(\w_\pi,1).
\end{eqnarray}
\end{lem}
\begin{proof}
There exist principal series $\tau_1, \tau_2 \in \Ir^{gn}(G_2)$ such that $\w_{\tau_1} = \w_{\tau_2} = \w$ and $n_{\tau_1} = n_{\tau_2} = \e$.
Let $\pi =  \th(\tau_1^\vee \bt \tau_2^\vee)$, and $\Pi = \th(\pi^\vee)$.
By \cite{G-T2}, the $L$-parameter of $\Pi$ is $\phi_{\tau_1} \op \phi_{\tau_2}$.
By Corollary \ref{cor:LPGSO22}, $L(s,\Pi) = L(s,\pi)$.
Let $r \ge 0$.
Let $W \in \W_\psi(\pi)$ be quasi-paramodular of level $2\e + r$ be as in Theorem \ref{thm:thetaGSO22}.
Let $\Phi =\Phi_{2\e + r}$.
By Proposition \ref{prop:ZWphiW}, $Z_0(s,W_\Phi) = q^{-r(s-\frac{1}{2})}\Gb(\w,1)L(s,\Pi)$. 
By the functional equation (\ref{eqn:FEGL}), 
\begin{eqnarray*}
Z_0(1-s,\Pi^\imath(a_{-2\e-r})(W_\Phi)^\imath) = \ep'_\pi \Gb(\w_\pi,1)L(1-s,\Pi^\imath).
\end{eqnarray*}
Comparing the constant terms of both sides, we obtain the assertion from (\ref{eqn:WPhi*}).
\end{proof}
Now, we prove the coincidences $L(s,\phi_\pi) = L(s,\pi)$ and $\ep(s,\phi_\pi,\psi) = \ep(s,\pi,\psi)$.
We have showed this for generic constituents of Borel and Siegel parabolic inductions in the previous section.
Hence, we may assume $L(s,\pi) = 1$.
Let $\wt{\pi} = \th(\pi^\vee)$.
If we write $L(s,\wt{\pi})^{-1} = \prod_{i = 1}^d (1-\a_i q^{-s})$ by some $\a_i \in \C$, then $L(1-s,\wt{\pi}^\imath)^{-1} = \prod_{i = 1}^d (1-\a_i^{-1} q^{s-1})$ (recall that $\wt{\pi}^\imath$ is equivalent to $\wt{\pi}^\vee$).
By the above argument, $\ep'_\pi \Gb(\w_\pi,1)$ is the constant term of $Z_0(1-s, \wt{\pi}^\imath(a_{-m_\pi}) (W_\Phi)^\imath)$.
From Theorem \ref{thm:main}, \ref{thm:Ls1}, the functional equations, and the above argument, it follows that 
\begin{eqnarray*}
\frac{Z_0(1-s, \wt{\pi}^\imath(a_{-m_\pi}) (W_\Phi)^\imath)}{L(1-s,\wt{\pi}^\imath)} =  q^{(m_\pi-m_{\wt{\pi}})(s-\frac{1}{2})}\frac{\ep_{\wt{\pi}} \Gb(\w_\pi,1)}{L(s,\wt{\pi})}. \label{eqn:FElast}
\end{eqnarray*} 
Comparing zeros of these polynomials in $\C[X,X^{-1}]$ with $X = q^s$, we conclude that $\{\a_i \}_{i = 1}^d = \{q\a_i \}_{i = 1}^d$ as sets.
Hence, $L(s,\wt{\pi}) = L(1-s,\wt{\pi}^\imath) = 1$.
Therefore, $L(s,\phi_\pi) = L(s,\wt{\pi}) = L(s,\pi)$.
Now, the zeta integral $Z_0(1-s, \wt{\pi}^\imath(a_{-m_\pi}) (W_\Phi)^\imath)$ is constant, and equals $\ep'_\pi \Gb(\w_\pi,1)$.
Thus, $m_\pi = m_{\wt{\pi}}, \ep'_\pi = \ep_\pi = \ep_{\wt{\pi}}$,  and $\ep(s,\phi_\pi,\psi) = \ep(s,\pi,\psi)$.
This completes the proof.

\end{document}